\documentclass[11pt, leqno, a4paper]{amsart}

\usepackage{dsfont}
\usepackage{amsfonts,amsmath,amsthm,amssymb} 
\usepackage{hyperref} 
\usepackage{geometry}
\usepackage{todonotes} 
\usepackage{xcolor}

\newcommand{\R}{\mathbb{R}}

\newtheorem{theorem}{Theorem}[section]
\newtheorem{lemma}[theorem]{Lemma}
\newtheorem{corollary}[theorem]{Corollary}
\newtheorem{proposition}[theorem]{Proposition}

\theoremstyle{remark}
\newtheorem{remark}[theorem]{Remark}
 
\numberwithin{equation}{section}

\newcommand*{\bydef}{\overset{\rm def}{=}}
\newcommand*{\norm}[1]{\left\Vert #1\right\Vert}
\renewcommand*{\div}{\operatorname{div}}

\newcommand*{\id}{\operatorname{Id}}
\newcommand*{\supp}{\operatorname{supp}}
\newcommand*{\loc}{\mathrm{loc}}

\title[Inhomogeneous Navier--Stokes--Maxwell]{Global unique solutions to the planar inhomogeneous Navier--Stokes--Maxwell equations}

\author{Diogo Ars\'enio}
\address{New York University Abu Dhabi \\
Abu Dhabi \\
United Arab Emirates  }
\email{\href{mailto:diogo.arsenio@nyu.edu}{diogo.arsenio@nyu.edu}}

\author{Haroune Houamed}
\address{New York University Abu Dhabi \\
Abu Dhabi \\
United Arab Emirates  }
\email{\href{mailto:haroune.houamed@nyu.edu}{haroune.houamed@nyu.edu}}

\author{Belkacem Said--Houari}
\address{University of Sharjah  \\
Sharjah   \\
United Arab Emirates  }  
\email{\href{mailto:bhouari@sharjah.ac.ae}{bhouari@sharjah.ac.ae}}

\begin{document}

\begin{abstract}
The evolution of an electrically conducting incompressible fluid with nonconstant density can be described by a set of equations combining the continuity,   momentum   and   Maxwell's equations; altogether known as the  inhomogeneous Navier--Stokes--Maxwell system. 

In this paper, we focus on the global well-posedness of these equations in two dimensions. Specifically, we are able to prove the existence of global energy solutions, provided that the initial velocity field belongs to the Besov space $\dot{B}^{r}_{p,1}(\mathbb{R}^2)$, with $r=-1+\frac{2}{p}$, for some $p\in (1,2)$, while the initial electromagnetic field enjoys some $H^s(\mathbb{R}^2)$ Sobolev regularity, for some $s \geq 2-\frac{2}{p} \in (0,1)$, and whenever the initial fluid density is bounded pointwise and close to a nonnegative constant.  Moreover, if it is assumed that $s>\frac{1}{2}$, then the solution is shown to be unique in the class of all energy solutions.

It is to be emphasized that the solutions constructed here are global and uniformly bounded with respect to the speed of light $c\in (0,\infty)$. This important fact allows us to derive the inhomogeneous MHD system as the speed of light tends to infinity.
\end{abstract}

\maketitle

\tableofcontents

\section{Introduction}

The evolution of an  electrically conducting inhomogeneous fluid of density $\rho=\rho(t,x)$, velocity $u=u(t,x)$ and pressure  $p=p(t,x)$ is governed by the inhomogeneous Navier--Stokes--Maxwell equations 
\begin{subequations}\label{Main_System}
\begin{equation}\label{MHD_System_Nonhom}
	\begin{cases}
		\begin{aligned}
	\text{\tiny(Continuity equation)}&&&	\partial_t\rho+\operatorname{div} (\rho u)=0, \\ 
			\text{\tiny(Navier--Stokes's equations)}&&&\partial_t (\rho u)+\operatorname{div}(\rho u\otimes u) -\nu\Delta u+\nabla p=j\times B,
			\\
			\text{\tiny(Amp\`ere's equation)}&&&\frac{1}{c} \partial_t E - \nabla \times B =- j ,
			\\
			\text{\tiny(Faraday's equation)}&&&\frac{1}{c} \partial_t B + \nabla \times E  = 0 ,
			\\
			&&&\div u=0,\ \div B=0,
		\end{aligned}
	\end{cases}  
\end{equation} 
where $\nu,c>0$ represent the viscosity of the fluid and the speed of light, respectively. Above,  the time-space variables $(t,x)$ are taken over the whole space $   \R^+ \times \R^2$.  The set of equations \eqref{MHD_System_Nonhom} is supplemented with the initial data 
\begin{eqnarray} 
(\rho, u, E, B)|_{t=0}=(\rho_0, u_0, E_0, B_0)
\end{eqnarray}
and completed by Ohm's law that connects the current density $j$ to the velocity and the electromagnetic field $(E,B)$. One may find at least two types of Ohm's law in the literature. They write
\begin{eqnarray}
&\text{\tiny(Compressible Ohm's law)}&\qquad  j= \sigma \left( cE +  u \times B \right),  \label{Ohms-law1} \\
&\text{\tiny(Incompressible Ohm's law)} &\qquad j= \sigma \big( cE + \nabla \overline{p}+  u \times B \big) , \quad \div j = \div E=0,\label{Ohms-law2}
\end{eqnarray}
depending on whether or not the fluid is assumed to have a neutral charge density, where  $\sigma>0$ refers to  the electrical conductivity. It is to be emphasized that the analysis we perform in the present paper applies to both versions of Ohm's law. 
\end{subequations}  
  
  \subsection{Overview of previous results}
  
  Generally speaking, $u$, $E$ and $B $ are three dimensional vector fields, 
 whereas   $\rho $ and $p $ are scalar functions.    
The system of equations \eqref{MHD_System_Nonhom} consists of the incompressible inhomogeneous Navier--Stokes system 
\begin{equation}\label{Navier_Stokes}
\left\{
\begin{array}{ll}
\partial_t\rho+\operatorname{div} (\rho u)=0,\vspace{0.1cm}\\
\partial_t (\rho u)+\operatorname{div}(\rho u\otimes u) -\nu 
\Delta u+\nabla p=0,\vspace{0.1cm}\\
\div   u=0, 
\end{array}  
\right. \tag{NS}
\end{equation}
coupled with the equations for the electromagnetic field (that is, the Maxwell equations) 
\begin{equation}\label{Maxwell_System}
\left\{
\begin{array}{ll}
\partial_t E-\nabla\times B=-j,\vspace{0.1cm}\\
\partial_t B+\nabla\times E=0,\vspace{0.1cm}\\
\sigma E=j,\vspace{0.1cm}\\
\div  B=0,
\end{array}
\right. 
\end{equation}
through   Lorentz's force $j\times B$.  In a conducting fluid, we should replace  the third equation in \eqref{Maxwell_System} by either \eqref{Ohms-law1} or \eqref{Ohms-law2}, thereby yielding the system of interest in this paper. 

Note that the divergence-free condition on $B$ is not an additional constraint, and the system is therefore not overdetermined. It is, however, a property     which is propagated for all positive times as soon as it is assumed to hold initially. This can be seen by taking the divergence of Faraday's equation.

When the density is constant,  the system of equations \eqref{MHD_System_Nonhom} is reduced to the homogeneous Navier--Stokes--Maxwell system
  \begin{equation}\label{NSM-equa}
	\begin{cases}
		\begin{aligned} 
	  &\partial_t u+ u\cdot \nabla u -\nu\Delta u+\nabla p=j\times B, &\div u =0,&
			\\
			 &\frac{1}{c} \partial_t E - \nabla \times B =- j ,  &
			\\
			& \frac{1}{c} \partial_t B + \nabla \times E  = 0 , &\div B = 0,   & 
		\end{aligned}
	\end{cases}\tag{h-NSM}
\end{equation} 
 coupled with   \eqref{Ohms-law1} or  \eqref{Ohms-law2}. 

The  systems     \eqref{Main_System} and \eqref{NSM-equa} arise  in  the study of the flow of    electrically conducting
fluids in plasma dynamics, which models the motion of charged particles (ions and electrons) in an electromagnetic field.
The  interested    reader is referred to   \cite{D-book,Pai_1962,Lions_1996_Book_1,Biskamp_1993} for the derivation of the foregoing equations and more details on the physics behind them.

  Allow us now to quickly review some properties of the homogeneous model \eqref{NSM-equa}. At least formally, smooth solutions of that system enjoy the following energy identity
\begin{equation}\label{Ident_Energy_1}
\frac{1}{2}\frac{d}{dt} \Vert (u, E, B)(t)\Vert_{L^2}^2 + \nu \Vert \nabla u (t)\Vert_{L^2}^2+ \frac{1}{\sigma}\Vert j(t)\Vert_{L^2}^2=0,  
\end{equation}
for all $t\geq 0$,  which is the only known a priori bound for  \eqref{NSM-equa}, so far.

 The Navier--Stokes--Maxwell equations are particularly difficult to analyze when compared to the Navier--Stokes   and the MHD equations. This is basically due to  the hyperbolic nature of Maxwell's equations and the structure of the nonlinear Lorentz force.  It is because of this special structure that   the existence of global weak solution   at the level of the energy \eqref{Ident_Energy_1}   remains an outstanding   open problem. 

More precisely, the bounds that are derived from the  energy identity \eqref{Ident_Energy_1} are not enough to build solutions,  for  classical methods fail to recover any compactness properties for the Lorentz force, thereby leaving the construction of global weak solutions with only square integrable  initial data  completely open.   

In order to overcome this lack of compactness, the idea is   to propagate more regularity for the electromagnetic fields.    In this direction, the homogeneous version of the equations \eqref{NSM-equa} has been considered by several authors where it is shown that global solutions do 
exist as soon as   the initial electromagnetic fields belong to   spaces that are slightly better than $L^2$. Some relevant results in this line of research can be found in \cite{a19, ag20, aim15, as, Ger_Mess_Ibra_2014, Ibrahim_2011, Masmoudi_2010}.  

In particular, in \cite{Masmoudi_2010}, the author proved the first result on the existence and uniqueness  of global  solutions to \eqref{NSM-equa}  for initial data $(u_0, E_0, B_0)$ belonging to $ L^2(\R^2)\times H^s(\R^2)\times H^s(\R^2), $ for  some $s\in (0,1)$. 
There, the main ingredient used to build such solutions relies on a functional inequality to bound the $L^1_tL^\infty_x$ norm of the velocity field in terms of a logarithmic contribution of Sobolev norms of the electromagnetic fields. It is then shown that the $\dot{H} ^s$ norm of the electromagnetic fields grows at most exponentially in time.

More recently, the work from \cite{ag20} established the first global well posedness result for the two-dimensional system \eqref{NSM-equa} with Ohms law \eqref{Ohms-law1}   where the solution $(u,E,B)$  is shown to be  bounded in
\begin{equation*}
	L^2_{\loc}(\mathbb{R}^+; L^\infty(\mathbb{R}^2))\times L^\infty _{\loc}(\mathbb{R}^+; H^s(\mathbb{R}^2))   \times L^\infty _{\loc}(\mathbb{R}^+; H^s(\mathbb{R}^2)) 
\end{equation*}
uniformly  with respect to the speed of light $c\in (0,\infty)$. As a consequence, this allowed the authors
to derive the MHD equations from Navier--Stokes--Maxwell equations as $c\to \infty$.  It is emphasized therein that the construction of global solutions uniformly with respect to the speed of light relies upon the improvement of the $L^1_tL^\infty$-bound on the velocity field, previously established in \cite{Masmoudi_2010}.
 
 As we are also  interested in the behavior of solutions in the asymptotic regime $c\rightarrow  \infty$, allow us  now to briefly highlight a few details on the resulting equations in that regime.  Observe, in particular, that  the system \eqref{Main_System} is reduced to the following inhomogeneous MHD equations when the speed of light   goes to infinity:
 \begin{equation}\label{MHD_System}
	\begin{cases}\tag{MHD}
		\begin{aligned} 
	&	\partial_t\rho+\operatorname{div} (\rho u)=0,\vspace{0.2cm}\\
&\partial_t (\rho u)+\operatorname{div}(\rho u\otimes u) -\nu 
\Delta u +\nabla\left( p- \frac{|B|^2}{2}\right)=B\cdot \nabla B, &\div u =0,\vspace{0.2cm}\\  
			&  \partial_t B  + u\cdot \nabla B - \frac{1}{\sigma} \Delta B  = B\cdot \nabla u , &\div B = 0.  & 
		\end{aligned}
	\end{cases}
\end{equation} 
Unlike \eqref{Main_System} and \eqref{NSM-equa}, the system \eqref{MHD_System}  enjoys a scaling invariance property. More precisely,  for any $ \lambda >0$, it is invariant under the transformation 
\begin{equation*}
(\rho_\lambda, u_\lambda, B_\lambda)(t,x) \bydef (\rho(\lambda^2t, \lambda x),   \lambda u(\lambda^2 t, \lambda x), \lambda B(\lambda^2t, \lambda x)) ,
\end{equation*}
in the sense that, if $(\rho,u,B)$ is a solution of \eqref{MHD_System} with the initial data $(\rho_0, u_0,B_0)$, then $(\rho_\lambda, u_\lambda, B_\lambda) $ is   also a  solution to the same equations   with the scaled initial data $(\rho_0( \lambda\cdot  ), \lambda u_0( \lambda\cdot   ),\lambda 
B_0(\lambda 
\cdot  ) )$.

Due to its parabolic nature, there are many more results on the analysis of \eqref{MHD_System} than there are on the Navier--Stokes--Maxwell equations \eqref{Main_System}. We refer to \cite{AMBP_2007__14_1_103_0, AP08, BIE201985, CL18, CTZ11,  DZ18,  Gerbeau_1997, GG14,   HW13,  HUANG2013511, LTY17,  LXZ15, MT83, WZ17} for some literature on the mathematical analysis of the inhomogenous system \eqref{MHD_System} as well as its homogeneous counterpart (i.e., in the case $\rho \equiv 1$).  We also refer the interested reader to \cite{Abidi:2021ud, AGZ12, AP07, D23, DM12, DM19, HB12, PZZ13} for more results on the particular case of a trivial magnetic field, i.e., when the system reduces to the inhomogeneous Navier--Stokes equations \eqref{Navier_Stokes}.

 Apart from that, the inviscid version of \eqref{NSM-equa}, corresponding the case $\nu=0$, is less studied and much more complicated since the   equations in that case are purely hyperbolic. We refer to \cite{ah, ah2} for recent advances on the analysis of the homogeneous Euler--Maxwell equations in the plane.

\subsection{Aims and main results}

The purpose of this paper is to investigate the global existence and uniqueness of solutions of  the inhomogeneous Navier--Stokes--Maxwell  equations \eqref{Main_System}, uniformy with respect to the speed of light, in functional spaces that are as close as possible to the natural energy of the system. 
   
Before we move on to the statements of our results, allow us first to  recall the formal bounds which follow from the structure of the equations \eqref{Main_System}. A calculation similar to \eqref{Ident_Energy_1} shows that smooth   solutions  to \eqref{Main_System}   satisfy the identity
\begin{equation}\label{Energy_Identity_u}
\mathcal{E}^2 (t) +2\nu\int_0^t \Vert \nabla u(\tau)\Vert_{L^2}^2d\tau + \frac{2}{\sigma}\int_0^t \Vert j(\tau) \Vert_{L^2}^2d\tau= \mathcal{E}^2(0)\bydef\mathcal{E}_0^2,
\end{equation}
for all $t\geq 0,$ where we set     
$$ \mathcal{E}^2(t) \bydef\|\sqrt{\rho} u(t)\|_{L^2}^2 + \|E(t)\|_{L^2}^2 + \|B(t)\|_{L^2}^2 .$$

 On the other hand, due to the incompressibility of the fluid, the continuity  equation  allows us to recover a useful maximum principle. Specifically, it holds,  for all positive times $t> 0$, that
 \begin{subequations}\label{Cons_L_p_density}  
 \begin{equation}\label{mass:conservation}
\Vert \rho(t)\Vert_{L^\infty}=\Vert \rho_0\Vert_{L^\infty} ,
\end{equation}
as well as 
 \begin{equation}\label{mass:conservation2}
\Vert 1-\rho(t)\Vert_{L^\infty}=\Vert 1- \rho_0\Vert_{L^\infty}.
\end{equation}
\end{subequations}

Throughout this paper, we only consider the no-vacuum case, i.e., we assume that the initial density is bounded  and remains away from zero.  That is to say, it is assumed, for some threshold constants  $\underline{\rho}$ and  $\overline{\rho}$,  that   
\begin{subequations} 
\begin{equation}\label{Assumption A}
 0< \underline{\rho} \leq \rho_0(x) \leq   \overline{\rho},  
\end{equation}  
for all $x\in \mathbb{R}^2$.

  Furthermore, we  assume that the density is close to a constant value, which can be assumed to be ``one''. For simplicity, we write that
  \begin{equation}\label{assumption:rho2}
\norm {\rho_0 - 1}_{L^\infty(\mathbb{R}^2)}  \ll 1,
\end{equation}
in the sense that the left-hand side is smaller than a small (universal) constant which is independent of the variables of the problem.  As we will see, this assumption is crucial in the proofs of our results.
\end{subequations}  

 The next result is central to our work and establishes the existence  and uniqueness of global weak solutions to \eqref{Main_System}, uniformly with respect to the speed of light.

\begin{theorem}\label{Thm:1}
	Let $p\in (1,2)$ and $s\in [s_*,1)$, with $s_* \bydef 2 \left( 1- \frac{1}{p}\right) \in (0,1)$.
	Let $\rho_0$ be a bounded function satisfying assumptions \eqref{Assumption A} and \eqref{assumption:rho2}. Further consider divergence-free vector fields $(u_0,E_0,B_0)$ satisfying the bounds
	$$ u_0 \in \dot{B}^{-1+\frac{2}{p}}_{p,1}(\mathbb{R}^2),\qquad (E_0,B_0)\in H^s(\mathbb{R}^2).$$
	Note that $E_0$ needs to be divergence-free when considering Ohm's law \eqref{Ohms-law2}. However, it does not need to be divergence-free if one considers \eqref{Ohms-law1}.
	
	Then, the inhomogeneous Navier--Stokes--Maxwell equations \eqref{Main_System}, with initial data $(\rho_0,u_0,E_0,B_0)$, has at least one global solution $(\rho,u,E,B)$ that satisfies the energy inequality \eqref{Energy_Identity_u}, the maximum principle \eqref{Cons_L_p_density} and enjoys the additional bounds
	\begin{equation*}
		\begin{gathered}
			u \in  L^\infty _{\loc}(\mathbb{R}^+; \dot{B}^{-1+\frac{2}{p}}_{p,1} (\mathbb{R}^2))
			\cap L^2_{\loc}(\mathbb{R}^+; \dot H^1\cap L^\infty(\mathbb{R}^2))
			\cap L^{q,1}_{\loc}(\mathbb{R}^+; \dot{W}^{2,p} (\mathbb{R}^2)),
			\\
			(E,B)\in L^\infty_{\loc}( \mathbb{R}^+;  \dot{H}^{s}(\mathbb{R}^2)) ,
			\qquad cE \in L^2_{\loc}( \mathbb{R}^+; \dot{H}^{s}(\mathbb{R}^2)),
		\end{gathered}
	\end{equation*}
	uniformly with respect to $c\in (0,\infty)$,
	where $\frac 1p+\frac 1q=\frac 32$.
	
	Moreover, if $s\in (\frac{1}{2},1)$, then the velocity field enjoys the Lipschitz bound
	$$ u \in L^1 _{\loc}( \mathbb{R}^+;  \dot{W}^{1,\infty}(\mathbb{R}^2))$$
	 and $ (\rho,u,E,B)$ is actually the unique energy solution to \eqref{Main_System}.
\end{theorem}

\begin{remark}
	In our main theorem, above, all bounds are uniform with respect to the speed of light $c\gg 1$, except the Lipschitz bound on the velocity field. However, this latter bound is only required to establish the uniqueness of solutions, whereas the proof of existence of solutions (uniformly with respect to the speed of light) does not rely on it.
\end{remark}

The proof of Theorem \ref{Thm:1} is given in several steps which are laid out throughout Sections \ref{Section:a priori ES} to \ref{Section.TW:es2}. More precisely, Section \ref{Section:a priori ES} is concerned with the justification of the existence of at least one global solution $(\rho,u,E,B)$. The proof of existence \emph{per se} is completed in Section \ref{proof_existence}. Then, the crucial Lipschitz bound on $u$, in the case $s\in (\frac 12, 1)$, is established in Proposition \ref{corollary:u-Lip} from Section \ref{Section.TW1}, as a consequence of essential time-weighted estimates. As for the uniqueness of solutions, it is split into two parts. First, in Section \ref{Section.stability}, we establish a general weak--strong uniqueness principle, which is stated in Theorem \ref{Thm:stability}. Second, we show in Section \ref{Section.TW:es2} that, in fact, when $s\in (\frac 12, 1)$, the solutions constructed earlier do enjoy all assumptions required by Theorem \ref{Thm:stability}, thereby yiedling the uniqueness of solutions.

\begin{remark}
	Theorem \ref{Thm:1} is comparable to the recent results on the inhomogeneous Navier--Stokes equations \eqref{Navier_Stokes} featured in \cite{Danchin_Wang_22}. Indeed, if $E\equiv B \equiv 0$, then the existence statement from Theorem \ref{Thm:1} is analogous, and almost identical, to Theorem 2.4 in \cite{Danchin_Wang_22}. However, our uniqueness statement remains slightly weaker due to the restriction of the parameter $s$ to the range $(\frac 12, 1)$, which is related to the impact of the electromagnetic field.
\end{remark}

\begin{remark}
	Theorem \ref{Thm:1} is also comparable to the results on the homogeneous Navier--Stokes--Maxwell equations \eqref{NSM-equa} from \cite{ag20} and \cite{Masmoudi_2010}. Indeed, assuming $\rho\equiv 1$ in the statement of Theorem \ref{Thm:1} leads to a well-posedness result for \eqref{NSM-equa} where the velocity field $u$ belongs to functional spaces which have the same homogeneity as the natural energy space $L^\infty(\mathbb{R}^+;L^2(\mathbb{R}^2))\cap L^2(\mathbb{R}^+;\dot H^1(\mathbb{R}^2))$, while the smoothness of the electromagnetic field $(E,B)$ is propagated in $H^s(\mathbb{R}^2)$.
\end{remark}

\begin{remark}
	The results from \cite{Abidi:2021ud} on the inhomogeneous Navier--Stokes system \eqref{NSM-equa} show that it is possible to construct solutions where the initial velocity field $u_0$ belongs to $B^0_{2,1}(\mathbb{R}^2)$, only. However, this comes at the cost of assuming that the initial density $\rho_0$ satisfies the regularity condition $\rho_0^{-1}-1\in \dot{B}^{\varepsilon}_{\frac{2}{\varepsilon},1}(\mathbb{R}^2)$, for some $\varepsilon>0$. It would be interesting to see whether Theorem \ref{Thm:1} can be adapted to reach a similar result in the endpoint case $p=2$.
\end{remark}

The fact that the bounds in Theorem \ref{Thm:1} are uniform with respect to the speed of light $c\in (0,\infty)$ allows us to study the limit $c\to \infty$ and derive the inhomogeneous MHD equations \eqref{MHD_System}. More precisely, we obtain the following corollary.

\begin{corollary}
	For any given initial data $(\rho_0,u_0,E_0,B_0)$ as in Theorem \ref{Thm:1}, consider the global solution $(\rho^c,u^c,E^c,B^c)$, for each $c\in (0,\infty)$, provided by that theorem. Then,  the set $\{(u^c,E^c,B^c)\}_{c>0}$ is relatively compact in $L^2_{t,x,\loc}$. More precisely, for any sequence  $\{(\rho^{c_n},u^{c_n},E^{c_n},B^{c_n})\}_{n\in \mathbb{N}}$ with $c_n\to \infty$, there is a subsequence, which we do not relabel, for simplicity, satisfying, as $n\to \infty$, that
	\begin{equation*}
		\rho^{c_n}   \rightharpoonup \rho , \quad \text{in} \quad L^2_{t,x,\loc},
	\end{equation*}
	and
	\begin{equation*}
		(u^{c_n},E^{c_n},B^{c_n})   \rightarrow (u ,E ,B ), \quad \text{in} \quad L^2_{t,x,\loc},
	\end{equation*}
	where $(\rho ,u ,E ,B )$ solves \eqref{MHD_System}. 
\end{corollary}

The proof of the preceding corollary is performed in the same fashion as the proof of \cite[Corollary 1.3]{ag20} and \cite[Corollary 1.2]{ah}. More precisely, it is established by relying on the uniform bounds satisfied by the solutions from Theorem \ref{Thm:1}, along with classical compactness techniques due to Aubin and Lions \cite{AJ63,  L61} (one can also make use of the sharp compactness criteria by Simon \cite[Corollary 1]{JS87}).

\subsection{Challenges and strategy of proof.}

We discuss now the main challenges in the analysis of \eqref{Main_System} and the crucial ingredients that are employed in the proof of Theorem \ref{Thm:1}. We split the discussion into two parts, which respectively address the existence and the uniqueness of solutions.

\subsubsection*{Existence of global solutions}

The existence of weak solutions to \eqref{Main_System} is proven by means of compactness arguments. Generally speaking, the idea is to approximate \eqref{Main_System} by another system with a unique smooth solution. Most importantly, the approximate system needs to be constructed in a way which preserves the same essential properties (such as conservation laws and a priori energy bounds) of the original set of equations \eqref{Main_System}.
Then, the next step consists in showing that the approximate solutions converge, at least in a distributional sense, to a global weak solution of the original system \eqref{Main_System}.

As previously explained, the energy a priori bound is not enough to provide us with sufficient information that would ensure the convergence
$$  j_n\times B_n \rightarrow j\times B, \quad  \text{in} \quad\mathcal{D}'(\mathbb{R}^2),$$
where $(j_n,B_n)_{n\in\mathbb{N}}$ is the smooth approximation of $(j, B)$. In order to overcome this lack of compactness, the idea is to seek the propagation of higher regularities of the electromagnetic field $(E,B)$ in $H^s(\mathbb{R}^2)$, with $s>0$. To that end, as shown in Proposition \ref{energy-Hs}, it will be necessary to establish a control of the velocity field $u$ in the space $L^2_{\loc} (\mathbb{R}^ +;L^\infty (\mathbb{R}^2))$. Showing this bound on $u$ will be one of the most challenging steps in our work.

To achieve it, we will split the velocity field $u=v+w$ into two parts: a fluid part denoted by $v$, and an electromagnetic part denoted by $w$. The corresponding equations are given by
\begin{equation}\label{v:equa}
\left\{
\begin{array}{ll}  \tag{v-EQ}
\rho  \left(  \partial_t v + u\cdot \nabla v  \right)- \Delta v + \nabla p_v   =0 , \\
\div v=0,\\
v|_{t=0}= u_0 
\end{array}
\right.   
\end{equation}  
and
\begin{equation}\label{w:equa}
\left\{
\begin{array}{ll}  \tag{w-EQ}
\rho  \left(  \partial_t w + u\cdot \nabla w  \right) - \Delta w + \nabla p_w =     j \times B, \\
\div  w=0,\\
w|_{t=0}=0,
\end{array}   
\right.   
\end{equation}
where the pressure is split as $p= p_v + p_w$, as well. Note that the above systems are both advected by the full velocity field $u=v+w$.

There are several reasons behind the preceding decomposition. In particular, observe first that the fluid system \eqref{v:equa} is precisely the inhomogeneous Navier--Stokes equations if there is no electromagnetic field, i.e., if $E\equiv B \equiv 0$. Accordingly, one can expect \eqref{v:equa} to be similar to \eqref{Navier_Stokes}.

Second, due to the divergence-free condition on the velocity field $u$, we can check that $v$ satisfies an $L^2$-energy identity (see Lemma \ref{Energy_Estimate_v_w}). Therefore, by combining that identity with the $L^2$-energy of the full velocity field \eqref{Energy_Identity_u}, it follows that $w$ also has finite energy.

A key idea in the analysis of $v$ relies on maximal parabolic regularity estimates in Lorentz spaces (in the time variable), in combination with time-weighted estimates in functional spaces which have the same homogeneity as the natural energy spaces $L^\infty_tL^2_x \cap L^2_t\dot{H}^1_x$. This step is adapted from the work of \cite{Danchin_Wang_22} on the inhomogeneous Navier--Stokes equations \eqref{Navier_Stokes}. However, this adaptation is not straightforward since the fluid component $v$ is advected by the full velocity field $u$, which is influenced by the electromagnetic field.

On the other hand, most of the estimates on $w$ will be done in higher regularity spaces. This will be possible by virtue of the additional Sobolev regularity of the electromagnetic field, which is exploited in the control of Lorentz's source term $ j\times B$.

This strategy eventually leads to the required control of $u$ in $L^2_{\loc} (\mathbb{R}^ +;L^\infty (\mathbb{R}^2))$, in the form of a logarithmic estimate
\begin{equation} \label{U-BOUND}
\begin{aligned}
\int_{t_0}^t  \norm {u(\tau)}^2_{  L^\infty }  d\tau
& \lesssim     \Big( g (t_0,t) + f(t_0,t) \log  \Big( e+     \frac{   \mathcal{E}_0 \norm  { B }_{L^\infty ([t_0,t ]; \dot{H}^s)}^2}{ f(t_0,t) }  \Big) \Big) 
\\
& \quad  \times \exp \left( C \norm {\rho_0}_{ L^\infty}^2\mathcal{E}_0^2 \right)  ,
\end{aligned}
\end{equation}
for any $0\leq t_0<t$,
where $g$ and $f$ are continuous functions of time which are locally bounded by constants depending only on the initial data, uniformly with respect to the speed of light. This estimate is established in Lemma \ref{lemma.boundedness-u}.

Now, on the one hand, we have already emphasized how the norm of the electromagnetic field $(E,B)$ in $L^\infty_{\loc} (\mathbb{R}^ +;H^s (\mathbb{R}^2))$ is controled by the norm of the velocity field $u$ in $L^2_{\loc} (\mathbb{R}^ +;L^\infty (\mathbb{R}^2))$. On the other hand, the preceding estimate shows that $u$ in $L^2_{\loc} (\mathbb{R}^ +;L^\infty (\mathbb{R}^2))$ is controled by $(E,B)$ in $L^\infty_{\loc} (\mathbb{R}^ +;H^s (\mathbb{R}^2))$. Combining these two estimates allows us to propagate the control of these norms on $u$ and $(E,B)$, for all positive times, uniformly with respect to $c\in (0,\infty)$.

Notice that \eqref{U-BOUND} is similar to a logarithmic bound established in \cite{ag20} for the homogeneous Navier--Stokes--Maxwell system \eqref{NSM-equa}. A key argument used therein relies on the functional interpolation inequality
\begin{equation}
	\label{AAAaAA}
	\norm { ({\rm Id}- S_0)h}_{L^2 ([t_0,t];L^\infty)}\lesssim \norm {  h}_{L^2([t_0,t]; \dot{H}^{1}) }
	\log^{\frac{1}{2}}  \left( e + \frac{  \norm { h}_{L^2([t_0,t];\dot{B}_{2,\infty}^{s}) }}{ \norm {  h}_{L^2([t_0,t]; \dot{H}^{1}) }} \right) ,
\end{equation}
for $s>1$ and suitable functions $h$, where $S_0$ is an operator which truncates the frequencies of $h$ to the unit ball (in Fourier variables).

The use of \eqref{AAAaAA} works well in the context of \eqref{NSM-equa}. However, in the setting provided by the inhomogeneous Navier--Stokes--Maxwell equations \eqref{Main_System}, this is not the case, due to the influence of the non-constant and non-smooth fluid density $\rho$. Nevertheless, we are able to show the validity of \eqref{U-BOUND} by relying on an alternative approach based on the following two key observations:
\begin{itemize}
	
	\item
	An analysis of \eqref{v:equa} shows that the bounds on $v$ in critical spaces, and eventually in $L^2_t L^\infty_x$, come with  a linear contribution of $\norm {w}_{L^{\eta,2}_t L^m_x} $, with $\tfrac{1}{\eta}+ \tfrac{1}{m}=\tfrac{1}{2}$. Despite the fact that $\eta,m\in (2,\infty)$, it is not possible to show that the latter bound on $w$ can be obtained by an interpolation between $L^\infty_tL^2$ and $L^2_t \dot{H}^1$. In other words, the natural a priori energy estimate is insufficient to control $w$ in $L^{\eta,2}_t L^m_x$.
	
	However, we achieve  \eqref{U-BOUND} by first establishing a control of $\norm {w}_{L^{\eta,2}_t L^m_x} $ in terms of a logarithmic contribution of a slightly supercritical norm of $w$. To be precise, we employ the interpolation inequality \eqref{Interpolation_Ineq}, instead of \eqref{AAAaAA}, to derive that
	\begin{equation*} 
		\begin{aligned}
			\norm {w}_{ L^{\eta,2}_tL^m} & \lesssim   \norm {w}_{ L^{\eta,\infty}_tL^m}  \left( 1 + \log^{\frac 12} \left( t^{ \frac 1\eta  }  \frac{ \norm {w}_{ L^{\infty}_tL^m} }{ \norm {w}_{ L^{\eta,\infty}_tL^m}}  \right ) \right),
		\end{aligned}
	\end{equation*}
	where $\frac{1}{\eta} + \frac{1}{m} = \frac{1}{2}$.
	
	Then, we observe  that $  L^{\eta,\infty}_t L^m$ is an interpolation space between $L^\infty_tL^2$ and $L^2_t \dot{H}^1$, for $ \tfrac{1}{\eta}+ \tfrac{1}{m}=\tfrac{1}{2}$, and we combine this property with a     bound  on the $ L^{ \infty}_tL^m$ norm of $w$  in terms of the $H^s$ norm of the magnetic field $B$. At the end, the preceding arguments allow us to derive a suitable logarithmic bound on $v$ in $L^2_tL^\infty$.
	
	\item
	The $L^2_tL^\infty$ bound on $w$ is simpler and follows from the fact that $w$ enjoys the supercritical regularity $w\in L^\infty_t \dot B^1_{p,2}\cap L^2_t\dot W^{2,p}$. This regularity estimate is shown in Lemma \ref{lemma:w-ES}. The ensuing $L^2_tL^\infty$ bound  on $w$ is then given in \eqref{w-last1}.
	
\end{itemize}

\subsubsection*{Uniqueness of energy solutions}

The stability estimates leading to the uniqueness of the solution of \eqref{Main_System} are inspired by the methods from \cite{Danchin_Wang_22} on the inhomogeneous Navier--Stokes equations \eqref{Navier_Stokes}. However, the arguments therein need to be considerably refined in order to handle the bounds enjoyed by the solution $w$ of \eqref{w:equa}, which are constrained by the $H^s$ regularity of the electromagnetic field.

Specifically, in Section \ref{Section.stability}, below, we establish a general weak--strong stability result which, subsequently, allows us to deduce the uniqueness of solutions stated in Theorem \ref{Thm:1}. To that end, we first show that the velocity field $u$ is Lipschitz in $x$, despite the roughness of the non-constant density $\rho$. This is achieved by considering, again, the decomposition $u=v+w$, where $v$ and $w$ solve \eqref{v:equa} and \eqref{w:equa}, respectively. Then, we deal with each part of the velocity field separately.

For the fluid part $v$, we employ time-weighted estimates (see Lemma \ref{lemma:v2}), in the spirit of \cite{Danchin_Wang_22}. It is important to notice that adding a time-weight in the norms of $v$ allows us to control it in higher regularity spaces.

As for $w$, we observe that its initial data is identically zero. Hence, it can be shown that its spatial regularity only depends on the smoothness of electromagnetic fields, through Lorentz's source term $j\times B$. Specifically, we prove in Lemma \ref{lemma:w2}, below, that $w$ enjoys suitable additional supercritical bounds, as soon as the Sobolev regularity of the electromagnetic fields is large enough ($s>\frac 12$, as stated in Theorem \ref{Thm:1}). Consequently, by an abstract interpolation argument, one then shows that $w$ is Lipschitz, as well. We point to Proposition \ref{corollary:u-Lip}, which summarizes the additional bounds on the velocity field under the restriction $s>\frac{1}{2}$.

After this step, the final time-weighted estimates on $w$ are established in Section \ref{subsection:w-TW}, and more precisely in Proposition \ref{prop:w:tw**}. It is important to point out that the proof of this proposition does not require the use of maximal parabolic regularity estimates. Instead, it relies on a crucial bound, proven in Lemma 
\ref{lemma:duality}, which utilizes the structure of Lorentz's force and is based on a duality argument.

Once the main bounds on $v$ and $w$ are established, we prove the actual weak--strong stability result stated in Theorem \ref{Thm:stability}. A fundamental idea leading to this new stability result relies on controlling the difference of the densities of two solutions $\delta \rho\bydef \rho_1 - \rho_2$ in the space
$$
Z(t)+X(t) \bydef  \sup_{\tau \in (0,t]} \left( \tau^{-1}  \norm {\delta \rho(\tau)}_{\dot{H}^{-1}}\right) + \sup_{\tau \in (0,t]} \left( \tau^{-\frac{1}{2}  }  \norm {\delta \rho(\tau)}_{\dot{W}^{-1,r}}\right),
$$
for some $r\in (2,\infty)$. The two components of the preceding norm are designed to deal with the terms 
$$
\int_{\mathbb{R}^3} \delta \rho \partial_t v \delta u dx
\qquad \text{and } \qquad
\int_{\mathbb{R}^3} \delta \rho \partial_t w \delta u dx,
$$
respectively, which appear in the energy estimate for the difference of the velocities of two solutions $ \delta u \bydef u_1 - u_2$.

The key ingredients that we utilize to establish the stability of $\delta \rho$ are obtained in Section \ref{subsection:TE}, where general results on the transport equation in negative Sobolev spaces are proven. This section is self-contained and of independent interest.

In summary, the art of the proof of the weak--strong stability, and, subsequently, of the uniqueness statement in Theorem \ref{Thm:1}, is displayed, first, in the suitable splitting of the velocity field $u=v+w$ and, second, in the appropriate  choice of the space defined by $X(t)+Z(t)$ for the control of $\delta \rho$. These new ideas allow us to obtain a stability result at a minimum cost, in terms of Sobolev regularity  $s>\frac{1}{2}$, on the electromagnetic field. Finally, we should emphasize that we have no proof of optimality of the threshold value $s=\frac{1}{2}$. However, in our proofs below, this threshold is justified by the need of  the Lipschitz regularity of the velocity field $u$, which seems rather sharp.

\subsection{Notation}

We introduce here some notation which is used routinely throughout the paper.

First, we emphasize that the definitions of all relevant functional spaces, as well as their basic properties, are recalled in the appendix.

Next, for any Banach space $X$, all $p\in [1,\infty]$ and any $t\in [0,\infty)$, we adopt the shorthand notation
$$ L^p_tX \bydef L^p([0,t); X).$$
In particular, the quantity $\norm {\cdot}_{  L^p_tX}$ will be regarded as a function of time $t\in [0,\infty)$.

Finally, unless otherwise mentioned, the letter ``$C$'' will be used to denote a universal constant  which only depends on fixed parameters. In particular, such a constant will always be assumed to be independent of the speed of light $c$, but it may depend on the size of the initial data. Furthermore, 
we will regularly utilize the symbol ``$A\lesssim B$'' in order to express that  ``$A\leq C B$'', for some universal constant $C>0$.

\section{A priori estimates and existence of global solutions} \label{Section:a priori ES}

Our roadmap to build  global solutions of \eqref{MHD_System_Nonhom} is based on standard compactness arguments. We refer to \cite{ah, Danchin_Wang_22} for similar problems where these techniques have been implemented. Specifically, here, we first consider the approximation scheme   
  \begin{equation}\label{AP-system} 
		\begin{cases}
			\begin{aligned}
			&\partial_t \rho _n +(S_nu_n) \cdot\nabla \rho_n  =0,&\\
				&\rho _n  (\partial_t u_n +(S_nu_n) \cdot\nabla u_n )- \nu \Delta u_n = - \nabla p_n + (S_nj_n) \times B_n, &\div u_n =0,&
				\\
				&\frac{1}{c} \partial_t E_n - \nabla \times B_n = - j_n ,  &
				\\
				&\frac{1}{c} \partial_t B_n + \nabla \times E_n  = 0 , &\div B_n =0,&  
			\end{aligned}
		\end{cases}
	\end{equation}
	supplemented with the initial data $$(\rho_n,u_n,E_n,B_n)|_{t=0}= S_n(\rho_0,u_{0},E_0,B_0)$$  and Ohm's law
$$ j_n= \sigma \big(cE_n + S_n\mathbb{P} (u_n \times B_n)\big),  \qquad \div j_n =0,  $$
or 
$$  j_n= \sigma \big(cE_n + S_n (u_n \times B_n)\big)  ,$$	
where $S_n$ is a truncation  in Fourier space that restricts the frequencies to the domain $\{|\xi|\leq 2^n \}$ and  $\mathbb{P}$  denotes Leray's projector   $\mathbb{P}\bydef \id - \nabla \Delta^{-1}\div.$

We note that the above approximate system has a unique strong solution $$ (\rho_n,u_n,E_n,B_n),$$ for any $n\in \mathbb{N}$. (This follows from classical fixed point arguments.) This approximate solution, being regular, will   be suitable for all the computations in this section.
We are now going to establish adequate a priori bounds, uniformly with respect to the parameter $n$, on the approximate sequence of solutions, which will then allow us to deduce the existence of solutions to \eqref{Main_System}, by standard compactness arguments.

However, as is customary, in order to simplify the notation in the estimates below, we are now going to perform all estimates directly on \eqref{Main_System} by assuming that its solutions are smooth, rather than working on the approximate system \eqref{AP-system}.

  \subsection{Energy estimate for electromagnetic fields}
 A necessary step towards proving Theorem \ref{Thm:1} is the analysis of Maxwell's system
 \begin{equation}\label{MX*}
\left\{
\begin{array}{ll}
\frac{1}{c}\partial_t E-\nabla\times B  = - j ,\vspace{0.2cm}\\
\frac{1}{c}\partial_t B+\nabla\times E=0,
\end{array}
\right. 
\end{equation}
where $j$ is given by \eqref{Ohms-law1} or \eqref{Ohms-law2}. 

In this paper, we do not need to perform any deep analysis of the system above.  More precisely, we will only need to control the electromagnetic fields in Sobolev spaces, and this is done by a simple energy estimate.  The following proposition, whose proof can be found in \cite[Proposition 2.1]{ag20}, will serve our purpose, later on.

 \begin{proposition}\label{energy-Hs}   
 Let $s\in (-1,1)$ and $E$, $B$ be a solution of Maxwell's system \eqref{MX*} supplemented by either choice of Ohm's laws \eqref{Ohms-law1} or \eqref{Ohms-law2}. Then, it holds, for any $0\leq t_0\leq t$, that
 \begin{equation*}
 \begin{aligned}
 \Vert B(t)\Vert_{\dot{H}^s}^2&+ \Vert E(t)\Vert_{\dot{H}^s}^2 +\sigma c^2\int_{t_0}^t\Vert E(\tau)\Vert_{\dot{H}^s}^2 d\tau\\
& \qquad \leq\,  \left(\Vert B(t_0)\Vert_{\dot{H}^s}^2+ \Vert E(t_0)\Vert_{\dot{H}^s}^2\right) \exp\left(C\sigma \int_{t_0}^t \Vert u(\tau)\Vert_{L^\infty\cap \dot{B}^1_{2,\infty}}^2d\tau\right) ,
\end{aligned} 
\end{equation*}   
where   $C>0$ is a universal constant that is independent of $c$.
 \end{proposition}
 Proposition \ref{energy-Hs} shows that the propagation of Sobolev regularity of the electromagnetic fields relies on the boundedness of the velocity field $u$ in $L^2_{\loc}(\R^+, L^\infty(\mathbb{R}^2))$. This bound is the subject of the next section.

 \subsection{Boundedness of velocity field}\label{Sec_Boun_Veloc}
 This section is devoted to the derivation of an a priori bound on the velocity field $u$ in $L^2_tL^\infty_x$ which will be combined afterwards with the energy estimate from Proposition \ref{energy-Hs}, above.

 Henceforth, we will often assume, for simplicity of notation, that $\nu=1$ and we emphasize that all the arguments below hold for any fixed $\nu>0$, as well.

In order to analyze the momentum  equation from \eqref{Main_System}, we   introduce a specific splitting of the velocity field into a fluid   part  and a magnetic part corresponding respectively to the decomposition
 $$ u = v + w ,$$
 where, by definition, the two parts  are governed by the equations \eqref{v:equa} and  \eqref{w:equa}, respectively.

Then, we deal with the  estimates of each part differently. Roughly speaking, the system \eqref{v:equa} will  be treated in a similar way to the inhomogeneous Navier--Stokes equations \eqref{Navier_Stokes}, even though it is advected by the full velocity field $u$. 

On the other hand, the estimates for \eqref{w:equa} will mainly be performed in supercritical spaces, for the trivial initial data in \eqref{w:equa} is obviously smooth, and the magnetic field in the source term of \eqref{w:equa} enjoys some suitable regularity, as will be seen later on.

The primary step for obtaining the necessary  a priori estimates is to show that the preceding decomposition of the velocity field still preserves the $L^2$-energy bound on both  $v$ and $w$. More precisely, we now prove the following lemma.

\begin{lemma}\label{Energy_Estimate_v_w}
It holds, for all $t\geq 0$, that
\begin{equation}\label{energy-v}
\norm{\sqrt{\rho}v(t)}_{L^2}^2 +2\nu \int_0^t \norm {\nabla v(\tau)}_{L^2}^2 d\tau = \norm{\sqrt{\rho_0}u_0}_{L^2}^2 \leq \mathcal{E}_0^2
\end{equation}
and 
\begin{equation}\label{energy-w}
\norm{\sqrt{\rho}w(t)}_{L^2}^2 +2\nu \int_0^t \norm {\nabla w(\tau)}_{L^2}^2 d\tau   \leq 2 \mathcal{E}^2_0,
\end{equation}
where $ \mathcal{E}_0 $ is defined in \eqref{Energy_Identity_u}.
\end{lemma}

\begin{proof}By a standard $L^2$-energy estimate for  \eqref{v:equa}, it is readily seen that  
\begin{equation*}
\begin{aligned}
 \frac 12 \frac{d}{dt} \int_{\mathbb{R}^2} \rho |v|^2(t,x) dx & + \int_{\mathbb{R}^2}| \nabla v (t,x)|^2 dx \\
 &=  \frac 12     \int_{\mathbb{R}^2} (\partial _t \rho) |v|^2 (t,x) dx - \frac 12  \int_{\mathbb{R}^2} \rho \div ( u   |v|^2 ) (t,x) dx  . 
  \end{aligned}
\end{equation*}
Therefore, notice that an integration by parts  followed by the fact that 
$$u\cdot \nabla \rho = - \partial _t \rho ,$$
 entails  that 
$$\int_{\mathbb{R}^2} \rho \div ( u   |v|^2 ) dx = \int_{\mathbb{R}^2} \partial _t \rho |v|^2 dx .$$ 
Thus, the first identity \eqref{energy-v} follows. 

 As for the second bound \eqref{energy-w},  it easily follows by combining \eqref{energy-v} with the energy bound on the full velocity field \eqref{Energy_Identity_u}, followed by a straightforward use of the triangle inequality.  The proof of Lemma \ref{Energy_Estimate_v_w} in now completed.  
\end{proof}

Now, we establish adequate bounds on each part of the velocity field and, eventually, deduce the desired $L^2_tL^\infty$ estimate of the full velocity field $u$. We begin with the bounds on  $w$ which will be employed afterwards   in the estimates of $v$.

\begin{lemma}[Estimates for the $w$-part]\label{lemma:w-ES}
Let $ s\in (0,1)$ and $p\in (1,2)$ be given by  $p\bydef \tfrac {2}{2-s}$. Further assume that \eqref{assumption:rho2}  is satisfied. Then, it holds, for any $0\leq t_0 < t $, that 
  \begin{equation*} 
\begin{aligned}
\norm {  w }_{L^\infty([t_0,t ];   \dot{B}^1_{p,2})}
+&\norm { (\partial_t w, \nabla ^2 w , \nabla p_w)}_{L^2([t_0,t ];   L^p)}
\\
&\lesssim  \mathcal{E}_0 \norm  { B }_{L^\infty ([t_0,t ]; \dot{H}^s)}   \exp \left(C  \norm {\rho_0}_{L^\infty}^4\mathcal{E}_0   ^4 \right),
\end{aligned}  
\end{equation*}
 where $C>0$ is a universal constant which is independent of $c$ and the initial data. 
\end{lemma}
\begin{proof}  
For simplicity, we prove the lemma for $t_0=0$ and we emphasize that the proof remains the same for any $t_0\geq 0$.   First of all, we rewrite \eqref{w:equa} as a perturbation of the Stokes system  
\begin{equation}\label{w-Stokes}
\left\{
\begin{aligned}
&\partial_t w - \Delta w + \nabla p_w =   (1-\rho) \partial_t w -\rho u\cdot \nabla w+   j \times B ,\\
& \div w=0,\quad w|_{t=0}=0.  
\end{aligned}
\right.
\end{equation} 
Then, by applying    the maximal regularity estimate  \eqref{Maximl_Regularity_Stokes} found in the appendix, we obtain that
 \begin{equation}\label{w:es:AAA1} 
\begin{aligned}
\norm {  w }_{L^\infty _t  \dot{B}^1_{p,2}} &+ \norm { (\partial_t w, \nabla ^2 w , \nabla p_w)}_{L^2_t  L^p}    \\
&  \lesssim   \norm {(1-\rho) \partial_t w}_{L^2_t L^p}  +\norm {\rho u \cdot \nabla w}_{L^2_t L^p}+ \norm{ j \times B }_{L^2_t L^p}  .
\end{aligned}
\end{equation}  
Thus, by further employing H\"older's inequalities, we find that 
 \begin{equation} \label{w:es:AAA2}
\begin{aligned}
\norm {  w }_{L^\infty _t  \dot{B}^1_{p,2}} + \norm {( \partial_t w, \nabla ^2 w , \nabla p_w)}_{L^2_t  L^p}
& \lesssim   \norm {1-\rho }_{L^\infty_{t}L^\infty } \norm{\partial_t w}_{L^2_t L^p}
\\
&\quad+   \norm {\rho}_{L^\infty_t L^\infty}\norm{ u }_{L^4 _t L^4} \norm { \nabla w}_{L^4_t L^{ \frac{4p}{4-p}}}
\\
&\quad + \norm{ j}_{L^2_tL^2} \norm  { B }_{L^\infty _t L^{\frac{2p}{2-p}}}.
\end{aligned}
\end{equation}

On the one hand, due to the embedding (see Lemma \ref{Emb_Besov_Tr_Lemma})
$$L^4_t \dot{B}^{\frac 12}_{ p,1}  \hookrightarrow L^4_tL^{ \frac{4p}{4-p}}$$
in combination with       the interpolation embedding (see   \eqref{interpolation.AP-B})
$$L^\infty _t  \dot{B}^0_{p,2} \cap L^2_t \dot{W}^{1,p} \hookrightarrow L^4_t \dot{B}^{\frac 12}_{ p,1},$$
we write  that
\begin{equation*}
\begin{aligned}
\norm { \nabla w}_{L^4_t L^{ \frac{4p}{4-p}}} \lesssim  \norm { \nabla w}_{L^4_t \dot{B}^{\frac 12}_{ p,1}} 
\lesssim  \norm { \nabla w}_{L^\infty_t \dot{B}^{ 0}_{ p,2}}^{\frac 12} \norm { \nabla w}_{L^2_t \dot{W}^{ 1,p}}^{\frac 12}. 
\end{aligned}
\end{equation*}
Therefore, Young's inequality implies that
\begin{equation*} 
\begin{aligned}
\norm { \nabla w}_{L^4_t L^{ \frac{4p}{4-p}}} 
\lesssim&\, \norm {  w}_{L^\infty_t \dot{B}^{ 1}_{ p,2}}+ \norm { \nabla^2 w}_{L^2_t L^p}. 
\end{aligned}
\end{equation*}

On the other hand,   recalling  that  $1-2/p=s-1$ and employing 
 Sobolev's embedding $\dot{H}^s \hookrightarrow  L^{ \frac{2p}{2-p}}$ (see Lemma \ref{Emb_Besov_Tr_Lemma}, again)  yields
\begin{equation}\label{Sobolev:embedding:AAA} 
\| B \|_{L^\infty _t L^{\frac{2p}{2-p}}}\lesssim \norm  { B }_{L^\infty _t \dot{H}^s}.
\end{equation}

Hence, by gathering the preceding estimates  and exploiting the maximum principle \eqref{mass:conservation}, \eqref{mass:conservation2}, the energy inequality \eqref{Energy_Identity_u}, and assumptions  \eqref{Assumption A}  and  \eqref{assumption:rho2}, we infer that 
  \begin{equation}\label{Estimate_Energy_w}
\begin{aligned}
\norm {  w }_{L^\infty _t  \dot{B}^1_{p,2}} &+ \norm { (\partial_t w, \nabla ^2 w , \nabla p_w)}_{L^2_t  L^p}  \\
&   \leq  C\norm {\rho_0}_{  L^\infty}\norm{ u }_{L^4 _t L^4} \left(  \norm {  w }_{L^\infty _t  \dot{B}^1_{p,2}} +  \norm {\nabla ^2 w}_{L^2_t  L^p}  \right) 
 +C \mathcal{E}_0 \norm  { B }_{L^\infty _t \dot{H}^s},
\end{aligned}
\end{equation}  
for some   constant $C>1$. Now, 
 we subdivide the interval $[0,t]$ as 
$$[0,t]=\bigcup_{i=0}^{N-1}[t_i,t_{i+1}],$$ in  such a way that 
\begin{equation}\label{decomposition1}
\left\{
\begin{aligned}
C \norm {\rho_0}_{ L^\infty}\norm{ u }_{L^4 ([t_i,t_{i+1}]; L^4)} &= \frac 12, \qquad \text{for all}\quad i\in \{ 0,1, \dots,  N-2\} ,\\
C \norm {\rho_0}_{ L^\infty}\norm{ u }_{L^4 ([t_{N-1},t_{N}]; L^4)} &\leq  \frac 12.
\end{aligned}
\right.
\end{equation}
Note that the preceding decomposition exists due to the global bound 
\begin{equation}\label{u-L44}
\norm u_{L^4(\mathbb{R}^+; L^4)} \lesssim \norm u_{L^\infty(\mathbb{R}^+; L^2)}^\frac{1}{2}\norm u_{L^2(\mathbb{R}^+; \dot{H}^1)}^\frac{1}{2}\lesssim \mathcal{E}_0.  
\end{equation} 
Next,   we set, for $i\in \{ 0,1, \dots,  N-1\} $, that
$$W(t_i,t_{i+1}) \bydef  \norm {  w }_{L^\infty([t_i,t_{i+1}];   \dot{B}^1_{p,2})} +  \norm { (\partial_t w,\nabla ^2 w, \nabla p_w)}_{L^2([t_i,t_{i+1}];   L^p)}    $$
and, for $ i \in \{ 1,2, \dots,  N-1 \}$, we define
$$W(t_i ) \bydef  \norm {  w(t_i) }_{  \dot{B}^1_{p,2} } \leq  W(t_{i-1},t_{i }).$$
Accordingly, recasting  \eqref{Estimate_Energy_w}  on each time interval $[t_i, t_{i+1}]$ and employing \eqref{decomposition1} yields, for any $i\in \{ 1,2,\dots , N-1\}$, that
$$
\begin{aligned}
W(t_i,t_{i+1}) &\leq  2C W(t_i ) + 2C  \mathcal{E}_0 \norm  { B }_{L^\infty _t \dot{H}^s} \\ 
& \leq  2 C W(t_{i-1},t_{i })  + 2C  \mathcal{E}_0 \norm  { B }_{L^\infty _t \dot{H}^s}
\end{aligned}
$$
and 
$$
\begin{aligned}
W(0,t_{ 1}) &\leq    2C  \mathcal{E}_0 \norm  { B }_{L^\infty _t \dot{H}^s}  .
\end{aligned}  
$$
Thus,  we find, by induction,  for any $i\in \{ 1,\dots, N-1\}$, that
$$
\begin{aligned}
W(t_i,t_{i+1}) &\leq      \mathcal{E}_0 \norm  { B }_{L^\infty _t \dot{H}^s} \sum_{k=1}^{i  } (2C)^k  \\
&\leq     \mathcal{E}_0\norm  { B }_{L^\infty _t \dot{H}^s}   (2C)^{i+1} .
\end{aligned}
$$  
Hence, we arrive at the conclusion, for any $t\geq 0$, that
\begin{equation*}
\begin{aligned}
    \norm {  w }_{L^\infty([0,t ];   \dot{B}^1_{p,2})} +  \norm {(\partial_t w, \nabla ^2 w,\nabla p_v) }_{L^2([0,t ];   L^p)} \lesssim  \mathcal{E}_0\norm  { B }_{L^\infty _t \dot{H}^s}   (2C)^N.
\end{aligned}
\end{equation*}
At last, we need to find an upper bound on the number of intervals $N$. This is done by  observing that \eqref{decomposition1} and     \eqref{u-L44} entail  the estimate 
\begin{equation*}
N\leq  \left( 2C \norm {\rho_0}_{L^\infty}\norm u_{L^4 ([0,t]; L^4)}  \right)^4  +1 \leq  \left( 2C \norm {\rho_0}_{L^\infty}\mathcal{E}_0  \right)^4  +1.
\end{equation*} 
 Accordingly, up to a suitable change in the constant $C$, we end up with
\begin{equation*}
\begin{aligned}
 \norm {  w }_{L^\infty([0,t ];   \dot{B}^1_{p,2})}
 +&  \norm {(\partial_t w, \nabla ^2 w,\nabla p_v )}_{L^2([0,t ];   L^p)}
 \\
 &\leq C  \mathcal{E}_0 \norm  { B }_{L^\infty([0,t ]; \dot{H}^s)}   \exp \left(C  \norm {\rho_0}_{L^\infty}^4\mathcal{E}_0   ^4 \right),
\end{aligned}
\end{equation*}   
for all positive times $t\geq 0$.  This completes the proof of Lemma \ref{lemma:w-ES}.
\end{proof}

We now turn to estimating the fluid-velocity part $v$. As we will see in the proof below,  this is going to be done by suitably adapting several techniques from \cite[Proposition 2.1]{Danchin_Wang_22}. However, we will need to pay particular attention to the contribution of $w$ in the set of equations  \eqref{v:equa}. 
 More precisely, according to Lemma \ref{lemma:w-ES} above,   bounds on higher regularity of $w$ come with a linear contribution in terms of Sobolev norms of $B$. These bounds on $w$ will be  combined with the interpolation results from Lemma \ref{lemma:interpolation} to show that the linear dependence on Sobolev norms of $B$ can be weakened and will only contribute in the form of a logarithmic correction. This is crucial to obtain the global bounds in Theorem \ref{Thm:1}. The relevant estimates on $w$ are given in Lemma \ref{lemma:v-ES}, below. 

Before getting into the details, we set up some useful notation. To that end, first recall that  $s\in (0,1)$ and $p\in (1,2)$ are fixed such that
\begin{equation*}
p = \frac{2}{2-s}.   
\end{equation*}  
 In addition, we introduce the parameters $q\in (1,2)$ and $\eta,m\in (2,\infty)$ so that
 \begin{equation}\label{p-q:def}
2-\frac 2q = -1 + \frac 2p,\qquad \text{or equivalently}\qquad \frac 1p + \frac 1q = \frac 32,
\end{equation}
 and 
 \begin{equation}\label{parameters_Relation}
\frac 1p= \frac 1m + \frac 12 ,\qquad \frac 1q= \frac 1\eta + \frac 12.
\end{equation}

\begin{lemma}[Estimates for the $v$-part]\label{lemma:v-ES}
Let $ s,p,q,\eta,m$ be the set of  parameters   introduced above, and let $v$ be the solution of \eqref{v:equa}. Then,  it holds, for any $0\leq t_0 < t$,  that
	\begin{equation*}
		\begin{aligned}
			& \Vert  v \Vert_{L^\infty ([t_0,t];  \dot{B}^{-1+ \frac 2p}_{p,1})}+  \Vert(\partial_t v, \nabla ^2 v, \nabla p_v)\Vert_{L^{q,1}([t_0,t];  L^p)}  + \Vert v\Vert_{L^{\eta,1}([t_0,t]; L^m)}
			\\
			&\quad\lesssim  \Big( \norm {v(t_0)}_{ \dot{B}^{-1+ \frac 2p}_{p,1} } +   \mathcal{E}_0 \big( 1+ \log^{\frac{1}{2}}(e+t-t_0) +      \norm {\rho_0}_{L^\infty}^2 \mathcal{E}_0  ^2 \big)  \Big) \exp \left( C \norm {\rho_0}_{ L^\infty}^2 \mathcal{E}_0^2 \right)
			\\
			&\qquad +    \norm w_{L^\eta ([t_0,t]; L^m)}   \log^{\frac 12} \left( e +    \frac{   \mathcal{E}_0 \norm  { B }_{L^\infty ([t_0,t ]; \dot{H}^s)}}{  \norm w_{L^\eta ([t_0,t]; L^m) } }  \right)  \exp \left( C \norm {\rho_0}_{ L^\infty}^2\mathcal{E}_0^2 \right).
		\end{aligned}
	\end{equation*}
\end{lemma}

 \begin{proof}
As before, we assume that $t_0=0$ to simplify the presentation of the proof, which can be straightforwardly adapted to the case $t_0>0$.  We begin with rewriting the equation \eqref{v:equa} as a perturbation of the Stokes  system 
\begin{equation*}
\left\{
\begin{array}{ll} 
\partial_t v - \Delta v+ \nabla p_v =-(\rho-1)\partial_tv- \rho u\cdot \nabla v       , \\
\div v=0,\\
v|_{t=0}= u_0.
\end{array}
\right.   
\end{equation*} 
 Then, applying Lemma \ref{M:reg:DW}  yields that
 \begin{equation*} 
\begin{aligned}
  \Vert  v \Vert_{L^\infty _t  \dot{B}^{-1+ \frac 2p}_{p,1}} &+  \Vert(\partial_t v, \nabla ^2 v, \nabla p_v)\Vert_{L^{q,1}_t  L^p}  + \Vert v\Vert_{L^{\eta,1}_t L^m}
  \\
   &\lesssim \norm {v(0)}_{ \dot{B}^{-1+ \frac 2p}_{p,1} }  +    \Vert(1-\rho) \partial_t v\Vert_{L^{q,1}_t L^p} +\Vert\rho u  \cdot \nabla v\Vert_{L^{q,1}_t L^p}.
\end{aligned}
\end{equation*}   
Now, by  definition of the parameters $\eta$ and $m$, above, observe that one has, by H\"older's inequality, that
\begin{equation*}
\begin{aligned}
\Vert\rho u  \cdot \nabla v\Vert_{L^{q,1}_t L^p}&= \Vert\rho (w+v) \cdot \nabla v\Vert_{L^{q,1}_t L^p}\\ 
&\lesssim \norm {\rho }_{L^\infty_{t,x}}\Vert \nabla v\Vert_{L^{2}_t L^2}\left(  \Vert v\Vert_{ L^{\eta,1}_tL^m}  +  \Vert w\Vert_{ L^{\eta,2}_tL^m}  \right) ,
\end{aligned}    
\end{equation*}
 where we have used the embedding $L^{\eta,1}\hookrightarrow L^{\eta,2}$
  (see Lemma \ref{lemma:Lorentz}, below).  
Accordingly, by virtue of the maximum principles \eqref{mass:conservation} and \eqref{mass:conservation2}, we find that 
 \begin{equation*} 
\begin{aligned}
  \Vert  v \Vert_{L^\infty _t  \dot{B}^{-1+ \frac 2p}_{p,1}} &+  \Vert(\partial_t v, \nabla ^2 v, \nabla p_v)\Vert_{L^{q,1}_t  L^p}  + \Vert v\Vert_{L^{\eta,1}_t L^m} \\ 
  & \lesssim   \norm {v(0)}_{ \dot{B}^{-1+ \frac 2p}_{p,1} }  +    \Vert(1-\rho_0)\Vert_{L^\infty} \Vert\partial_t v\Vert_{L^{q,1}_t L^p}  \\
 & \quad + \norm {\rho_0 }_{L^\infty }\Vert \nabla v\Vert_{L^{2}_t L^2} \left(  \Vert v\Vert_{ L^{\eta,1}_tL^m}  +  \Vert w\Vert_{ L^{\eta,2}_tL^m}  \right)  . 
\end{aligned}
\end{equation*}   
Therefore,  due to  assumption  \eqref{assumption:rho2},  and  by  utilizing  the energy bound  \eqref{energy-v} on $v$, we infer that  
\begin{equation*}
\begin{aligned}
  &\Vert  v \Vert_{L^\infty _t  \dot{B}^{-1+ \frac 2p}_{p,1}} +  \Vert(\partial_t v, \nabla ^2 v, \nabla p_v)\Vert_{L^{q,1}_t  L^p}  + \Vert v\Vert_{L^{\eta,1}_t L^m} \\
  & \quad \leq  C   \norm {v(0)}_{ \dot{B}^{-1+ \frac 2p}_{p,1} }   + C\mathcal{E}_0\norm {\rho_0 }_{L^\infty_{ x}}   \Vert w\Vert_{ L^{\eta,2}_tL^m}     +C \norm {\rho_0 }_{L^\infty_{ x}}   \Vert \nabla v\Vert_{L^{2}_t L^2} \Vert v\Vert_{ L^{\eta,1}_tL^m}    ,
\end{aligned}
\end{equation*}  
for some universal constant $C>0$. 

Next, we consider a  partition of the   time interval 
$$[0,t]=\bigcup_{i=0}^{N-1}[t_i,t_{i+1}],$$ in such a way that 
\begin{equation*}
\left\{ 
\begin{aligned}
C \norm {\rho_0}_{ L^\infty}\norm{ \nabla v }_{L^2 ([t_i,t_{i+1}]; L^2)} & = \frac 12, \qquad \text{for all}\quad  i\in \{ 0,1, \dots,  N-2\} ,\\
C \norm {\rho_0}_{ L^\infty}\norm{ \nabla v }_{L^2 ([t_{N-1},t_{N}]; L^2)} & \leq  \frac 12.
\end{aligned}
\right. 
\end{equation*}
Accordingly,  up to a suitable change  in the constant $C$,  we obtain,  by  straightforward computations which are similar to the proof of Lemma \ref{lemma:w-ES}, for all $t\geq 0$,  that 
 \begin{equation} \label{AAA}
\begin{aligned}
  &\Vert  v \Vert_{L^\infty _t  \dot{B}^{-1+ \frac 2p}_{p,1}} +  \Vert(\partial_t v, \nabla ^2 v, \nabla p_v)\Vert_{L^{q,1}_t  L^p}  + \Vert v\Vert_{L^{\eta,1}_t L^m} \\
  &\quad\leq C  \Big( \norm {v(0)}_{ \dot{B}^{-1+ \frac 2p}_{p,1} } +   \mathcal{E}_0 \norm {\rho_0}_{L^\infty}   \norm {w}_{ L^{\eta,2}_tL^m} \Big) \exp \left( C \norm {\rho_0}_{ L^\infty}^2 \norm{ \nabla v }_{L^2_t L^2}^2   \right)\\
  &\quad\leq C  \Big( \norm {v(0)}_{ \dot{B}^{-1+ \frac 2p}_{p,1} } +   \mathcal{E}_0 \norm {\rho_0}_{L^\infty}   \norm {w}_{ L^{\eta,2}_tL^m} \Big) \exp \left( C \norm {\rho_0}_{ L^\infty}^2 \mathcal{E}_0^2 \right), 
\end{aligned} 
\end{equation} 
where we used the energy bound \eqref{energy-v} on $v$.
 
 Now, we need to control   the norm of $w$ on the right-hand side of the preceding estimate.  To that end, we first  employ the interpolation inequality in Lorentz spaces  \eqref{Interpolation_Ineq}   to obtain  that
\begin{equation*}
\begin{aligned}
    \norm {w}_{ L^{\eta,2}_tL^m} & \lesssim   \norm {w}_{ L^{\eta,\infty}_tL^m}  \left( 1 + \log^{\frac 12} \left( t^{ \frac 1\eta  }  \frac{ \norm {w}_{ L^{\infty}_tL^m} }{ \norm {w}_{ L^{\eta,\infty}_tL^m}}  \right ) \right).
          \end{aligned}
\end{equation*} 
Hence,  by virtue of the embedding $L^\eta_t \hookrightarrow  L^{\eta,\infty}_t  $ (see Lemma \ref{lemma:Lorentz}, below), together with the monotonicity  of the function $x\mapsto x \log^{\frac{1}{2}}( \frac{a}{x}),$ for $x\in (0,a)$, and the embedding 
$\dot{B}^1_{p,2} \hookrightarrow L^m (\mathbb{R}^2)$ (see Lemma \ref{Emb_Besov_Tr_Lemma}),
we deduce that
 \begin{equation*} 
\begin{aligned}
    \norm {w}_{ L^{\eta,2}_tL^m}       & \lesssim   \norm {w}_{ L^{\eta}_tL^m}  \left( 1 + \log^{\frac 12} \left( t^{ \frac 1\eta  }  \frac{ \norm {w}_{ L^{\infty}_t\dot{B}^1_{p,2}} }{ \norm {w}_{ L^{\eta}_tL^m}}  \right ) \right).
       \end{aligned}
\end{equation*} 
Consequently, by further incorporating the bound on $w$ from Lemma \ref{lemma:w-ES},   we end up with 
 \begin{equation*} 
\begin{aligned}
    \norm {w}_{ L^{\eta,2}_tL^m}       & \lesssim   \norm w_{L^\eta_t L^m}  \left( 1 + \log^{\frac 12} \left( t^{ \frac 1\eta  }  \frac{C \mathcal{E}_0 \norm  { B }_{L^\infty _t \dot{H}^s }   \exp \left(C  \norm {\rho_0}_{L^\infty}^4 \mathcal{E}_0 ^4 \right) }{  \norm w_{L^\eta_t L^m} }  \right ) \right)\\
     & \lesssim   \norm w_{L^\eta_t L^m}  \left( 1+ \log^{\frac{1}{2}}(e+t)   +    \norm {\rho_0}_{L^\infty}^2 \mathcal{E}_0  ^2    +   \log^{\frac 12} \left( e+    \frac{    \mathcal{E}_0\norm  { B }_{L^\infty _t \dot{H}^s }}{  \norm w_{L^\eta_t L^m} }  \right)    \right) .
       \end{aligned}
\end{equation*} 
Therefore, observing that the definition of $\eta$ and $m$, in \eqref{parameters_Relation}, implies that 
\begin{equation*} 
\frac 1\eta + \frac 1m = \frac 12,
\end{equation*}
it follows,   by   interpolation,   that 
\begin{equation}\label{L_m_Estim}
\begin{aligned}
 \norm {w}_{ L^{\eta}_tL^m} &\leq  \norm {w}_{ L^{\infty}_tL^2} +  \norm {w}_{ L^{2}_t\dot{ H}^1} \lesssim \mathcal{E}_0 ,
\end{aligned}
\end{equation}
where we have used the energy bound \eqref{energy-w} on $w$.

Finally, we deduce that
 \begin{equation*} 
\begin{aligned}
     \norm {w}_{ L^{\eta,2}_tL^m}        & \lesssim \mathcal{E}_0 \left( 1+ \log^{\frac{1}{2}}(e+t) +      \norm {\rho_0}_{L^\infty}^2 \mathcal{E}_0  ^2 \right)  + \norm w_{L^\eta_t L^m}   \log^{\frac 12} \left(e+  \frac{    \mathcal{E}_0 \norm  { B }_{L^\infty _t \dot{H}^s }}{  \norm w_{L^\eta_t L^m} }  
\right) 
  .  
       \end{aligned}
\end{equation*} 
All in all, inserting the preceding bound  in \eqref{AAA} and employing the energy inequality  \eqref{energy-v} yields 
 \begin{equation*}  
\begin{aligned}
 \Vert  v \Vert_{L^\infty _t  \dot{B}^{-1+ \frac 2p}_{p,1}} & +  \Vert(\partial_t v, \nabla ^2 v, \nabla p_v)\Vert_{L^{q,1}_t  L^p}  + \Vert v\Vert_{L^{\eta,1}_t L^m} \\
  &\lesssim  \Big( \norm {v(0)}_{ \dot{B}^{-1+ \frac 2p}_{p,1} } +   \mathcal{E}_0 \left( 1+ \log^{\frac{1}{2}}(e+t) +      \norm {\rho_0}_{L^\infty}^2 \mathcal{E}_0  ^2 \right)  \Big) \exp \left( C \norm {\rho_0}_{ L^\infty}^2 \mathcal{E}_0^2 \right)\\
   & \quad +    \norm w_{L^\eta_t L^m}   \log^{\frac 12} \left( e  +    \frac{   \mathcal{E}_0 \norm  { B }_{L^\infty _t \dot{H}^s }}{  \norm w_{L^\eta_t L^m} }  \right)  \exp \left( C \norm {\rho_0}_{ L^\infty}^2\mathcal{E}_0^2 \right),
\end{aligned}   
\end{equation*}  
thereby completing the proof of the lemma. 
\end{proof}

An interpolation argument allows us to extend the bounds on $v$ found in the preceding lemma to a wider range of functional spaces. This is content of the next result.
\begin{corollary} \label{RMK:ab}
Under the assumptions of Lemma \ref{lemma:v-ES}, it holds that
	\begin{equation*}
		\begin{aligned}
			& \Vert  v \Vert_{L^{a,1}([t_0,t]; L^b)}
			\\
			&\quad\lesssim  \Big( \norm {v(t_0)}_{ \dot{B}^{-1+ \frac 2p}_{p,1} } +   \mathcal{E}_0 \big( 1+ \log^{\frac{1}{2}}(e+t-t_0) +      \norm {\rho_0}_{L^\infty}^2 \mathcal{E}_0  ^2 \big)  \Big) \exp \left( C \norm {\rho_0}_{ L^\infty}^2 \mathcal{E}_0^2 \right)
			\\
			&\qquad +    \norm w_{L^\eta ([t_0,t]; L^m)}   \log^{\frac 12} \left( e +    \frac{   \mathcal{E}_0 \norm  { B }_{L^\infty ([t_0,t ]; \dot{H}^s)}}{  \norm w_{L^\eta ([t_0,t]; L^m) } }  \right)  \exp \left( C \norm {\rho_0}_{ L^\infty}^2\mathcal{E}_0^2 \right).
		\end{aligned}
	\end{equation*}
for any $(a,b)\in (q,\infty) \times (p,\infty)$ satisfying 
$$ \frac{1}{a} + \frac{1}{b}=\frac{1}{2}.$$
\end{corollary}

The justification of the preceding corollary can be done by a direct application of the second estimate from Lemma  \ref{M:reg:DW} combined with the bound on $v$ from  Lemma \ref{lemma:v-ES}.  

The endpoint case $b=\infty$ in the previous corollary is admissible provided that the Lorentz space in time is weakened and replaced by the standard Lebesgue space $L^2$. 
Indeed, to see that, one can use the inequality
\begin{equation*}
\norm {v}_{L^2_t L^\infty} \lesssim \norm {v}_{L^\infty_t \dot{B}^{-1+\frac 2p}_{p,1}} + \norm {\nabla ^2 v}_{L^{q}_t L^p},
\end{equation*}  
where $p$ and $q$ are related through \eqref{p-q:def}, which can be proved by means of   standard interpolation arguments.

 As a consequence of Lemmas \ref{lemma:w-ES} and \ref{lemma:v-ES}, we now establish a bound on the  $L^2_tL^\infty_x$ norm of the full velocity field $u=v+w$ which only depends logarithmically on Sobolev norms of $B$. Later, the next result will be combined with Proposition \ref{energy-Hs} to deduce the final global estimates.

\begin{lemma}[Boundedness of the velocity field $u$]\label{lemma.boundedness-u}
 Under the assumptions of Lemmas \ref{lemma:w-ES} and \ref{lemma:v-ES},  it  holds,  for any $0\leq t_0 < t$, that 
 \begin{equation*} 
\begin{aligned}
&\int_{t_0}^t  \norm {u(\tau)}^2_{  L^\infty }  d\tau
\\
&\qquad \leq C  \left( g (t_0,t) + f(t_0,t) \log  \Big( e+     \frac{   \mathcal{E}_0^2 \norm  { B }_{L^\infty ([t_0,t ]; \dot{H}^s)}^2}{ f(t_0,t) }  \Big) \right)  \exp \left( C \norm {\rho_0}_{ L^\infty}^2\mathcal{E}_0^2 \right)  ,
\end{aligned}
\end{equation*} 
where we set, for all $t\geq t_0$, that
\begin{equation*}
g (t_0,t) \bydef \norm {v(t_0)}_{ \dot{B}^{-1+ \frac 2p}_{p,1} }^2 +   \mathcal{E}_0^2 \left( 1+ \log (e+t-t_0 ) +      \norm {\rho_0}_{L^\infty}^4 \mathcal{E}_0  ^4\right),   
\end{equation*}  
and
\begin{equation}\label{f:def}
f (t_0,t) \bydef \norm w_{L^\eta ([t_0,t]; L^m)}^2 +\norm w_{L^2([t_0,t]; \dot{H}^1)}^2 .     
\end{equation}
\end{lemma}

\begin{remark}
	Note that a similar bound on $u$ in $L^2_tL^\infty_x$ has been  established  in  \cite{ag20} for the homogeneous Navier--Stokes--Maxwell equations  \eqref{NSM-equa} where different arguments have been used. Here,  in the case of the inhomogeneous Navier--Stokes--Maxwell equations \eqref{Main_System}, as we emphasized before, the adaptation of the techniques from \cite{ag20} is not possible due to the presence of a rough non-constant density $\rho$. 
\end{remark}

\begin{proof}
We recall that $u=v+w$, where $v$ and $w$ solve \eqref{v:equa} and \eqref{w:equa}, respectively.
The proof of the desired bound on $u$ will be achieved by combining the estimates from  Lemmas \ref{lemma:w-ES} and  \ref{lemma:v-ES}.

 For the $w$-part,  we first split it according to its high and low frequencies. Then, we employ Lemma \ref{Lemma_h_low_Estimate}   to obtain that 
\begin{equation*}
\begin{aligned}
 \norm {w}_{L^2 ([t_0,t];L^\infty)} ^2 
 & \lesssim \norm {S_0w}_{L^2 ([t_0,t];L^\infty)} ^2 + \norm { (\text{Id}- S_0)w}_{L^2 ([t_0,t];L^\infty)}^2 \\
  & \lesssim \norm {  w}_{L^2([t_0,t]; \dot{H}^1) }^2 \log  \left( e + \frac{  \norm { w}_{L^2([t_0,t];L^2) }^2}{ \norm {  w}_{L^2([t_0,t]; \dot{H}^1) }^2} \right)\\
  & \quad  + \norm {  w}_{L^2([t_0,t]; \dot{H}^1) }^2 \log   \left( e + \frac{ \norm w_{L^2([t_0,t]; \dot{H}^{1+s} )  }^2}{ \norm { w}_{L^2([t_0,t];\dot{H}^1 )}^2} \right) \\
  & \bydef  \mathrm{T}_1+\mathrm{T}_2,
 \end{aligned}
\end{equation*} 
where  the fact  that $s\in (0,1)$ has been used in the last inequality.

On the one hand, the energy inequality \eqref{energy-w} implies  that
\begin{equation*}
\begin{aligned}    
  \rm{T}_1 &\lesssim    \norm {  w}_{L^2([t_0,t]; \dot{H}^1) }^2 \log  \left( e + \frac{  (t-t_0) \mathcal{E}_0^2}{ \norm {  w}_{L^2([t_0,t]; \dot{H}^1) }^2} \right)\\
  &\lesssim   \mathcal{E}_0^2 \log  \left( e +   t-t_0  \right).
 \end{aligned}
\end{equation*} 
On the other hand, owing to the Sobolev embedding  (see  \eqref{Sobolev_Embedding_Lemma}, below) 
$$ \dot{W}^{2,p}(\mathbb{R}^2)\hookrightarrow \dot{H}^{1+s}(\mathbb{R}^2), \quad \text{for }\quad  p= \frac{2}{2-s},  $$
it follows,  by means of Lemma \ref{lemma:w-ES}, that 
\begin{equation*}
\begin{aligned}
\mathrm{T}_2
 & \lesssim 
 \norm {  w}_{L^2([t_0,t]; \dot{H}^1) }^2 \log   \left( e + \frac{  \norm { \nabla ^2w}_{L^2([t_0,t];L^p)}^2}{ \norm { w}_{L^2([t_0,t]; \dot{H}^1) }^2} \right) \\
  & \lesssim 
 \norm {  w}_{L^2([t_0,t];\dot{H}^1) }^2 \log   \left( e + \frac{ C \mathcal{E}_0^2 \norm  { B }_{L^\infty ([t_0,t ]; \dot{H}^s)}^2    \exp \left(C  \norm {\rho_0}_{L^\infty}^4\mathcal{E}_0   ^4 \right)
}{ \norm { w}_{L^2([t_0,t]; \dot{H}^1) }^2} \right)\\
 & \lesssim    \norm {\rho_0}_{L^\infty}^4\mathcal{E}_0   ^6 + 
 \norm {  w}_{L^2([t_0,t]; \dot{H}^1) }^2 \log   \left( e + \frac{   \mathcal{E}_0^2 \norm  { B }_{L^\infty ([t_0,t ]; \dot{H}^s)}^2      
}{ \norm { w}_{L^2([t_0,t]; \dot{H}^1) }^2} \right).
 \end{aligned}
\end{equation*}
Therefore,  gathering  the previous estimates yields 
 \begin{equation}\label{w-last1}
\begin{aligned}
 \norm {w}_{L^2 ([t_0,t];L^\infty)}  ^2
 & \lesssim    \mathcal{E}_0^2 \log  \left( e +   t-t_0  \right) +  \norm {\rho_0}_{L^\infty}^4\mathcal{E}_0   ^6 \\
 & \qquad + 
 \norm {  w}_{L^2([t_0,t]; \dot{H}^1) }^2 \log   \left( e + \frac{   \mathcal{E}_0^2 \norm  { B }_{L^\infty ([t_0,t ]; \dot{H}^s)}^2      
}{ \norm { w}_{L^2([t_0,t]; \dot{H}^1 }^2)} \right),
 \end{aligned}
\end{equation}
which takes care of the velocity component $w$.

As  for the boundedness of the $v$-part,  we first utilize interpolation and Young's inequalities to write that 
\begin{equation} \label{v_Interplation_Ine}
\|v\|_{L^2 L^\infty}\lesssim \|v\|_{L^\infty \dot{B}_{p,1}^{-1+\frac 2p}}^{\frac{2-q}{2}}\|\nabla^2 v\|_{L^{q,1}L^p}^{ \frac{q}{2}} \lesssim  \|v\|_{L^\infty \dot{B}_{p,1}^{-1+\frac 2p}} + \|\nabla^2 v\|_{L^{q,1}L^p},
\end{equation}
where we recall that $p$ and $q$ are related through \eqref{p-q:def}. 

Then, we incorporate  the estimate from   Lemma \ref{lemma:v-ES} into the inequality \eqref{v_Interplation_Ine} to find that 
\begin{equation}\label{v-last1}
\begin{aligned}
&\Vert v\Vert _{L^2 ([t_0,t];L^\infty) }^2
\\
&\lesssim  \Big( \norm {v(t_0)}_{ \dot{B}^{-1+ \frac 2p}_{p,1} }^2 +   \mathcal{E}_0^2 \left( 1+ \log (e+t-t_0) +      \norm {\rho_0}_{L^\infty}^4 \mathcal{E}_0  ^4 \right)  \Big) \exp \left( C \norm {\rho_0}_{ L^\infty}^2 \mathcal{E}_0^2 \right)\\
   &\quad +    \norm w_{L^\eta ([t_0,t]; L^m)}^2   \log  \left( e+     \frac{   \mathcal{E}_0^2 \norm  { B }_{L^\infty ([t_0,t ]; \dot{H}^s)}^2}{  \norm w_{L^\eta ([t_0,t]; L^m) }^2 }  \right)  \exp \left( C \norm {\rho_0}_{ L^\infty}^2\mathcal{E}_0^2 \right).
\end{aligned}
\end{equation}
At last, summing \eqref{w-last1} and \eqref{v-last1} completes the proof of the lemma. 
\end{proof}

 \subsection{Nonlinear energy estimate  and global bounds}

We show now how to combine Proposition \ref{energy-Hs} and Lemma \ref{lemma.boundedness-u} in order to obtain the final estimate for the velocity and electromagnetic fields in terms of the initial data.

\begin{lemma}\label{lemma:boundedness-u}
Under the assumptions of Lemmas \ref{lemma:w-ES} and \ref{lemma:v-ES}, there is a positive constant $C_*\geq 1$, depending only on the initial data $(\rho_0,u_0,E_0,B_0)$ such that  
$$   \Vert (E,B)\Vert_{L^\infty ([0,t ]; \dot{H}^s)}  +  c   \Vert E \Vert_{L^2([0,t]; \dot{H}^s )}  \leq  C_* (1+t)^{C_*} $$
and 
 \begin{equation*} 
\norm {v}_{L^\infty ([0,t];\dot B^{-1+\frac{2}{p}}_{p,1})} +    \norm {u } _{ L^2([0,t]: L^\infty )}   \leq  C_*  \log (e+t)   ,
\end{equation*} 
 for all $t\geq 0$.
 \end{lemma}
 
 \begin{proof}
 	We begin by introducing the notation
 	\begin{equation*}
 		 \begin{aligned}
A(t_0,t) & \bydef       \Vert B \Vert_{L^\infty ([t_0,t ]; \dot{H}^s)}^2+ \Vert E \Vert_{L^\infty ([t_0,t ]; \dot{H}^s)}^2  + \sigma c^2\Vert E \Vert_{L^2 ([t_0,t ]; \dot{H}^s)}^2   ,
  \\
 D(t_0, t) &\bydef   \norm {v }^2_{L^\infty ([t_0,t];  \dot B^{-1+\frac{2}{p}}_{p,1} )}  +   \norm {u }^2_{L^2([t_0,t];  L^\infty\cap\dot H^1) }   ,
\end{aligned}
 	\end{equation*}
 	with the natural extension by continuity at $t=t_0$ given by
 	\begin{equation*}
 		 \begin{aligned} 
 A(t_0, t_0) &\bydef \Vert B(t_0)\Vert_{ \dot{H}^s}^2+ \Vert E(t_0)\Vert_{ \dot{H}^s}^2, 
 \\
 D(t_0, t_0) &\bydef   \norm {v (t_0)}^2_{ \dot B^{-1+\frac{2}{p}}_{p,1} }  .  
 \end{aligned}
 	\end{equation*}
	The bounds from Proposition \ref{energy-Hs} and Lemmas \ref{lemma:v-ES} and \ref{lemma.boundedness-u} can then be recast, after minor simplifications, as 
 	\begin{equation*}
 		A(t_0,t) \leq  A(t_0,t_0) \exp \left(K_0 D(t_0,t)\right)
 	\end{equation*}
 	and
 	\begin{equation*}
 		D(t_0,t) \leq  \alpha (t)  + K_0 D(t_0,t_0) +  K_0 f(t_0,t) \log \left( e+ \frac{ K_0 A(t_0,t)}{f(t_0,t)} \right),
 		 	\end{equation*}
	where $f(t_0,t)$ is given in \eqref{f:def} and
 	\begin{equation*}
		\begin{aligned}
			K_0 & \bydef  C \exp \left( C  (1+ \norm {\rho_0}_{ L^\infty})^2\mathcal{E}_0^2 \right),
			\\
			\alpha (t) & \bydef {K_0}\log(e + t),
		\end{aligned}
 	\end{equation*}
	for some possibly large constant $C\geq 1$.
 	
 	Now, noticing that the preceding bounds hold for any $0\leq t_0 < t$ and considering a partition of the interval $[0,t]$ into
 	$$ [ 0,t] = \bigcup _{i=0}^{N-1} [t_i,t_{i+1}]$$
such that
 \begin{equation}\label{decom:!!!}
 {\rm K}_0^2f (t_i,t_{t_{i+1}}) = \frac{1}{2}  \quad  \text{ and } \quad   {\rm K}_0 ^2f (t_{N-1},t_{N}) \leq \frac{1}{2},
\end{equation}   
for any $i \in \{ 0,1,\dots N-2 \}$, we obtain the iterative system 
\begin{equation}\label{finite:sequences:0}
	\left\{ 
	\begin{aligned}
		d_{i+1}&\leq \alpha (t) + {K_0} d_i + \frac{1}{2{K_0}} \log (e+ 2{K_0^3} a_{i+1}),
		\\
		a_{i+1} &\leq a_{i} \exp({K_0} d_{i+1}),
	\end{aligned}
	\right.
\end{equation} 
for all $i\in \{0,1,..., N-1\}$, where we have denoted 
\begin{equation*}
	d_j \bydef D(t_{j-1},t_{j})  \quad \text{and} \quad   a_j \bydef A(t_{j-1},t_{j}),
\end{equation*}
for all $j\in \{1,..., N\}$, and
\begin{equation*}
	d_0 \bydef D(0,0)  \quad \text{and} \quad   a_0 \bydef A(0,0).
\end{equation*}
It is possible to show that any nonnegative families $(d_i)_{0\leq i\leq N}$ and $(a_i)_{0\leq i\leq N}$ satisfying the above recursive inequalities also enjoy the bounds
 	\begin{equation}\label{finite:sequences}
		\begin{aligned}
			d_{i} & \leq (4K_0)^{i} \left(  \alpha(t) +  d_0 +   \log (e+ 2K_0^3 a_{0}) \right),
			\\
			a_{i} & \leq  \left(  (e+ 2K_0^3 a_{0})  \exp\left( K_0 d_0\right)\exp\left( K_0 \alpha(t) \right) \right)^{(4K_0)^{i}},
		\end{aligned}
 	\end{equation}
 	for all integers $i\in [1,N]$. For the sake of clarity, we give a complete justification of \eqref{finite:sequences} in Lemma \ref{lemma:exp-vs-log}, below.
	
	Next, we deduce from \eqref{finite:sequences} that
 	\begin{equation*}
 		\begin{aligned}
 			D(0,t) 
 			& \leq \sum_{i=1}^{N} d_i
 			\leq \left(  \alpha(t) +  d_0 +   \log (e+ 2{K_0}^3 a_{0}) \right) \sum_{i=1}^{N}  (4{K_0})^{i}
 			\\
 			 & \leq 2\left(  \alpha(t) +  d_0 +   \log (e+ 2{K_0}^3 a_{0}) \right)   (4{K_0})^{N}
 		\end{aligned}
 	\end{equation*}
 	and
 	\begin{equation*}
 		\begin{aligned}
 			A(0,t) 
 			&\leq \sum_{i=1}^{N}  a_i 
 			\leq \sum_{i=1}^{N}  \left(  (e+ 2{K_0}^3 a_{0})  \exp\left( {K_0} d_0\right)\exp\left( {K_0} \alpha(t) \right) \right)^{(4{\rm K_0})^{i}}
 			\\
 			 & \leq N\left(  (e+ 2{K_0}^3 a_{0})  \exp\left( {K_0} d_0\right)\exp\left( {K_0} \alpha(t) \right) \right)^{(4{K_0})^{N}}.
 		\end{aligned}
 	\end{equation*}
 	In order to conclude, we only need now to find an upper bound on the exponent  $N$. To that end,   observe    that \eqref{decom:!!!} allows us to write that 
$$
\begin{aligned}
 \frac{N-1}{2} &= {K}_0^2 \sum_{i=0}^{ N -2} f (t_i,t_{t_{i+1}}) \\
 &= {K}_0 ^2 \left( \sum_{i=0}^{N-2} \norm {w}_{ L^\eta([t_i,t_{i+1}]; L^m)}^2\right)  + {K}_0 ^2\norm {w}_{ L^2([0, t_{N-1}], \dot{H}^1 )}^2  .
 \end{aligned}
 $$
 Thus,   by  H\"older's inequality  in  $\ell^p$ spaces, we infer that 
 $$
 \begin{aligned}
   \frac{N-1}{2}  &\leq {K}_0^2 \left( (N-1)^{1-\frac{2}{\eta}}  \left( \sum_{i=0}^{N-2}  \norm {w}_{ L^\eta([t_i,t_{i+1}]; L^m)}^\eta\right)^\frac{2}{\eta}    + \norm {w}_{ L^2([0, t_{N-1}], \dot{H}^1 )}^2 \right)\\
  &= {K}_0^2 \left( (N-1)^{1-\frac{2}{\eta}}   \norm {w}_{ L^\eta([0,t_{N-1}]; L^m)} ^2 + \norm {w}_{ L^2([0, t_{N-1}], \dot{H}^1 )}^2 \right)\\
   &\leq {K}_0^2 \left( (N-1)^{1-\frac{2}{\eta}}  C\mathcal{E}_0^2  + \mathcal{E}_0^2 \right),
\end{aligned}
 $$  
 where we employed the energy bound \eqref{L_m_Estim} in the last step.
 Therefore, assuming that $N\geq 2$, we end up with the control
 \begin{equation*}
 	N \lesssim 1+ \left( {K_0} \mathcal{E}_0\right) ^\eta.
 \end{equation*}
 All in all, incorporating this bound into the estimates on $D(0,t)$ and $A(0,t)$, above, we arrive at the estimates found in the  statement of the lemma, which completes the proof.
 	 \end{proof}
 	 
 	 We conclude this section with a justification of \eqref{finite:sequences}.
	 
 \begin{lemma}\label{lemma:exp-vs-log}
 	Let $N\in \mathbb{N}^*$ and consider two positive sequences $(d_i)_{0\leq i\leq N}$ and $(a_i)_{0\leq i\leq N}$ satisfying \eqref{finite:sequences:0},
 	for some $K_0>1$, $\alpha>0$  and all integers $i\in [1,N-1]$. Then, the bounds \eqref{finite:sequences} hold, for all integers $i\in [1,N]$.
 \end{lemma}
 
 \begin{proof}
 	By substituting the bound on $d_{i+1}$ into the one on $a_{i+1}$ in \eqref{finite:sequences:0}, we find that 
 	\begin{equation*}
 		 	a_{i+1}  \leq   a_i \exp(K_0^2d_i) \exp(K_0\alpha) (e+ 2K_0^3 a_{i+1})^{\frac{1}{2}},
 	\end{equation*}
 	for all integers $i\in [0,N-1]$. Thus, one deduces that
 	\begin{equation*}
 	\begin{aligned}
 		e+	2K_0^3a_{i+1}
 		 &\leq  e+ 2K_0^3 a_i \exp(K_0^2d_i) \exp(K_0\alpha) (e+ 2K_0^3 a_{i+1})^{\frac{1}{2}}
 		 \\
 		 & \leq (e+ 2K_0^3 a_i) \exp(K_0^2d_i)\exp(K_0\alpha) (e+ 2K_0^3 a_{i+1})^{\frac{1}{2}},
 	\end{aligned}
 		\end{equation*}
 		whereby
 	\begin{equation*}
 		e+	2K_0^3a_{i+1}  \leq  \exp(2K_0\alpha) \exp(2K_0^2d_i)(e+ 2K_0^3 a_i)^2.
 	\end{equation*}
 	Hence, we obtain that 
 	\begin{equation*}
 	A_{i+1} \bydef 	\frac{1}{K_0} \log(e+	2K_0^3a_{i+1} ) \leq 2\alpha + 2K_0d_i  + 2A_i.
 	 		 	\end{equation*}
 	Therefore, substituting this bound into the estimate on $d_{i+1}$ given in \eqref{finite:sequences:0} yields that 
 	\begin{equation*}
 		d_{i+1} \leq 2 \alpha + 2K_0d_i  + A_i.
 	\end{equation*}  
 	Now, summing the last two estimates, we deduce the weaker bound  
 	\begin{equation*}
 	\alpha + 	d_{i+1} + A_{i+1} \leq  4K_0\left(  \alpha + 	d_{i} + A_{i} \right),
 	\end{equation*}
 	which leads to the control 
 	\begin{equation*}
 		\alpha + 	d_{i+1} + A_{i+1} \leq (4K_0)^{i+1} \left(\alpha  + 	d_{0} + A_{0} \right),
 	\end{equation*}
 	for all integers $i\in [0,N-1]$. The simpler bounds \eqref{finite:sequences} then follow from a direct elementary  computation, thereby concluding the proof.
 \end{proof}

\subsection{Persistence of initial regularity of velocity field}
Notice that   we did not establish, yet, any information on the persistence of the regularity of the $w$-part in the Besov space $ \dot{B}^{s}_{p,1}$, with $s=-1+\frac{2}{p}$. This bound does not  play  a role neither in the proof of Sobolev regularity of  electromagnetic fields, nor in the previously obtained $L^2_tL^\infty$ bound on the velocity field. However, it is required to propagate  initial regularity of the full velocity field $u$. 

In the next lemma, we show, by employing  the results from the  previous sections,  that
$$w \in L^\infty_t\dot{B}^{-1+\frac{2}{p}}_{p,1},$$
whereby deducing the persistence of the initial regularity of the velocity field $u$.
 
\begin{lemma}\label{w_Estimate}
Under the assumptions of  Lemmas   \ref{lemma:w-ES} and \ref{lemma:v-ES}, it holds that
$$ w\in L^\infty _{\loc}(\mathbb{R}^+; \dot{B}^{-1+\frac{2}{p}}_{p,1} ),$$ 
uniformly with respect to the speed of light $c\in (0,\infty)$.
\end{lemma}

\begin{proof}
 We proceed as in the proof of     Lemma  \ref{lemma:v-ES}.  By applying  Lemma \ref{M:reg:DW} to the perturbed Stokes system  \eqref{w-Stokes}, we find, on any time interval $[0,t]$, that
\begin{equation*}
	\begin{aligned}
		\Vert  w \Vert_{L^\infty _t  \dot{B}^{-1+ \frac 2p}_{p,1}} +  \Vert(\partial_t w, \nabla ^2 w, \nabla p_w)\Vert_{L^{q,1}_t  L^p}
		&\lesssim   \Vert(1-\rho) \partial_t w\Vert_{L^{q,1}_t L^p}
		+\Vert\rho u  \cdot \nabla w\Vert_{L^{q,1}_t L^p}
		\\
		& \quad + \Vert j\times B\Vert_{L^{q,1}_t L^p} ,
	\end{aligned}
\end{equation*}
where the parameter $q\in (1,2)$ is introduced in the statement of Lemma \ref{lemma:v-ES}.
Therefore, by employing  H\"older's inequality, the maximum principle \eqref{mass:conservation2} and the embedding of Lorentz spaces $ L^2_{\loc}(\mathbb{R}^+) \hookrightarrow L^{q,1}_{\loc}(\mathbb{R}^+)$ (see Lemma \ref{Lemma:interpolation X}, in the appendix), we find, in view of \eqref{parameters_Relation}, that 
\begin{equation*} 
	\begin{aligned}
		& \Vert  w \Vert_{L^\infty _t  \dot{B}^{-1+ \frac 2p}_{p,1}} +  \Vert(\partial_t w, \nabla ^2 w, \nabla p_w )\Vert_{L^{q,1}_t  L^p}  \\
		&\qquad \lesssim   \Vert 1-\rho_0  \Vert_{L^\infty} \Vert \partial_t w\Vert_{L^{q,1}_t L^p} +\Vert\rho_0 \Vert_{L^\infty}   \Vert u \Vert _{L^{\eta,2}_t L^m} \Vert  \nabla w\Vert_{L^2_tL^2} + t^{\frac{1}{\eta} } \Vert j\times B\Vert_{L^{2}_t L^p}.
	\end{aligned}
\end{equation*}

Hence, by further utilizing assumption \eqref{assumption:rho2}, we arrive at the estimate 
 \begin{equation} \label{persistence:w}
\begin{aligned}
  \Vert  w \Vert_{L^\infty _t  \dot{B}^{-1+ \frac 2p}_{p,1}} &+  \Vert(\partial_t w, \nabla ^2 w)\Vert_{L^{q,1}_t  L^p}
  \\
  &\lesssim    \Vert\rho_0 \Vert_{L^\infty} \Vert u \Vert _{L^{\eta,2}_t L^m} \Vert  \nabla w\Vert_{L^2_tL^2}+ t^{\frac{1}{\eta}  } \Vert j\times B\Vert_{L^{2}_t L^p} .
\end{aligned}
\end{equation}   
The conclusion of the proof relies upon showing that the right-hand side above is finite, for any $t\geq 0$. To that end, we first write that 
$$ \begin{aligned}
 \Vert u \Vert _{L^{\eta,2}_t L^m}  & \leq  \Vert v \Vert _{L^{\eta,1}_t L^m}  +  \Vert w \Vert _{L^{\eta,1}_t L^m}  .
\end{aligned}
$$
In particular, the bound on $v$ is established in Lemma \ref{lemma:v-ES}, whereas the bound on $w$  is obtained by combining the result of Lemma \ref{lemma:w-ES} with the embedding   
\begin{equation*} 
\begin{aligned}
 \norm {w}_{L^{\eta,1}_t L^m} & \lesssim t^{\frac{1}{\eta}}\norm {w}_{L^{\infty}_t \dot{B}^0_{m,2}} \\
  & \lesssim t^{\frac{1}{\eta}}\norm {w}_{L^{\infty}_t \dot{B}^1_{p,2}},
\end{aligned}   
\end{equation*}
which follows from Lemmas \ref{Emb_Besov_Tr_Lemma} and \ref{lemma:Lorentz}, below.

As for the control of the source term $j\times B$ in $L^2_t L^p$, this   has already been  shown in the proof of Lemma \ref{lemma:w-ES}. In particular, see \eqref{w:es:AAA1}, \eqref{w:es:AAA2} and \eqref{Sobolev:embedding:AAA} from which
we deduce that
$$ \norm {j\times B}_{L^2_t L^p } \lesssim \norm {j}_{L^2_t L^2} \norm {B}_{L^\infty_t \dot{H}^s}.$$
All in all, combining the preceding bounds with the energy estimates \eqref{Energy_Identity_u} and \eqref{energy-w}, and plugging them in \eqref{persistence:w} yields that $w$ is bounded in the desired space. This completes the proof of the lemma.
\end{proof}

\subsection{Summary}

We conclude this section by collecting in one single statement all the essential bounds proven so far. This will allow for a more convenient construction of our proofs later on.

\begin{proposition}\label{Thm-summary}  Let $p\in (1,2)$,  $s_* \bydef 2 \big( 1- \tfrac{1}{p}\big) \in (0,1)$ and $s\in[s_*,1)$.  Further consider a set of parameters $( q,a,b)$   in $(1,2)\times (2,\infty)^2$ such that
 \begin{equation*} 
\frac 1p + \frac 1q = \frac 32 \qquad  \text{and}  \qquad  
  \frac 1a + \frac 1b  = \frac 12 .
\end{equation*}     
 Let $\rho_0$ be a bounded function satisfying assumptions \eqref{Assumption A} and  \eqref{assumption:rho2} for some given constants $\underline{\rho} , \overline{\rho} \in (0,\infty) $.
 Let $(\rho_0,u_0,E_0,B_0)$ be initial data for \eqref{Main_System} where   $u_0  $ is divergence free and
   $$ u_0 \in \dot{B}^{-1+\frac{2}{p}}_{p,1}(\mathbb{R}^2),\quad  \quad (E_0,B_0)\in H^s(\mathbb{R}^2).$$ 
  
Then, any  regular solution $(\rho,u,E,B)$ associated to the preceding initial data enjoys the   bounds 
$$ \rho \in L^\infty(\mathbb{R}^+; L^\infty  (\mathbb{R}^2)), \quad  (u,E,B)\in L^\infty (\mathbb{R}^+; L^2  (\mathbb{R}^2)), \quad (\nabla u, j) \in L^2 (\mathbb{R}^+; L^2  (\mathbb{R}^2)), $$
$$   u \in    L^2 _{\loc}(\mathbb{R}^+; L^\infty (\mathbb{R}^2)) \cap  L^{q,1} _{\loc}(\mathbb{R}^+;\dot{W}^{2,p} (\mathbb{R}^2)),$$
$$     (E,B)\in L^\infty_{\loc}( \mathbb{R}^+;  \dot{H}^{s}(\mathbb{R}^2)) ,\qquad cE \in L^2_{\loc}( \mathbb{R}^+;  \dot{H}^{s}(\mathbb{R}^2)).$$  

In addition to that, the velocity field is split as $u=v+w$, where $v$ and $w$ are governed by the equations \eqref{v:equa} and \eqref{w:equa}, respectively. Together with their associated pressures $p_v$ and $p_w$, they enjoy the bounds
$$( v, w) \in L^\infty _{\loc}(\mathbb{R}^+; \dot{B}^{-1+\frac{2}{p}}_{p,1} (\mathbb{R}^2) )\cap L^{a,1} _{\loc}(\mathbb{R}^+;L^b (\mathbb{R}^2)) , \qquad  w\in  L^\infty _{\loc}(\mathbb{R}^+; \dot{B}^{1}_{p,2} (\mathbb{R}^2) ) , $$
$$(\partial_t v, \nabla ^2v, \nabla p_v ) \in  L^{q,1} _{\loc}(\mathbb{R}^+;L^p (\mathbb{R}^2)),  \qquad  (\partial_tw, \nabla ^2w, \nabla p_w )\in  L^{2} _{\loc}(\mathbb{R}^+;L^p(\mathbb{R}^2) ).$$ 
All the above bounds are uniform with respect to the speed of light $c\in (0,\infty)$.
\end{proposition}

We stress that, except for the control
$$w\in L^{a,1} _{\loc}(\mathbb{R}^+;L^b (\mathbb{R}^2)),$$
all the bounds  stated in  the  proposition above have been previously established in \eqref{Ident_Energy_1}, \eqref{mass:conservation}, \eqref{mass:conservation2}, Lemmas \ref{lemma:w-ES} and \ref{lemma:v-ES}, Corollary \ref{RMK:ab} (with \eqref{L_m_Estim}), and Lemma \ref{lemma:boundedness-u}. Moreover, the propagation of the regularity of $E$ and $B$ in $H^s$ is a consequence of Proposition \ref{energy-Hs}. As for the remaining bound on $w$, it is similar to Corollary \ref{RMK:ab} and follows from the combination of existing bounds on $w$ with the interpolation estimate \eqref{v_L_r_s_Estimate} from Lemma \ref{M:reg:DW} in the appendix.

\subsection{Proof of existence of solutions}
\label{proof_existence}

For clarity, we outline here the proof of the first part of Theorem \ref{Thm:1}, that is, the proof of existence of solutions in the non-Lipschitz setting. To that end, recall that we are considering a sequence of strong solutions $ (\rho_n,u_n,E_n,B_n)$ to the approximate system \eqref{AP-system} which enjoy all a priori bounds collected in Proposition \ref{Thm-summary}, uniformly in $n$.

In view of these bounds, one can then show, with a classical application of the Aubin--Lions compactness lemma, that $(u_n,E_n,B_n)_{n\in\mathbb{N}}$ is relatively compact in the strong topology of $L^2_\mathrm{loc}(dtdxdv)$. Note, though, that $(\rho_n)_{n\in\mathbb{N}}$ might not enjoy strong compactness properties. However, up to extraction of a subsequence, one can assume that it converges in the weak* topology of $L^\infty$. This is then largely sufficient to ensure the weak stability of all nonlinear terms in \eqref{AP-system}, which implies that any limit point $(\rho,u,E,B)$ is a weak solution of \eqref{Main_System}.

Let us point out that the solution $(\rho,u,E,B)$ also enjoys the bounds stated in Proposition \ref{Thm-summary}. This is a consequence of the fact that the spaces involved in these bounds all satisfy the Fatou property (see \cite[Theorems 2.25 and 2.72]{bcd11} for the Fatou property in Besov spaces  and  Theorem \cite[Theorem 1.4.11]{Grafakos} for the same property in Lorentz spaces).

\section{Time-weighted estimates (I) and Lipschitz bound on velocity field}
\label{Section.TW1}

In this section, we show how to   establish the Lipschitz regularity of the velocity field $u$ stated in Theorem \ref{Thm:1} with $s\in(\frac 12, 1)$, as a consequence of Proposition \ref{Thm-summary}. This regularity is crucial and will be employed afterwards in the proof of  uniqueness.

 Again, the decomposition $u=v+w$, where $v$ and $w$ are the respective solutions of \eqref{v:equa} and \eqref{w:equa}, will be utilized in this section.

\subsection{Preliminary time-weighted estimates and higher regularities}

Before we proceed with the derivation of the Lipschitz bound   on $v$, we need to prove important  time-weighted estimates. This is done in the spirit of \cite{Danchin_Wang_22} and constitutes a crucial step. It has also been previously exploited in \cite{LI20176512, PZ13} in the context of the inhomogeneous Navier--Stokes equations.

\begin{lemma}[Time-weighted estimate for $v$] \label{lemma:v2}
	Under the assumptions stated in Proposition \ref{Thm-summary},  it  holds, uniformly with respect to the speed of light $c>0$, that
	$$ tv \in  L^\infty_{\loc}(\mathbb{R}^+; \dot{B}^{1 + \frac{2}{m}}_{m,1}(\mathbb{R}^2) ),
	\quad
	\big( \partial _t (tv), \nabla ^2 (tv) , \nabla (tp_v)  \big) \in L^{\eta,1}_{\loc}( \mathbb{R}^+; L^m(\mathbb{R}^2)),$$ 
	where  $\eta$ and $m$ are defined by \eqref{parameters_Relation}.   
\end{lemma}

\begin{proof}
 First of all, we multiply the   equation defining $v$, that is \eqref{v:equa}, by $t$ and rewrite it as a perturbation of the Stokes system 
 \begin{equation*}   
 \partial_t (tv ) - \Delta (tv)  + \nabla (tp_v)  = \rho v   + (1-\rho) \partial_t (tv)  -\rho (tu)\cdot \nabla v    . 
\end{equation*} 
Then, recalling the identity
\begin{equation}\label{identity:parameters:A}
	\frac 1\eta + \frac 1m = \frac 12,
\end{equation}
and applying Lemma \ref{M:reg:DW}, followed by H\"older's inequalities and    \eqref{mass:conservation}-\eqref{mass:conservation2}, yields
\begin{equation*}
\begin{aligned}
&\norm{tv}_{L^\infty_t\dot{B}^{2-\frac {2}{\eta}}_{m,1}}
+\norm{\big(\partial _t (tv), \nabla ^2 (tv), \nabla (tp_v)\big)}_{L^{\eta,1}_tL^m}  \\
&\qquad\lesssim  \norm {1-\rho}_{L^\infty_{t,x}}\norm{\partial _t (tv) }_{L^{\eta,1}_tL^m} + \norm {\rho}_{L^\infty_{t,x}} \left( \norm { v }_{L^{\eta,1}_tL^m}  + \norm { tu\cdot \nabla v }_{L^{\eta,1}_tL^m}  \right) \\
&\qquad\lesssim  \norm {1-\rho_0}_{L^\infty}\norm{\partial _t (tv) }_{L^{\eta,1}_tL^m}   + \norm {\rho_0}_{L^\infty} \left( \norm { v }_{L^{\eta,1}_tL^m}  + \norm {  u } _{L^{\eta,1}_tL^m}  \norm {t\nabla v }_{L^\infty_{t,x}}  \right).
\end{aligned} 
\end{equation*}
Next, by virtue of   the embedding 
\begin{equation*}
\dot{B}^{2-\frac 2\eta  }_{m,1} = \dot{B}^{1+ \frac 2m}_{m,1}  \hookrightarrow \dot{W}^{1,\infty}(\mathbb{R}^2),
\end{equation*}
together with assumption  \eqref{assumption:rho2},
we find that
\begin{equation*}
\begin{aligned}
\norm{tv}_{L^\infty_t\dot{B}^{1 + \frac{2}{m}}_{m,1}} & +  \norm{\big(\partial _t (tv), \nabla ^2 (tv),  \nabla (tp_v)\big)}_{L^{\eta,1}_tL^m}  \\
&\leq   C \norm {\rho_0}_{L^\infty} \left( \norm { v }_{L^{\eta,1}_tL^m}  + \norm {  u } _{L^{\eta,1}_tL^m}  \norm {t v }_{L^\infty_t\dot{B}^{1 + \frac{2}{m}}_{m,1}}  \right), 
\end{aligned} 
\end{equation*} 
for some  universal constant $C>0 $. Note that, in view of \eqref{identity:parameters:A}, Proposition \ref{Thm-summary} ensures the bound  
$$  u, v \in  L^{\eta,1}_{\loc}(\mathbb{R}^+; L^m(\mathbb{R}^2)).$$      
Thus, in order to conclude, we proceed as in the proofs of Lemmas \ref{lemma:w-ES} and \ref{lemma:v-ES}. More precisely, we consider a finite subdivision of time  
$$ [0,t]= \bigcup_{i=0}^{N-1} [t_{i}, t_{i+1}]$$
in such a way that 
\begin{equation*}
\left\{ 
\begin{aligned}
C \norm {\rho_0}_{L^\infty}\norm {u}_{L^{\eta,1}([t_{i}, t_{i+1}]; L^m)}  &= \frac 12, \qquad \text{for all}\quad  i\in \{ 0,1, \dots,  N-2\} ,\\
C \norm {\rho_0}_{L^\infty}\norm {u}_{L^{\eta,1}([t_{N-1}, t_{N}]; L^m)} &\leq  \frac 12.
\end{aligned}
\right. 
\end{equation*}
Accordingly, by induction and by following the method of proof\footnote{Here, note that the decomposition in time happens in Lorentz spaces and not in standard Lebesgue  spaces  as is the case in the proofs of Lemmas \ref{lemma:w-ES} and \ref{lemma:v-ES}. Nevertheless, the same procedure applies in this situation. This is done with complete details in the proof of Proposition 2.2 in \cite{Danchin_Wang_22}.} of Lemmas \ref{lemma:w-ES} and \ref{lemma:v-ES}, we arrive at the desired bounds.  The proof of the lemma is therefore completed.
  \end{proof}

  The following lemma provides us with improved bounds on the $w$-part as soon as the regularity parameter $s$ is large enough.  This will be employed to prove the Lipshitz bound on the full velocity field $u$.

  \begin{lemma}[Higher regularities for $w$] \label{lemma:w2}
  Under the assumptions of Proposition \ref{Thm-summary}, we further  assume  that $ s\in (\frac{1}{2},1)$. Moreover, we define
 $$k\bydef \frac {1}{1-s} \in (2,\infty)$$   
 and we fix two parameters 
 $$2<\sigma   <k \quad \text{and} \quad  2< r< \tfrac{1}{1-(\tfrac{1}{\sigma}-\tfrac{1}{k})} . $$
   Then, it holds that 
   $$w\in  L^\infty_{\loc}(\mathbb{R}^+; \dot{B}^{2-\frac2r}_{\sigma,r}(\mathbb{R}^2)), \qquad \partial_t w, \nabla ^2w, \nabla p_w\in  L^{r}_{\loc}(\mathbb{R}^+;L^{\sigma}(\mathbb{R}^2)),$$
where the bounds may not be uniform with respect to the speed of light $c\in (0,\infty)$.
  \end{lemma}

\begin{proof}
   We begin with applying  Lemma \ref{MR:lemma:1} to  \eqref{w:equa} to obtain, for any $t>0$, that
\begin{equation}\label{Energy_Estimate_u_e}
	\begin{aligned}
		\norm{ w}_{ L^\infty_t \dot{B}^{2-\frac2r}_{\sigma,r}}
		& +  \norm{\big(\partial_t w, \nabla^2 w, \nabla p_w\big)}_{ L^{r}_t L^\sigma }
		\\
		& \lesssim   \norm{1-\rho}_{ L^\infty_{t,x}}\norm{\partial_t w}_{ L^{r}_t L^\sigma }
		+ \norm {\rho}_{L^\infty_{t,x}} \norm { u\cdot\nabla w}_{L^{r}_t L^\sigma }+ \norm {j\times B}_{L^{r}_t L^\sigma }.
	\end{aligned}
\end{equation}
Therefore, introducing the parameters 
\begin{equation*}
	\frac{1}{b} \bydef \frac{1}{\sigma} - \frac{1}{k} \in (0,\tfrac{1}{2}), \qquad \frac{1}{a} \bydef \frac{1}{2} - \frac{1}{b} \in (0,\tfrac{1}{2})
\end{equation*} 
and 
\begin{equation*}
	\frac{1}{\ell} \bydef \frac{1}{r} - \frac{1}{a} = \frac{1}{r} - \left( \frac{1}{2} - \left (\frac{1}{r} - \frac{1}{k}\right)\right) \in (0,\tfrac{1}{r}),
\end{equation*}
we obtain, by applying H\"older's inequality, that
\begin{equation*} 
	\begin{aligned}
		\norm{ w}_{ L^\infty_t \dot{B}^{2-\frac2r}_{\sigma,r}}
		& +  \norm{\big(\partial_t w, \nabla^2 w, \nabla p_w\big)}_{ L^{r}_t L^\sigma }  
		\lesssim     \norm { u }_{L^{a}_t L^b }\norm { \nabla w}_{L^{\ell }_t L^ k }+  \norm {j\times B}_{L^{r}_t L^\sigma },
	\end{aligned}
\end{equation*}
where we have employed \eqref{mass:conservation} to estimate the second term in the right-hand side of \eqref{Energy_Estimate_u_e}, and \eqref{mass:conservation2} and \eqref{Assumption A} to absorb the first term in the right-hand side.

 By further expanding $j$ using Ohm's law \eqref{Ohms-law1} or \eqref{Ohms-law2} and applying H\"older's inequality, again, we infer that
\begin{equation*} 
	\begin{aligned}
		\norm{ w}_{ L^\infty_t \dot{B}^{2-\frac2r}_{\sigma,r}}
		 +  \norm{\big(\partial_t w, \nabla^2 w, \nabla p_w\big)}_{ L^{r}_t L^\sigma }  
		&\lesssim   
		  \norm { u }_{L^{a}_t L^b }\norm { \nabla w}_{L^{\ell }_t L^ k }
		+ c  t^{\frac{1}{r}}  \norm { (E,B) }_{L^{\infty }_t L^{2\sigma} }^2
		\\
		&\quad +  t^\frac{1}{\ell }\norm { u }_{L^{a}_t L^b }\norm { B}_{L^{\infty }_t L^ {2k} }^2.
	\end{aligned}
\end{equation*}
Now, employing the embedding 
\begin{equation*}
	\dot H^s \hookrightarrow L^{2k}(\mathbb{R}^2),
\end{equation*}
in conjunction with the interpolation inequalities
\begin{equation*}
	\norm { (E,B) }_{L^{\infty }_t L^{2\sigma} } \lesssim \norm { (E,B) }_{L^{\infty }_t ( L^{2} \cap L^{2k}) } \lesssim \norm { (E,B) }_{L^{\infty }_t H^s }
\end{equation*}
and 
\begin{equation*}  
 \begin{aligned}
 	  \norm {\nabla w}_{L^\ell_t L^k}
	\lesssim \norm {\nabla w}_{L^\ell_t \dot B^{ \frac{2}{\sigma}-  \frac{2}{k}}_{\sigma,1}}
	&= \norm {\nabla w}_{L^\ell_t \dot B^{ -1 +\frac{2}{r}+  \frac{2}{\ell}}_{\sigma,1}}
	\\
	&\lesssim \norm { \nabla  w }_{L^\infty_t \dot{B}^{1-\frac2r}_{\sigma,r} \cap L^r _t \dot W^{1,\sigma} }
	= \norm {   w }_{L^\infty_t \dot{B}^{2-\frac2r}_{\sigma,r} \cap L^r _t \dot W^{2,\sigma} },
 \end{aligned}
 \end{equation*}
we end up with the functional inequality 
\begin{equation}\label{h:bound:1}
	h(t) \bydef \norm{ w}_{ L^\infty_t \dot{B}^{2-\frac2r}_{\sigma,r}}
		 +  \norm{\big(\partial_t w, \nabla^2 w, \nabla p_w\big)}_{ L^{r}_t L^\sigma }  \lesssim H(0)+ F_c(t) + U(t) h(t),
\end{equation}
where 
\begin{equation*}
	U(t) \bydef \norm {u}_{L^a_t L^b}
\end{equation*}
and
\begin{equation*}
	 F_c(t) \bydef  \left (t^\frac{1}{\ell }\norm { u }_{L^{a}_t L^b } + ct^{\frac{1}{r}} \right)\norm { (E,B)}_{L^{\infty }_t H^s }^2.
\end{equation*}
Note that $F_c(t)$ is a locally-bounded function of time, thanks to Proposition \ref{Thm-summary}. Accordingly, by virtue of the fact that $a<\infty$, one can repeat the bootstrap argument presented in the proof of Lemma \ref{lemma:v2} to deduce the desired bound on $w$. More precisely, one can split the time interval $[0,t]$ into a union of a finite  number of intervals such that, on each subinterval, the coefficient $U(t)$ is small enough to absorb the term $h(t)$ in the right-hand side of \eqref{h:bound:1} by the left-hand side. This completes the proof of the lemma.
  \end{proof}
 
\begin{remark}\label{RMK-S1}
	A straightforward combination of the bounds from Proposition \ref{Thm-summary} and Lemma \ref{lemma:w2}  yields the additional controls
	\begin{equation}
		\label{partial_w}
		w \in  L^\infty _{\loc}(\mathbb{R}^+;\dot{H}^1(\mathbb{R}^2)), \qquad  \partial_t w, \nabla ^2w, \nabla p_w\in  L^2_{\loc}(\mathbb{R}^+;L^2(\mathbb{R}^2)).
	\end{equation}
	Indeed, since $p<2<\sigma$, we see that
	$$\begin{aligned}
	\norm {w(t)}_{\dot{H}^1}^2 = \sum_{ j\in \mathbb{Z} } \norm {\nabla \Delta_j w}_{L^2} ^2
	&\lesssim \sum_{ j\in \mathbb{Z} } \norm {\nabla \Delta_j w}_{L^p} ^2  +  \sum_{ j\in \mathbb{Z} } \norm {\nabla \Delta_j w}_{L^\sigma} ^2  \\
	&=   \norm {w(t)}_{\dot{B}^1_{p,2}}^2 +  \norm {w(t)}_{\dot{B}^1_{\sigma,2}}^2   ,   
	\end{aligned} $$
	for any $t\geq 0$. Accordingly, the $ L^\infty_t \dot{H}^1$ bound on $w$ follows  by further employing the interpolation inequality    
	 \begin{equation*} 
  \begin{aligned}  \norm {\nabla w}_{L^\infty_t  \dot{B}^0_{\sigma ,2}} 
  & \lesssim  \norm {\nabla w}_{L^\infty_t \dot{B}^{1-\frac 2r}_{\sigma ,\infty} }^{ \alpha}\norm { \nabla w}_{L^\infty_t \dot{B}^{-2+ \frac 2\sigma }_{\sigma ,\infty } }^{ 1-\alpha}\\
   & \lesssim  \norm { w}_{L^\infty_t \dot{B}^{2-\frac 2r}_{\sigma ,r} }^{ \alpha}\norm { w}_{L^\infty_t L^2}^{ 1- \alpha},
  \end{aligned}
  \end{equation*}
  for some $\alpha \in (0,1)$, which holds due to the fact that   $ -2+ \tfrac{2}{\sigma}<0< 1 -\frac{2}{r}$.
  
	Likewise, the $L^2_{t,x}$  bound  on $\partial_t w$, $\nabla ^2 w$ and $\nabla p_w$ is obtained as a consequence of the embedding $L^2_tL^p\cap L^r_t L^\sigma \hookrightarrow L^2_tL^2$, which holds on any bounded interval of time since $r>2$ and $p<2<\sigma$.

	Furthermore, we   deduce  that     $\dot{w}\bydef \partial_t w + u \cdot \nabla w$ enjoys the     bound 
	$$   \dot{w} \in  L^2_{\loc}(\mathbb{R}^+;L^2(\mathbb{R}^2)). $$    
	This follows by writing that 
	$$\norm {\dot{w}}_{L^2_t L^2}  \lesssim
	\norm { \partial_t w}_{L^2_t L^2}
	+ \norm { u}_{L^2_t L^ \infty} \norm { \nabla w}_{L^\infty _t L^2} ,$$
	where the right-hand side   is finite due to the preceding bounds together with the $L^2_tL^\infty$ control on $u$ from Proposition \ref{Thm-summary}.
\end{remark}

  \subsection{Lipschitz regularity of the velocity field and summary}

We are now in a position to establish the Lipschitz regularity of the velocity field.  The Lipschitz bound on $v$  is obtained by employing the time-weighted estimates from Lemma \ref{lemma:v2}, whereas the corresponding bound on $w$ follows from Lemma \ref{lemma:w2} and includes a slight improvement in 
terms  of time integrability.

  Unlike the case of the inhomogeneous Navier--Stokes equations which is studied in \cite{Danchin_Wang_22},  we stress that it is important here to establish a refined Lipschitz bound in terms of Besov spaces. In particular, this refinement will be crucial in the stability analysis of \eqref{Main_System}, as we shall see later on in the proof of  Theorem \ref{Thm:stability}.

\begin{proposition}[Lipschitz bound and summary]\label{corollary:u-Lip}
	Let the assumptions of Proposition \ref{Thm-summary} be   fulfilled and further assume      that the regularity index $s$ therein belongs to $(\frac{1}{2},1)$. Then,  the solutions of  \eqref{v:equa} and \eqref{w:equa} enjoy the bounds 
	$$ v \in L^{1}_{\loc}( \mathbb{R}^+;\dot{B}^{1+ \frac{2}{m}}_{m,1} (\mathbb{R}^2)),
	\qquad  w \in L^{\infty}_{\loc}( \mathbb{R}^+;\dot{B}^{  \frac{2}{m}}_{m,1} (\mathbb{R}^2)) \cap  L^{2}_{\loc}( \mathbb{R}^+;\dot{B}^{1+ \frac{2}{m}}_{m,1} (\mathbb{R}^2)).$$  
	In addition, the field $v$ also satisfies the   time-weighted estimates
	$$ t^\frac{1}{2} v \in L^{\infty}_{\loc}( \mathbb{R}^+;\dot{B}^{  \frac{2}{m}}_{m,1} (\mathbb{R}^2)) \cap  L^{2}_{\loc}( \mathbb{R}^+;\dot{B}^{1+ \frac{2}{m}}_{m,1} (\mathbb{R}^2)),$$
	and    
	$$  t  v \in L^{\infty}_{\loc}( \mathbb{R}^+;\dot{B}^{ 1+  \frac{2}{m}}_{m,1} (\mathbb{R}^2)) \cap  L^{ \eta,1 }_{\loc}( \mathbb{R}^+;\dot{W}^{2, m} (\mathbb{R}^2)),$$
	uniformly with respect to the speed of light $c\in (0,\infty)$.
\end{proposition} 
  \begin{remark}\label{RMK-corollary1}
   Due to the embedding 
   $$\dot{B}^{ \frac{2}{m}}_{m,1}\hookrightarrow L^\infty(\R^2),$$ 
   it is readily seen that the bounds in the preceding proposition entail that
   \begin{equation}\label{v:TW1}
	   v\in L^1_\mathrm{loc}(\mathbb{R}^+; \dot W^{1,\infty}(\mathbb{R}^2)),
	   \qquad
   t^\frac{1}{2} v\in   L^\infty _{\loc}(\mathbb{R}^+; L^\infty(\R^2))\cap L^2_{\loc}(\mathbb{R}^+; \dot{W}^{1,\infty}(\R^2)),
   \end{equation}
   and 
   \begin{equation*}
   w\in   L^\infty _{\loc}(\mathbb{R}^+; L^\infty(\R^2))\cap L^2_{\loc}(\mathbb{R}^+; \dot{W}^{1,\infty}(\R^2)).
   \end{equation*}
   Moreover, recalling that $u=v+w$, it holds that
   \begin{equation}\label{u:TW1}
t^\frac{1}{2} u\in   L^\infty _{\loc} (\mathbb{R}^+; L^\infty(\R^2) )\cap L^2_{\loc}(\mathbb{R}^+; \dot{W}^{1,\infty}(\R^2)).
\end{equation}
These bounds will come in handy, later on.
  \end{remark}

\begin{remark}
	In the proposition, above, all bounds on $v$ are uniform with respect to the speed of light. However, for technical reasons, the control of $w$ may grow as $c\to \infty$.
\end{remark}

  \begin{proof}[Proof of Proposition \ref{corollary:u-Lip} ]  Recalling that 
   $$ \frac{1}{p}= \frac{1}{m}+ \frac{1}{2},$$ 
  we employ the interpolation inequality \eqref{lemma:interpolation}   to write that 
  \begin{equation*} 
  \begin{aligned}
  \int_0^t \| v (\tau)\|_{\dot{B}^{1+ \frac{2}{m}}_{m,1}} d\tau & \lesssim   \int_0^t  \tau^{-\frac 2m} \norm{\nabla  ^2v (\tau)}_{L^p} ^{ 1-\frac 2m} \norm{\tau \nabla  ^2v (\tau)}_{L^m} ^{\frac 2m} d\tau  .
  \end{aligned}
  \end{equation*}
  Therefore, utilizing H\"older's inequality in Lorentz spaces (see Lemma \ref{lemma:Lorentz} in the appendix), it   follows that 
    \begin{equation*} 
  \begin{aligned}
  \int_0^t \| v (\tau)\|_{\dot{B}^{1+ \frac{2}{m}}_{m,1}} d\tau  & \lesssim   \norm{ \tau^{-\frac 2m} }_{L^{\frac m2,\infty}(\mathbb{R}^+)}\norm{\nabla  ^2v }_{L^{q,1}_tL^p} ^{ 1-\frac 2m} \norm{\tau \nabla  ^2v  }_{L^{\eta,1}_tL^m} ^{\frac 2m}\\
  & \lesssim   \norm{\nabla  ^2v }_{L^{q,1}_tL^p} +  \norm{\tau \nabla  ^2v  }_{L^{\eta,1}_tL^m}  .
  \end{aligned}
  \end{equation*} 
 Now, by virtue of Proposition \ref{Thm-summary} and Lemma \ref{lemma:v2}, it is readily seen that the right-hand side above is finite for any $t\in \mathbb{R}^+$,  thereby achieving the Lipschitz bound on $v$. 
 
As for $w$, we  are going to prove the slightly stronger bound 
\begin{equation}\label{w_bound_Improved}
w \in L^{\infty}_{\loc}( \mathbb{R}^+;\dot{B}^{  \frac{2}{\sigma}}_{\sigma,1} (\mathbb{R}^2)) \cap  L^{2}_{\loc}( \mathbb{R}^+;\dot{B}^{1+ \frac{2}{\sigma}}_{\sigma,1} (\mathbb{R}^2)),
\end{equation}
where $\sigma $ is introduced in Lemma \ref{lemma:w2}, which implies the desired bound by embedding and the fact that $\sigma<m$.

To that end, we first employ the embedding 
$$L^\infty_{\loc}(\mathbb{R}^+; \dot{B}^{1}_{p,2}(\mathbb{R}^2)) \hookrightarrow L^\infty_{\loc}(\mathbb{R}^+; \dot{B}^{1 - 2(\frac 1p-\frac 1\sigma )}_{\sigma,2}(\mathbb{R}^2)) , $$
together with the interpolation inequality   \eqref{interpolation.AP-B}, to deduce that
\begin{equation*} 
	\norm w _{ L^\infty_t\dot B^{ \frac{2}{\sigma}}_{\sigma,1}}
	\lesssim
	\norm w_{L^\infty_t\dot B^{1   }_{p,2}}^{  \frac{1- \frac{2}{\sigma} }{ \frac{1}{2}- \frac{2}{\sigma}  + \frac{1}{p} }}
	\norm w_{L^\infty_t\dot B^{ 2-  \frac{2}{\sigma}}_{\sigma,2}}^{  \frac{\frac{1}{p}- \frac{1}{2} }{ \frac{1}{2}- \frac{2}{\sigma}  + \frac{1}{p} }}, 
\end{equation*}
where we used the fact that $p<2<\sigma$. Since, by Proposition \ref{Thm-summary} and Lemma \ref{lemma:w2}, the right-hand side above is finite, this implies the bound $w \in L^\infty_t\dot B^{ \frac{2}{\sigma}}_{\sigma,1}$.

Likewise, by employing the interpolation inequality \eqref{lemma:interpolation}, we see that
\begin{equation*} 
\norm {\nabla w} _{L^2_t \dot{B}^{\frac{2}{\sigma}} _{k,1}} \lesssim \norm {\nabla^2 w}_{L^2_tL^p}^{ \frac{\frac{1}{2}- \frac{1}{\sigma}}{\frac{1}{p}- \frac{1}{\sigma}}} \norm {\nabla^2 w}_{L^2_tL^\sigma}^{   \frac{\frac{1}{p}-\frac{1}{2}} {\frac{1}{p}- \frac{1}{\sigma}}}, 
\end{equation*}
for $ p<2<\sigma$, which allows us to further infer the remaining bound on $w$. This completes the proof of \eqref{w_bound_Improved}.

 We are now left with the time-weighted estimates on $v$. Observe, though, that the bound on $tv$ has been previously established in Lemma \ref{lemma:v2}. As for the bound on $t^\frac 12 v$, it is obtained by interpolation methods, as before. More precisely, the interpolation inequality 
\begin{equation*}
\tau^\frac{1}{2} \norm{ v(\tau)}_{\dot B^\frac 2m_{m,1}}
\lesssim
\norm { v(\tau)}_{ \dot B^{-1+\frac 2m}_{m,\infty} }^{\frac{1}{2}}
\norm { \tau v(\tau)}_{ \dot B^{1+\frac 2m}_{m,1}}^{\frac{1}{2}} 
\lesssim
\norm { v(\tau)}_{ L^2}^{ \frac{1}{2}} \norm { \tau v(\tau)}_{ \dot{B}^{ 1+ \frac{2}{m}}_{ m,1 }}^{\frac{1}{2}} ,
\end{equation*}
where   the right-hand side above belongs to $L^\infty_{\loc}(\mathbb{R}^+)$ due to the energy inequality \eqref{energy-v} and Lemma \ref{lemma:v2}, allows us to deduce the first bound on $t^\frac 12 v$. Then, writing
\begin{equation*}
\int_0^t\tau
\norm { v(\tau)}_{ \dot{B}^{ 1+ \frac{2}{m}}_{ m,1 }}^2  d\tau
\lesssim \sup_{\tau\in [0,t]} \norm { \tau v(\tau)}_{ \dot{B}^{ 1+ \frac{2}{m}}_{ m,1 }} \int_0^t  \norm {   v(\tau)}_{ \dot{B}^{ 1+ \frac{2}{m}}_{ m,1 }}  d\tau  ,
\end{equation*}
it is readily seen, by Lemma \ref{lemma:v2} and the previous control of $v$ in $L^1_t\dot{B}^{ 1+ \frac{2}{m}}_{ m,1 }$, that the right-hand side above is  finite,  for any $t>0$. This concludes the proof of the proposition. 
  \end{proof}

  \section{Stability and uniqueness}\label{Section.stability}

  In this section, we prove a general stability result for   the   inhomogeneous Navier--Stokes--Maxwell system   \eqref{Main_System}. This will then  be employed  to deduce   the uniqueness of the solution constructed in the previous section. Our main stability result  reads as follows.

\begin{theorem}\label{Thm:stability}
Let $ (\rho_i,u_i,E_i,B_i)_{ i\in \{ 1,2\}}$ be a set of  two energy  solutions to \eqref{Main_System}. Assume that one of the two solutions enjoys some  additional regularity    in the sense that 
 $$u_2 \in L^1_{\loc}( \mathbb{R}^+; B^{1+\frac{2}{r}}_{r,1}) \cap L^2_{\loc}( \mathbb{R}^+; L^\infty), \quad j_2 \bydef \sigma (cE_2 + u_2 \times B_2) \in L^ \frac{a}{a-1} _{\loc}( \mathbb{R}^+; L^a  ),$$
$$ t^{\frac{1}{2} } \dot{w} _2 \in  L^2_\loc(\mathbb{R}^+; H^1), \qquad t  \dot{v} _2 \in L^2_\loc(\mathbb{R}^+; L^ \infty) \cap L^q_\loc(\mathbb{R}^+; \dot{W}^{2,p}),    $$
for some $ r, a\in (2,\infty)$, $s \in (\frac{1}{2},1)$ and   $p,q\in (1,2)$ with 
$$ p = \frac{2}{2-s} \qquad \text{and} \qquad\frac{1}{p}+ \frac{1}{q}=\frac{3}{2}, $$
where we denote $$\dot{w}_2 \bydef \partial_t w_2 + u_2\cdot \nabla  w_2 , \qquad \dot{v}_2 \bydef \partial_t v_2 + u_2\cdot  \nabla   v_2,$$ and $u_2=v_2 + w_2$ is any decomposition of the velocity field $u_2.$ Then, it holds that $$ (\rho_1,u_1,E_1,B_1) \equiv (\rho_2,u_2,E_2,B_2).$$
\end{theorem}

The proof of Theorem \ref{Thm:stability} will be detailed at the end of this section. Before that, we first establish some key estimates on the solution of a general transport equation in negative Sobolev spaces which will be usefull later on in the proof.

\subsection{Forced transport equation in negative Sobolev spaces}\label{subsection:TE}
In this section, we prove general time-weighted estimates in negative Sobolev spaces of solutions to transport equations advected by a Lipschitz velocity field and supplemented with a vanishing initial datum. In particular, the results of Proposition \ref{Prop:TR:***}, below, besides being of independent interest,     will be useful in the proof of Theorem \ref{Thm:stability}, later on.

\begin{proposition}\label{Prop:TR:***}
Let $f\in L^\infty(\mathbb{R}^+;L^\infty)$ be the unique solution of the linear transport equation
\begin{equation}\label{TR:Equa_1}
\partial_t f + u \cdot \nabla f = G, \qquad f|_{t=0}\equiv 0 ,
\end{equation}  
where $u$ is a divergence-free vector field  enjoying the bound
$$ u \in L^1_{\loc}(\mathbb{R}^+; B^{1+\frac{2}{m}}_{m,1}),$$
for some $m\in (2,\infty).$
Assume further that there is $\alpha\in (1,\infty)$ such that 
$$u \in L^\alpha_{\loc}(\mathbb{R}^+; L^r), \qquad G\in  L^\alpha_{\loc}(\mathbb{R}^+; \dot{W}^{-1,r}),$$
for some $ r\in [m,\infty)$. 
Then,  it holds that 
$$X_{\beta}(t)\bydef \sup_{\tau \in (0,t]} \left( \tau^{-\frac 1\beta }  \norm {f(\tau)}_{\dot{W}^{-1,r}}\right)<\infty,  $$
for all $ 0\leq \frac 1 \beta  \leq   1- \frac{1}{\alpha}$ and   any $t>0$. Moreover, one has that 
\begin{equation*}
	t\mapsto \norm {f(t)}_{\dot W^{-1 ,r}} \in C(\mathbb{R}^+)
\end{equation*}  
and  
$$ \lim_{t\rightarrow 0} X_\beta (t)=0,$$
 for all $ 0\leq \frac 1\beta  \leq   1- \frac{1}{\alpha}$.
  Furthermore,    denoting
$$T^*\bydef  \sup \left \{\tau \in [0,T]:    \norm {f(\tau)}_{\dot{W}^{-1,r}}  <  \tau ^{1 - \frac{1}{\alpha}} \right\}, $$
for a fixed $T\in \mathbb{R}^+, $ then it holds that  $T^*>0$ and that
$$ 
\begin{aligned}
	X_\beta (t) 
	&\lesssim_t    \norm {G }_{L^\beta  ([0,t];\dot{W}^{-1,r})} 
	\\
	& \quad +   \log \left(e +  \norm f_{L^\infty([0,t];L^\infty)}  \right)    \int_0^t \norm {   u(\tau)}_{ {B}^{1+ \frac{2}{m}}_{m,1}}  X_\beta(\tau)  \log\left( e+\frac{1}{X_\beta(\tau)}  \right) d\tau ,
\end{aligned}
$$
for all $0 \leq \frac1\beta < 1-\frac{1}{\alpha}$ and any  $t\in [0, T^*].$
\end{proposition}   

\begin{proof}
Integrating the transport equation  \eqref{TR:Equa_1}  with respect to time yields that
$$f(t,\cdot )= -\int_0^t \div (uf)(\tau,\cdot )d\tau + \int_0^t G(\tau,\cdot )d\tau.$$
Hence,     it follows, for all $0\leq \frac 1\beta \leq 1-\frac{1}{\alpha} $  and any  $t>0$, that 
$$ t^{ -\frac 1\beta } \norm {f(t)}_{\dot{W}^{-1,r}} \leq  t^{ 1 - \frac{1}{\alpha}- \frac 1\beta} \left( \norm { u}_{L^\alpha_t L^r} \norm f_{L^\infty_{t,x}} + \norm { G}_{L^\alpha_t \dot{W}^{-1,r}}\right) ,  $$
whereby we deduce that 
$$ X_\beta(t)<\infty, \quad \text{ and } \qquad  \lim_{t\rightarrow 0} X_\beta(t)=0.$$    
Note in passing that by writing
\begin{equation*}
	f(t,\cdot ) - f(t_0,\cdot )= \int_{t_0}^t \div (uf)(\tau,\cdot )d\tau + \int_{t_0}^t G(\tau,\cdot )d\tau,
\end{equation*}
  for any $0\leq t_0\leq t$, we find by means of the same argument that 
\begin{equation*}
	t\mapsto \norm {f(t)}_{\dot W^{-1 ,r}} \in C(\mathbb{R} ^+).
\end{equation*} 
Thus, in particular, we deduce that $T^*>0$.

Next, we perform more precise estimates. To that end, we denote  the Fourier multiplier  of the  symbol  $|\xi|^{-1}$ by  $|D|^{-1}$. Then, by  applying this   Fourier multiplier  to  \eqref{TR:Equa_1}  and performing    standard $L^r$ estimates for transport equations,  we obtain that
$$  
\begin{aligned}
	\norm {f(t)}_{\dot{W}^{-1,r}}
& \lesssim     \norm {G }_{L^1_t\dot{W}^{-1,r}}   + \norm { [|D|^{-1} ,u\cdot \nabla ] f }_{L^1_t L^r}\\
 & \lesssim   \norm {G }_{L^1_t\dot{W}^{-1,r}}   + \norm { [|D|^{-1} ,u\cdot \nabla ] f }_{L^1_t B^0_{r,1}},
\end{aligned}
$$
for all $t\geq 0$.  Now, we show how to estimate   the commutator term  in the right-hand side by starting by  splitting it   into five terms,  implementing Bony's  decomposition  (see Section \ref{Paradifferential product estimates} in the appendix), 
$$[|D|^{-1} ,u\cdot \nabla ] f = \sum_{i=1}^5 \mathcal{I}_i, $$
where we set 
$$ \mathcal{I}_1 \bydef  |D|^{-1} \left(  T_{ \nabla f}u \right) , \qquad  \mathcal{I}_2 \bydef   -     T_{ |D|^{-1}  \nabla f} u ,\qquad \mathcal{I}_3 \bydef  |D|^{-1} \div  R(u,f), $$
$$ \mathcal{I}_4 \bydef   -  R(u,   |D|^{-1} \nabla  f) , \qquad   \mathcal{I}_5 \bydef  [ |D|^{-1},T_u ] \nabla f. $$ 
 Estimating the first four terms relies on standard properties of Bony's decomposition, which  lead to the control 
$$\begin{aligned}
\norm { \sum_{i=1}^4  \mathcal{I}_{i}}_{L^1_t B^0_{r,1}}  & \lesssim \int_0^t \norm { u(\tau)}_{ B^{1+\frac{2}{m}}_{m,1}} \norm { f (\tau)}_{B ^{-1}_{r,\infty}} d\tau
\\
&\lesssim \int_0^t \norm { u(\tau)}_{ B^{1+\frac{2}{m}}_{m,1}} \norm { f (\tau)}_{\dot{W} ^{-1,r} } d\tau,
\end{aligned}$$
 for any $ m\in (2,\infty)$ and any $r\in [m,\infty)$.
 Indeed, a straightforward justification of the previous estimates is obtained by an  application of \cite[Theorem 2.47 and Theorem 2.52]{bcd11} in combination with suitable embeddings of Besov spaces.

As for the estimate of $\mathcal{I}_{5}$, we find, by applying Lemma \ref{lemma:commutator} from the appendix with the value $s=-2$ and any given $\delta\in (0,1)$, that   
\begin{equation}\label{commutator:estimate:1}
	\begin{aligned}
		\norm {    \mathcal{I}_{5}}_{L^1_t B^0_{r,1}} 
		& \lesssim  \int_0^t \norm {\nabla u (\tau)}_{L^\infty} \norm { f(\tau)}_{B^{-1}_{r, \infty}}  \log \left(e+ \frac{    \norm { u (\tau)}_{B^{\delta }_{r,\infty}} \norm { f (\tau)}_{B^{0}_{\infty , \infty}} }{ \norm {\nabla u (\tau)}_{L^\infty} \norm { f (\tau)}_{B^{-1}_{r, \infty}}   } \right)  d\tau
		\\
		& \lesssim  \int_0^t \norm {  u (\tau)}_{\dot W^{1, \infty}\cap B^{\delta}_{r,\infty}} \norm { f(\tau)}_{\dot{W} ^{-1,r} }  \log \left(e+ \frac{      \norm { f}_{L^\infty_{t,x}} }{   \norm { f (\tau)}_{\dot{W} ^{-1,r} }   } \right)  d\tau,
	\end{aligned}
\end{equation}
where we have employed the embedding  (see Lemma \ref{Emb_Besov_Tr_Lemma}) $$ \dot{W} ^{-1,r}   \hookrightarrow  \dot{B} ^{-1}_{r,\infty}   \hookrightarrow   B^{-1}_{r,\infty} (\mathbb{R}^2).$$

Therefore, gathering the foregoing estimates on all $\mathcal{I}_i$, with $i=1,\ldots, 5$,   yields that  
$$ 
\begin{aligned}
 \norm {f(t)}_{\dot{W}^{-1,r}}  
 &   \lesssim  \norm {G }_{L^1_t\dot{W}^{-1,r}}  +   \log \left(e+ \norm f_{L^\infty_{t,x}}\right)
 \\ 
 & \quad\times  \int_0^t \norm {  u(\tau)}_{ \dot{B}^{1+ \frac{2}{m}}_{m,1} \cap B^{\delta}_{r,\infty}}  \norm {f(\tau)}_{\dot{W}^{-1,r}}  \log \left(e+  \frac{1}{   \norm {f(\tau)}_{\dot{W}^{-1,r}} }  \right) d\tau     .
\end{aligned} 
 $$  
Now, assuming that 
\begin{equation}\label{T:*:DEF}
	\norm {f(\tau)}_{\dot{W}^{-1,r}} < \tau^{ 1-\frac{1}{\alpha }} ,
\end{equation}
implies, by a direct computation, that 
$$   \log\left( e+ \frac{1}{ \norm {f(\tau)}_{\dot{W}^{-1,r}}} \right) < \log\left( \tau^\frac{1}{\beta} + 1  \right) +    \frac{1-\frac{1}{\alpha}}{ 1-\frac{1}{\alpha}- \frac 1\beta }   \log\left(e+  \frac{1}{ \tau ^ {-\frac 1\beta} \norm {f(\tau)}_{\dot{W}^{-1,r}}} \right) ,
$$
for any $0\leq \frac 1\beta < 1-\frac{1}{\alpha} $. Hence, we arrive at the control

$$ 
\begin{aligned}
 \norm {f(t)}_{\dot{W}^{-1,r}} 
 & \lesssim_t  \norm {G }_{L^1_t\dot{W}^{-1,r}}  
 + \log \left(e+ \norm f_{L^\infty_{t,x}}\right)
 \\
 & \quad  \times  \int_0^t \norm {  u(\tau)}_{ \dot{B}^{1+ \frac{2}{m}}_{m,1} \cap B^{\delta}_{r,\infty}}  \norm {f(\tau)}_{\dot{W}^{-1,r}} \log\left(e+  \frac{1}{ \tau ^ {-\frac 1\beta} \norm {f(\tau)}_{\dot{W}^{-1,r}}} \right) d\tau     .
\end{aligned} 
 $$  
At last, multiplying both sides by $t^{-\frac 1\beta }$, employing H\"older's inequality in time along with the monotonicity of the function $x\mapsto x \log\left(e+ \frac{1}{x}\right),$ for $x>0$,
leads to the   bound  
$$ \begin{aligned}
 X_\beta(t) &\lesssim_t       \norm {G }_{L^{\beta }_t\dot{W}^{-1,r}}    +  \log \left(e+ \norm f_{L^\infty_{t,x}}\right) 
   \\ 
   & \quad \times  \int_0^t \norm {  u(\tau)}_{ \dot{B}^{1+ \frac{2}{m}}_{m,1}  \cap B^{\delta}_{r,\infty}}  X_\beta(\tau ) \log \left(e+  \frac{1}{   X_\beta(\tau ) }  \right) d\tau  ,
\end{aligned}  $$
for all $ t >0$ satisfying the condition \eqref{T:*:DEF} with $\tau=t$. The proof of the final estimate  is achieved by noticing that 
\begin{equation*}
	B^{1+\frac{2}{m}}_{m,1} \hookrightarrow  B^{\delta + \frac{2}{r}}_{r,\infty} \hookrightarrow B^{\delta}_{r,\infty},
\end{equation*}
due to the assumptions on the parameters $m$, $r$ and $\delta$. This completes the proof of the proposition. 
\end{proof}

The reason why we have a log-type control in the preceding lemma stems from the estimate on the commutator $[|D|^{-1},T_u]$ given in \eqref{commutator:estimate:1}, where $r>2$.
However, in the case corresponding to $r=2$, due to the identity $\dot{B} ^{0}_{2,2}=L^2 $, it is possible to establish  an improved bound without logarithmic correction. The following proposition provides a precise estimate in this setting, with a proof which differs from the stability estimate established in \cite[Proposition 5.1]{Danchin_Wang_22}.

\begin{proposition}\label{Prop:TR:***2}
	Let $f$ be the solution of the linear transport equation \eqref{TR:Equa_1}. Assume that 
$$ u \in L^1_{\loc}(\mathbb{R}^+; \dot{B}^{1+\frac{2}{m}}_{m,1}),$$ 
for some $m\in [1,\infty)$ and  that there is $\alpha\in [1,\infty]$ such that 
$$u \in L^\alpha_{\loc}(\mathbb{R}^+; L^2)\qquad \text{and}\qquad G\in  L^\alpha_{\loc}(\mathbb{R}^+; \dot{H}^{-1}).$$ 
Then,  it holds,  for any $t>0$, that 
$$z(t)\bydef \sup_{\tau \in (0,t]} \left( \tau^{-1 +\frac{1}{\alpha}}  \norm {f(\tau)}_{\dot{H}^{-1 }}\right)<\infty ,\qquad  \lim_{t\rightarrow 0} z(t)=0,$$
and    
$$ z(t) \lesssim    \norm {G }_{L^\alpha([0,t];\dot{W}^{-1,r})} \exp\left(  C  \int_0^t \norm {  u(\tau)}_{ \dot{B}^{1+\frac{2}{m}}_{m,1}}   d\tau \right).$$ 
\end{proposition}

 \begin{proof}     
  We proceed by first localizing frequencies in the transport equations \eqref{TR:Equa_1} and writing,  for any $j\in \mathbb{Z}$, that
  $$\partial_t \Delta_j f+   u  \cdot \nabla \Delta_jf  =-   \Delta_j  G - \Big[  \Delta_j, u\cdot \nabla  \Big] f.   $$   
  Then, by an $L^2$-energy estimate, we find that 
  $$ \begin{aligned}
   \norm {f (t)}_{\dot{H}^{-1} } &\leq   \norm { G}_{L^1_t \dot{H}^{-1}} +\int_0^t \norm{  2^{-j}\norm {\Big[  \Delta_j, u\cdot \nabla  \Big] f (\tau)}_{L^2} } _{\ell^2(j\in \mathbb{Z})}   d\tau  .
  \end{aligned}$$ 
Hence, by appealing to    \cite[Lemma 2.100]{bcd11}, we deduce, for all $t>0$, that 
   $$ \begin{aligned}
   \norm {f (t)}_{\dot{H}^{-1} } &\lesssim   \norm { G}_{L^1_t \dot{H}^{-1}} +\int_0^t \norm {u(\tau)}_{ \dot{B}^{1+\frac{2}{m}}_{m,1}} \norm {f (\tau)}_{\dot{H}^{-1} } d\tau  .
  \end{aligned}$$    
  The proof is then completed by  employing H\"older's inequality in time to estimate $G$ in $ L^\alpha_t \dot{H}^{-1}$, and applying Gr\"{o}nwall's inequality. 
\end{proof}

\subsection{Proof of weak--strong uniqueness}

We are now in a position to prove Theorem \ref{Thm:stability} by employing Propositions \ref{Prop:TR:***} and \ref{Prop:TR:***2} for different  values of $\alpha$. To that end, allow us first to set up some notation to be used in the proof.

The difference of two solutions will be denoted by 
\begin{equation*}
\delta f \bydef f _1-f_2,\quad \text{for any } f\in\{\rho, u, E, B,j,p\}.
\end{equation*}
In particular, we observe that the differences solve the system of equations
\begin{equation}\label{system_difference}
\left\{
\begin{aligned}
\partial_t(\delta \rho)+u_2\cdot \nabla (\delta \rho) + \delta u\cdot \nabla \rho_1 
&= 0  ,
\\
\rho_1 \big( \partial_t(\delta u)+  u_1\cdot \nabla \delta u\big) -\Delta \delta u +\nabla \delta p
&=-\delta \rho \dot{u}_2-\rho_1\delta u\cdot \nabla u_2
\\&\quad+\delta j\times B_1+j_2\times \delta B,
\\
\frac{1}{c} \partial_t (\delta E) - \nabla \times (\delta B) 
&=-\delta j,
\\
\frac{1}{c} \partial_t (\delta B) + \nabla \times \delta E  
&= 0,
\\
 \delta j
 &= \sigma \left( c\delta E +  u_1 \times \delta B+\delta u\times B_2 \right)
 \\
\delta\rho|_{t=0}=\delta u|_{t=0}=\delta E|_{t=0}=\delta B|_{t=0}&=0,
\end{aligned}
\right.
\end{equation} 
where we adopt the notation 
\begin{equation*}
\dot f \bydef(  \partial_t + u _2\cdot \nabla )f.
\end{equation*}
   Moreover, for a fixed $r\in (2,\infty)$, we define 
   \begin{equation*} 
   \left\{
   \begin{aligned}
   X(t)\bydef&\, \sup_{\tau \in [0,t]} \left( \tau^{-\frac{1}{2}  }  \norm {\delta \rho(\tau)}_{\dot{W}^{-1,r}}\right),\\
   Y (t) \bydef &\, \Bigg( \sup_{\tau \in [0,t]}\Big( \norm {\sqrt{\rho_1}\delta u (\tau) }_{L^2 } ^2+  \norm {( \delta E, \delta B ) (\tau) }_{L^2 } ^2 \Big) 
   \\ & \qquad +  \int_0^t \Big(\norm {\delta u (\tau) }_{\dot{H}^1 }^2 +  \norm {\delta j (\tau) }_{L^2 }^2\Big) d\tau \Bigg)^{\frac{1}{2}} ,\\
   Z(t) \bydef&\, \sup_{\tau \in [0,t]} \left( \tau^{-1}  \norm {\delta \rho(\tau)}_{\dot{H}^{-1}}\right).
   \end{aligned}
   \right.
   \end{equation*} 
   Notice that  definition of the function of time $X$ corresponds to the function $X_\beta$, with $\beta=2$, from Proposition \ref{Prop:TR:***}. Moreover, here, we will be considering the case $\alpha=r$ in the notation of that proposition. 
   
    The proof is split into two steps for a better readability.

\subsubsection*{Energy estimates}
    We begin by establishing adequate estimates for $\delta \rho$ by employing the results of  Proposition \ref{Prop:TR:***}.  To that end,  we observe that the first equation  in \eqref{system_difference} can be recast as  
\begin{equation}\label{Transport_Difference}
\partial_t(\delta \rho)+u_2\cdot \nabla (\delta \rho)=-\div(\delta u \rho_1),\qquad \delta\rho|_{t=0}=0. 
\end{equation} 
Then, we apply Proposition \ref{Prop:TR:***} to \eqref{Transport_Difference}   followed by  H\"older's inequality     to infer that 
 $$ \begin{aligned}     
X(t)&\lesssim    \norm { \delta u \rho_1 }_{L^2_t L^r}  +    \log\left(e +  \norm {\delta \rho}_{L^\infty_{t,x}}  \right)    \int_0^t \norm {   u_2(\tau)}_{ {B}^{1+ \frac{2}{r}}_{r,1}}  X(\tau)  \log\left(e+ \frac{1}{X(\tau)}  \right) d\tau \\
 &\lesssim  t^{\frac{1}{r}  } \norm { \delta u}_{L^\frac{2r}{r-2} _t L^r}\norm { \rho_1 }_{L^\infty_{t,x} }
 \\
 &\quad +    \log \left(e +  \norm {\delta \rho}_{L^\infty_{t,x}}  \right)    \int_0^t \norm {   u_2(\tau)}_{ {B}^{1+ \frac{2}{r}}_{r,1}}  X(\tau)  \log\left(e+ \frac{1}{X(\tau)}  \right) d\tau ,
\end{aligned}  $$ 
for any $t\in [0,T^*]$, where $T^*$ is defined by 
$$ T^*\bydef  \sup \left \{t \in [0,T]:  \sup_{\tau \in [0,t]} \left(  \tau^{\frac{1}{r}-1 } \norm {\delta \rho(\tau)}_{\dot{W}^{-1,r}}  \right) <1 \right\}.$$
The time $T>0$ will be fixed in the next step of the proof and will only depend on  fixed universal constants. Note, in passing, that $T^*>0$, thanks to Proposition \ref{Prop:TR:***}, and, in fact, it will be shown at the end of this proof that $T^*=T$.

We assume now that all estimates are performed on the time interval $[0,T^*]$. However, for convenience of notation, we will not write this explicitly.

Now, observe,  since $r\in (2,\infty)$ and $ \rho_1 >0$, that a straightforward interpolation argument allows us to write 
$$  \norm { \delta u}_{L^\frac{2r}{r-2} _t L^r} \lesssim   \norm { \delta u}_{L^\infty_t L^ 2}  +  \norm { \delta u}_{ L^ 2_t \dot{H}^1}  \lesssim Y(t).$$
Therefore, by the boundedness of the densities $\rho_1$ and $\rho_2$, we find that  
\begin{equation}\label{STAB:EQ:1}
 \begin{aligned}
X(t)& \lesssim  t^{\frac{1}{r}  }  Y(t)  +      \int_0^t \norm {   u_2(\tau)}_{ {B}^{1+ \frac{2}{r}}_{r,1}}  X(\tau)  \log\left( e+\frac{1}{X(\tau)}  \right) d\tau,
\end{aligned} 
\end{equation}
where all the implicit constants are allowed to depend on the upper and lower bounds on the initial densities.

On the other hand,  by employing  Proposition \ref{Prop:TR:***2},  we obtain from \eqref{Transport_Difference} that   
\begin{equation*}
Z(t) \lesssim  \norm {\div(\delta u \rho_1)}_{L^\infty_t \dot{H}^{-1}} \exp \left(C \int_0^t \norm {  u_2(\tau)}_{{B} ^{1+\frac{2}{r}}_{r,1} } d\tau \right).
\end{equation*}
Thus, it follows, by  H\"older's inequality and from the maximum principle \eqref{mass:conservation}, that 
\begin{equation}\label{Bound:Z}
Z(t) \lesssim Y(t) \exp \left( C\int_0^t \norm {  u_2(\tau)}_{{B} ^{1+\frac{2}{r}}_{r,1} } d\tau \right).
\end{equation} 

Next, we focus on the estimate of the velocity and the electromagnetic fields contained in $Y(t)$. According to the assumption in Theorem \ref{Thm:stability},  the velocity field splits as $ u_2=v_2 + w_2.$  
Hence,  performing an energy estimate   on \eqref{system_difference} yields
\begin{equation*}
\begin{aligned}
\frac{1}{2} \frac{d}{dt} \norm { \sqrt{\rho_1}\delta u }_{L^2} ^2+ \norm {\delta u }_{\dot{H}^1} ^2 & = -   \int_{\mathbb{R}^2} \delta \rho \dot{v}_2 \cdot \delta u \,dx   -    \int_{\mathbb{R}^2} \delta \rho \dot{w}_2 \cdot \delta u  \,dx
\\
& \quad -   \int _{\mathbb{R}^2} \rho _1 (\delta u \cdot \nabla u_2 ) \cdot 
\delta 
u  \,dx  +  \int _{\mathbb{R}^2} ( \delta j\times B_1 )\cdot \delta u  \,dx 
\\
& \quad+ \int _{\mathbb{R}^2} ( j_2 \times \delta B )\cdot \delta u \,dx,
\end{aligned}
\end{equation*}
and
\begin{equation*}
\begin{aligned}
\frac{1}{2} \frac{d}{dt} \left(  \norm {  \delta B }_{L^2} ^2+   \norm {  \delta E }_{L^2} ^2 \right)+ \frac{1}{\sigma}  \norm {\delta j }_{L^2} ^2 & =     \int _{\mathbb{R}^2}  \delta j \cdot( \delta u \times B_1)   \,dx
\\
& \quad + \int _{\mathbb{R}^2}  \delta j \cdot(  u_2 \times \delta B    ) \,dx,
\end{aligned}
\end{equation*} 
where we have used that $\rho_1$ solves a transport equation advected by $u_1$ in the first identity, above. Therefore, 
summing the preceding equations  and integrating with respect to time, we arrive at   
\begin{equation*}
\begin{aligned}
Y ^2  (t)   
& \lesssim -  \int_0^t\int_{ \mathbb{R}^2} \delta \rho \dot{v}_2 \cdot \delta u(\tau)  \,dx d\tau  
\\
&\quad -   \int_0^t\int_{ \mathbb{R}^2} \delta \rho \dot{w}_2 \cdot \delta u(\tau) \,dx-   \int_0^t\int 
_{\mathbb{R}^2} \rho _1 (\delta u \cdot \nabla u_2 ) \cdot \delta u (\tau)\,dx 
\\
& \quad 
+ \int_0^t\int _{\mathbb{R}^2} \Big( \delta j \cdot(  u_2 \times \delta B    ) (\tau)+ ( j_2 \times \delta B )\cdot \delta u (\tau) \Big)\,dx\vspace{0.2cm}\\
& \bydef  \sum_{i=0}^3  \mathcal{I}_i(t) ,
\end{aligned}
\end{equation*}
and we take care of each term separately.

In order to estimate  $\mathcal{I}_3$, we apply  H\"older's inequality together with the two dimensional interpolation  inequality  (see Lemma \ref{Emb_Besov_Tr_Lemma} and \eqref{interpolation.AP-B}, in the appendix) 
\begin{equation}\label{INTER.a}
\Vert f \Vert_{L^\frac{2a}{a-2}} \lesssim \Vert f \Vert_{\dot{H}^{ \frac{2}{a}}}\lesssim  	\Vert f \Vert_{ L^2}^{ 1- \frac{2}{a}} \Vert    f\Vert_{\dot{H}^1}^{\frac{2}{a}}   , \qquad a \in (2,\infty),
\end{equation} 
to obtain,   for any $\varepsilon>0$, that   
\begin{equation*}
\begin{aligned}
\mathcal{I}_3(t)  &  \leq \int_0^t \norm {\delta j(\tau)}_{L^2}   \norm {u_2(\tau)}_{L^\infty}   \norm {\delta B(\tau)}_{L^2}  d\tau  \\
& \quad+ \int_0^t  \norm {  j_2(\tau)}_{L^a}   \norm {\delta B(\tau) }_{L^2}   \norm {\delta u(\tau)}_{L^\frac{2a}{a-2} } d\tau   \\  
&    \leq \int_0^t \norm {\delta j (\tau)}_{L^2}   \norm {u_2(\tau)}_{L^\infty}   \norm {\delta B (\tau)}_{L^2} d\tau  
\\
& \quad + \int_0^t \norm {  j_2 (\tau)}_{L^a}   \norm {\delta B (\tau)}_{L^2}   \norm {\delta u(\tau)}_{L^2}^{1-\frac{2}{a}} \norm {\delta u(\tau)}_{\dot{H}^1}^\frac{2}{a} d\tau    .
\end{aligned}
\end{equation*}
Therefore, applying Young's inequality to deduce that
$$ \begin{aligned}
 \norm {  j_2 (\tau)}_{L^a}   \norm {\delta B (\tau)}_{L^2}   &\norm {\delta u(\tau)}_{L^2}^{1-\frac{2}{a}}   \norm {\delta u(\tau)}_{\dot{H}^1}^\frac{2}{a}  \\
& \quad \leq   \frac{ \varepsilon  }{2} \norm {\delta u(\tau)}_{\dot{H}^1}^2  + C_{\varepsilon}   \norm {  j_2 (\tau)}_{L^a} ^{\frac{a}{a-1}} \norm {\delta B (\tau)}_{L^2} ^{\frac{a}{a-1}}  \norm {\delta u(\tau)}_{L^2} ^\frac{a- 2}{a-1}  \\
& \quad \leq    \frac{ \varepsilon  }{2}   \norm {\delta u(\tau)}_{\dot{H}^1}^2 + C_{\varepsilon}   \norm {  j_2 (\tau)}_{L^a} ^{\frac{a}{a-1}} \left( \norm {\delta B (\tau)}_{L^2}^2 +   \norm {\delta u(\tau)}_{L^2} ^2  \right)
\end{aligned}
$$
and 
$$ \norm {\delta j(\tau)}_{L^2}   \norm {u_2(\tau)}_{L^\infty}   \norm {\delta B(\tau)}_{L^2}  \leq   \frac{  \varepsilon}{2} \norm {\delta j(\tau)}_{L^2 }^2+   C_\varepsilon    \norm {u_2(\tau)}_{L^\infty}^2   \norm {\delta B (\tau)}_{L^2}^2 ,  $$
we arrive at the estimate
$$\mathcal{I}_3(t)    \leq   \varepsilon  \left(\norm {\delta j}_{L^2_{t,x}}^2  + \norm {\delta u}_{L^2_t\dot{H}^1}^2 \right)  +  C_\varepsilon \int_0^t \left(  \norm {u_2(\tau)}_{L^\infty}^2 +      \norm {  j_2 (\tau)}_{L^a} ^\frac{a}{a-1} \right) Y^2(\tau) d\tau.$$
The first term will be absorbed by the dissipation terms in $Y(t)$ and the second term will be estimated using Gr\"onwall's inequality.

We turn now to the estimate of $\mathcal{I}_{2}$, which   is obtained by employing  H\"older's inequality, the maximum principle \eqref{mass:conservation} and the Sobolev embedding  \eqref{sobolev-infinity} to reach the conclusion that 
$$ \begin{aligned}
 \mathcal{I}_2(t) &\leq \int_0^t \norm {\rho_1(\tau)}_{L^\infty} \norm {\nabla u_2(\tau)}_{L^\infty}  \norm {\delta u(\tau)}_{L^2}^2 d\tau\\
 &\lesssim \norm{\rho_0}_{L^\infty}\int_0^t\norm {\nabla u_2(\tau)}_{L^\infty}  Y^2(\tau)d\tau\\
 & \lesssim \int_0^t  \norm {  u_2(\tau)}_{\dot{B}^{1+\frac{2}{r}}_{r,1}}Y^2(\tau)d\tau.
\end{aligned}$$
As for the estimate of $\mathcal{I}_1$, we first proceed  by duality and we utilize H\"older's inequality to find that 
$$ \begin{aligned}
 \mathcal{I}_1(t) & \leq  \int_0^t \norm { \delta \rho(\tau)}_{\dot{W}^{-1,r} } \norm { \dot{w}_2 \cdot \delta u(\tau)}_{\dot{W}^{1,\frac{r}{r-1}} } d\tau \\ 
  & \leq\int_0^t   \norm { \delta \rho(\tau)}_{\dot{W}^{-1,r} } \left(  \norm { \dot{w}_2 (\tau) }_{L^{\frac{2r}{r-2}} }   \norm { \nabla \delta {u}  (\tau)}_{L^2 }   +  \norm {\nabla  \dot{w}_2 (\tau) }_{L^2}   \norm {  \delta {u}  (\tau)}_{L^ {\frac{2r}{r-2}}  }   \right)  d\tau\\
 & \leq   \int_0^t X(\tau) \left(  \norm { \tau^{\frac{1}{2}  }\dot{w}_2 (\tau) }_{L^{\frac{2r}{r-2}} }   \norm { \nabla \delta {u}  (\tau)}_{L^2 }   +  \norm {\tau^{\frac{1}{2} }\nabla  \dot{w}_2 (t) }_{L^2}   \norm {  \delta {u}  (\tau)}_{L^ {\frac{2r}{r-2}}  }   \right) d\tau.   
\end{aligned}$$ 
Therefore,   employing \eqref{INTER.a} with $a=r$, we obtain that 
$$ \begin{aligned}
 \mathcal{I}_1(t) &  \leq  \int_0^t X(\tau) \left(  \norm { \tau^{\frac{1}{2} }\dot{w}_2 (\tau) }_{L^2 } ^{1-\frac{2}{r}} \norm { \tau^{\frac{1}{2} }\dot{w}_2 (\tau) }_{ \dot{H}^1}^{\frac{2}{r}}   \norm {  \delta {u}  (\tau)}_{\dot{H}^1 } \right)  d\tau\\
 &\quad  +  \int_0^t X(\tau)\left(  \norm {\tau^{\frac{1}{2} }   \dot{w}_2 (\tau) }_{\dot{H}^1}   \norm {  \delta {u}  (\tau)}_{ L^2}^{1-\frac{2}{r} }  \norm { \delta {u}  (\tau)}_{ \dot{H}^1}^{ \frac{2}{r} }  \right) d\tau.
\end{aligned}$$   
By further employing Young's inequalities to write, for any $\varepsilon>0$ and some $C_\varepsilon >0$, that 
$$   \begin{aligned}
X(\tau)   \norm { \tau ^{\frac{1}{2} }\dot{w}_2 (\tau) }_{L^2 } ^{1-\frac{2}{r}}  &\norm {  \tau^{\frac{1}{2} }\dot{w}_2 (\tau) }_{ \dot{H}^1}^{\frac{2}{r}}   \norm {  \delta {u}  (\tau)}_{\dot{H}^1 } 
\\
& \quad \leq   C_\varepsilon X^2 (\tau )    \norm { \tau^{\frac{1}{2} }\dot{w}_2 (\tau) }_{H^1} ^2   + 
\frac{\varepsilon}{2} 
  \norm { \delta {u} (\tau)    }_{ \dot{H}^1 } ^2      , 
\end{aligned}  $$    
and, as long as $r\in (2,\infty)$, that 
$$ \begin{aligned}
X(\tau) &\left(  \norm {\tau^{\frac{1}{2} }  \dot{w}_2 (\tau) }_{\dot{H}^1}   \norm {  \delta {u}  (\tau)}_{ L^2}^{1-\frac{2}{r} }  \norm { \delta {u}  (\tau)}_{ \dot{H}^1}^{ \frac{2}{r} }  \right) \\
& \qquad\qquad \leq  C_\varepsilon    \norm{\tau^{\frac{1}{2} }   \dot{w}_2 (\tau) }_{\dot{H}^1}^{\frac{r}{r-1}} X^{ \frac{r}{r-1}} (\tau)  \norm {  \delta {u}  (\tau)}_{ L^2}^{\frac{r-2}{r-1}}  + \frac{\varepsilon}{2}  \norm { \delta {u} (\tau)   }_{ \dot{H}^1 } ^2   
\\
& \qquad\qquad \leq  C_\varepsilon   \left(1+ \norm{\tau^{\frac{1}{2} }   \dot{w}_2 (\tau) }_{\dot{H}^1}^{2}\right) \Big(  X^{2} (\tau)  +  Y^{2} (\tau)  \Big) + \frac{\varepsilon}{2}  \norm { \delta {u} (\tau)   }_{ \dot{H}^1 } ^2   ,
\end{aligned}$$
we finally arrive at the bound 
  $$ \begin{aligned}
 \mathcal{I}_1(t) &  \leq  C_{\varepsilon}\int_0^t   \left(  1+  \norm { \tau^{\frac{1}{2} }\dot{w}_2 (\tau) }_{H^1} ^2   \right)  \left( X^2(\tau) +  Y^2(\tau)\right)  d\tau    + \varepsilon  \norm { \delta {u}   }_{L^2_t\dot{H}^1 } ^2    .    
\end{aligned}$$

In order to estimate $\mathcal{I}_0$, we proceed as in \cite{Danchin_Wang_22}.  To that end, by  duality, we find that 
\begin{equation*} 
\begin{aligned}
\mathcal{I}_0(t) \leq&\, \int_0^t \|\delta \rho \|_{\dot{H}^{-1}}\|\dot{v}_2 \cdot \delta u(\tau) \|_{\dot{H}^1} \, d\tau\\
 \leq&\,
 \int_0^t Z(\tau)\|\tau \dot{v}_2 \cdot \delta u(\tau) \|_{\dot{H}^1} \, d\tau\\
 \leq&\,
\int_0^t Z(\tau)\Big(\|\tau \nabla\dot{v}_2 \cdot \delta u(\tau) \|_{L^2} +\|\tau \dot{v}_2 \cdot \delta \nabla u(\tau) \|_{L^2}\Big)\, d\tau.
\end{aligned}
\end{equation*}
Then, by employing H\"older's inequality in combination with the embedding
\begin{equation*} 
\dot{W}^{1,p} \hookrightarrow L^m(\R^2),\qquad \text{where}\qquad \frac{1}{m} + \frac{1}{2}=\frac{1}{p} , 
\end{equation*} 
  we infer that 
\begin{equation*} 
\begin{aligned}
\mathcal{I}_0(t) \leq 
\int_0^t Z(\tau)\Big(\|\tau \nabla^2\dot{v}_2\|_{L^p} \| \delta u(\tau) \|_{L^{p'}} +\|\tau \dot{v}_2\|_{L^\infty} \| \delta \nabla u(\tau) \|_{L^2}\Big)\, d\tau ,
\end{aligned}
\end{equation*}
where $p'$ denotes the conjugate exponent of $p$.  Therefore,  utilizing the fact that $\frac{1}{p}+ \frac{1}{q}= \frac{3}{2}$ and appealing to the interpolation inequality 
\begin{equation*} 
\begin{aligned}
  \| \delta u(\tau) \|_{L^{p'}} \lesssim&\,  \| \delta u(\tau) \|_{L^2}^{2 - \frac{2}{p}} \| \nabla \delta u(\tau) \|_{L^2}^{\frac 2p -1 } = \| \delta u(\tau) \|_{L^2}^{  \frac{2}{q} -1} \|  u(\tau) \|_{\dot{H}^1}^{2- \frac{2}{q}} ,
\end{aligned}
\end{equation*}
followed by Young's inequality, we obtain  that 
\begin{equation*}   
 \begin{aligned}
 \mathcal{I}_0(t) &  \leq     C_{\varepsilon}  \int_0^t    \norm { \tau \dot{v}_2 (\tau) }_{L^\infty } ^2   Z^2(\tau) d\tau \\
& \quad+  C_{\varepsilon }      \int_0^t\norm { \tau \dot{v}_2 (\tau) }_{\dot{W}^{2,p} } ^q    \norm {  \delta {u}  (\tau)}_{ L^2}^{2-q}  Z^q(\tau)d\tau  + \varepsilon  \norm { \delta {u}   }_{L^2_t 
\dot{H}^1 } ^2   ,
\end{aligned}
\end{equation*}
for any $\varepsilon >0$. Hence, we infer, for $q<2$,  that 
\begin{equation*}   
 \begin{aligned}
 \mathcal{I}_0(t) &  \leq       C_{\varepsilon }      \int_0^t \left( \norm { \tau \dot{v}_2 (\tau) }_{L^\infty } ^2  + \norm { \tau \dot{v}_2 (\tau) }_{\dot{W}^{2,p} } ^q  \right) \left( Y^2(\tau) +  Z^2(\tau)\right) d\tau   + \varepsilon  \norm { \delta {u}   }_{L^2_t \dot{H}^1 } ^2   .
\end{aligned}
\end{equation*}

All in all, gathering the foregoing estimates, while choosing $\varepsilon$ as small as it is needed,  leads us to the final bound
\begin{equation} \label{Y:ES0}
Y^2 (t) \lesssim  \int_0^t \mathcal{A}(\tau) \left( X^{2}(\tau) +   Y^2(\tau) + Z^2(\tau)\right) d\tau  ,
\end{equation}   
where we set 
$$  \begin{aligned}
\mathcal{A}(\tau) &\bydef 1 +   \norm { \tau \dot{v}_2 (\tau) }_{L^\infty } ^2 +  \norm { \tau \dot{v}_2 (\tau) }_{\dot{W}^{2,p} } ^q  + \norm { \tau^{\frac{1}{2} }\dot{w}_2 (\tau) }_{H^1} ^2   \\
& \quad+   \norm {  u_2(\tau)}_{B^{1+\frac{2}{r}}_{r,1}} + \norm {u_2(\tau)}_{L^\infty}^2 + \norm {j_2(\tau)}_{L^a}^ {\frac{a}{a-1}}. 
\end{aligned}$$

\subsubsection*{Conclusion of the proof}

We show now how to combine the preceding bounds to deduce the uniqueness of solutions to \eqref{MHD_System_Nonhom}. This uniqueness is achieved on a fixed time interval $[0,T]$, for a suitable value of $T>0$. Then, in a subsequent step, the uniqueness is extended to hold on $\mathbb{R}^+_\loc$ by a standard bootstrap argument.

  To that end,  first observe  that  the assumptions in the statement of Theorem \ref{Thm:stability} ensure  that  
$$ \mathcal{A} \in L^1_{\loc}(\mathbb{R}^+).$$ 
Moreover, by virtue of \eqref{Bound:Z} and \eqref{Y:ES0},
we find that
\begin{equation}\label{STAB:EQ:2}
Y(t) \lesssim  \mathcal{U}(t) \left(  \int_0^t \mathcal{A}(\tau) \left( X^{2}(\tau) +   Y^2(\tau)  \right) d\tau  \right)^\frac 12,
\end{equation}
where, we set
$$ t\mapsto \mathcal{U}(t) \bydef  \exp \left(C \int_0^t \norm {  u_2(\tau)}_{\dot{B} ^{1+\frac{2}{r}}_{r,1} } d\tau \right) \in L^\infty_\loc(\mathbb{{R}^+}).$$
Furthermore, the bound  \eqref{STAB:EQ:1}     implies, for some   constant $C_*>0$ which only depends on the upper and lower bounds on the initial densities $\rho_1$ and $\rho_2$, that 
\begin{equation*} 
 \begin{aligned}
X(t)& \leq C_*  t^{\frac{1}{r}  }  Y(t)  +    C_*  \int_0^t \mathcal{A}(\tau)  X(\tau)  \log\left(e+ \frac{1}{X(\tau)}  \right) d\tau,
\end{aligned} 
\end{equation*}
  for any $t\in [0,T^*]$, where we recall that $T^*$ is defined as
$$ T^*\bydef   \sup \left \{t \in [0,T]:  \sup_{\tau \in [0,t]} \left(  \tau^{\frac{1}{r}-1 } \norm {\delta \rho(\tau)}_{\dot{W}^{-1,r}}  \right) <1 \right\},$$
and $T$ is now fixed in terms of the final constant $C_*$ and   is defined by $$   T\bydef  \left(\frac{1}{2C_*}  \right)^{ r}. $$   Consequently, we infer,  for any $t\in (0,T^*]$, that
\begin{equation}\label{STAB:EQ:1-b}
 \begin{aligned}
X(t)& \leq \frac{1}{2} Y(t)  +    C_*  \int_0^t \mathcal{A}(\tau) X(\tau)  \log\left(e+ \frac{1}{X(\tau)}  \right) d\tau.
\end{aligned} 
\end{equation}

Next, by summing \eqref{STAB:EQ:2} and \eqref{STAB:EQ:1-b}, we find that 
\begin{equation*}
	\begin{aligned}
		\mathcal {H}(t) 
		&\bydef X(t) + Y(t)
		\\
		& \lesssim   \int_0^t \mathcal{A}(\tau) \mathcal {H}(\tau)  \log\left(e+ \frac{1}{\mathcal {H}(\tau)}  \right) d\tau + \mathcal {H}^\frac 12 (t) \mathcal{U}(t) \left(  \int_0^t \mathcal{A}(\tau) \mathcal {H}(\tau) d\tau  \right)^\frac 12
		\\
		& \lesssim   \varepsilon \mathcal {H}  (t) +  \int_0^t \mathcal{A}(\tau) \mathcal {H}(\tau)  \log\left(e+ \frac{1}{\mathcal {H}(\tau)}  \right) d\tau + C_\varepsilon \mathcal{U}^2(t)  \int_0^t \mathcal{A}(\tau) \mathcal {H}(\tau) d\tau   .
	\end{aligned}
\end{equation*}
Therefore, choosing $\varepsilon$ small enough yields that 
\begin{equation*}
	\mathcal {H}(t)  \lesssim (1+ \mathcal {U}(t))^2 \int_0^t \mathcal{A}(\tau) \mathcal {H}(\tau)  \log\left(e+ \frac{1}{\mathcal {H}(\tau)}  \right) d\tau,
\end{equation*}
for all $t \in [0,T^*]$. Hence, by Osgood's lemma (see \cite[Lemma A.1]{hhz2023}),
it follows that 
\begin{equation*}
	\mathcal{H}\equiv X\equiv Y \equiv  0
\end{equation*}
on $[0,T^*]$. In particular, we deduce, by virtue of the definition of $ X(t)$,  that 
\begin{equation*}
	\norm {\delta \rho(t)}_{\dot{W}^{-1,r}} = 0,
\end{equation*}
for all $t \in [0,T^*]$, which in turn implies that 
\begin{equation*}
	\sup_{t \in [0,T]} \left(  t^{\frac{1}{r}-1 } \norm {\delta \rho(t)}_{\dot{W}^{-1,r}}  \right) = 0 <1.
\end{equation*} 
Consequently,  by the very definition of $T^*$, we arrive at the conclusion that $$T^*= T =\left(\frac{1}{2C_*}\right) ^{r}.$$

In summary, we have shown  that the uniqueness of solutions holds   on the whole fixed time  interval $ [0,\left( 2C_* \right) ^{-r}]$. Therefore, by a bootstrap argument, repeating the previous procedure of proof, it can be shown that the uniqueness   holds on any finite interval of time in $\mathbb{R}^+$.
It is very important to mention here that the constant $C_*$, and eventually the maximal time interval $T$, depend  only on universal constants and on the lower and upper bounds on the   densities, which are conserved quantities. This guarantees in particular  that  the  size   of the time interval is unchanged in  the  bootstrap algorithm, thereby ensuring that this uniqueness argument holds on any finite time interval. The proof of  Theorem \ref{Thm:stability} is now completed. \qed

\section{Time-weighted estimates (II)}
\label{Section.TW:es2}

The bounds established in Propositions \ref{Thm-summary} and  \ref{corollary:u-Lip} will now be employed to obtain   additional  time-weighted estimates on the solutions constructed in Section \ref{Section:a priori ES}, in the case $s\in(\frac 12,1)$ (so that all bounds from Section \ref{Section.TW1} hold true). These new estimates will ensure that the hypotheses of Theorem \ref{Thm:stability} are all satisfied, which will then lead to the uniqueness stated in Theorem \ref{Thm:1}.

Before getting into the details, allow  us   to  recall a few facts that are routinely used below.   First, the velocity field is split into  
$$ u=v+ w , \qquad     v|_{t=0}=u_0, \qquad  w|_{t=0}=0,$$
 where each part is governed by the equations
\begin{equation}\label{v-equa*}
\rho \left(\partial_t v + u \cdot \nabla v \right) - \Delta v+ \nabla p_v = 0, 
\end{equation}  
and 
\begin{equation}\label{w-equa*}
\rho \left(\partial_t w + u \cdot \nabla w \right) - \Delta w+ \nabla p_w = j\times B,
\end{equation}
respectively.  As for the current density $j$,  we recall that it is given by Ohm's law \eqref{Ohms-law1} or \eqref{Ohms-law2}. On the other hand, the electromagnetic field  $(E,B)$ is a solution to the Maxwell equations \eqref{MX*} and, finally, the density $\rho$ is transported by the flow of the full velocity field $u$
\begin{equation*}
\partial_t \rho + u \cdot \nabla \rho=0, \qquad \rho|_{t=0}= \rho_0 .
\end{equation*}      
 
 The bounds on $v$ are obtained in the same fashion as in \cite{Danchin_Wang_22} by utilizing the maximal  smoothing effect of the Stokes operator, whose proof  is   outlined in Section \ref{subsection:v-TW}, for the sake of completeness.  As for the bounds on $w$ presented in Section \ref{subsection:w-TW}, below, they will rest on the structure of the Lorentz force $j\times B$ in combination with  an appropriate duality argument. In particular, the control of $w$ will not rely on the maximal smoothing effect of the Stokes operator.

\subsection{Control of $v$}
\label{subsection:v-TW}

We estimate $v$ by employing the method from \cite{Danchin_Wang_22}, which corresponds to the case $E\equiv B \equiv 0$. This is possible due to the fact that, as we will see, the field $w$ enjoys bounds which are better than the ones satisfied by $v$. Thus, the advection term $ u\cdot \nabla v$ will essentially behave as $v\cdot \nabla v$.

\begin{proposition}[in the spirit of {\cite[Proposition  3.4]{Danchin_Wang_22}}]
	\label{prop:v:tv1}
	It holds, under the assumptions of Proposition \ref{corollary:u-Lip}, that
	$$  t^\frac{1}{2} v \in L^\infty_\loc(\mathbb{R}^+; \dot{H}^1(\mathbb{R}^2))    ,$$
	and
	$$ t^{\frac{1}{2}} \dot{v} , t^{\frac{1}{2}} \partial_t v \in L^2_\loc(\mathbb{R}^+; L^2(\mathbb{R}^2))   .$$
\end{proposition}

\begin{proof}
	Note first, since $\dot{v}= \partial_t v+u\cdot\nabla v $,  that 
$$  \int_0^t \tau^{\frac{1}{2}}\norm { \dot{v}(\tau) -   \partial_t v(\tau) }_{  L^2}^2 d\tau \leq \|\tau^{\frac{1}{2}} u\|_{L^\infty_tL^\infty}^2    \norm {\nabla v }_{L^2_t  L^2}^2    ,$$
where the right-hand side is finite, for any $t>0$, due to \eqref{energy-v} and \eqref{u:TW1}. Hence,     our claim is reduced to showing that 
$$ t ^\frac{1}{2} v \in L^\infty_\loc(\mathbb{R}^+; \dot{H}^1)    \qquad \text{and} \qquad   t^{\frac{1}{2}} \partial_t v  \in L^2_\loc(\mathbb{R}^+; L^2)   .$$
To that end, taking the inner product of \eqref{v-equa*} with $ t\partial_t v$ gives, for any $t>0$, that 
\begin{equation*}
t \int_{\mathbb{R}^2} \rho |\partial_t v|^2 dx +\frac{1}{2} \frac{d}{dt} \left(  t \int_{\mathbb{R}^2} |\nabla v|^2 dx \right) =  \frac{1}{2} \int_{\mathbb{R}^2} |\nabla v|^2 dx -  t  \int_{\mathbb{R}^2} \rho (u \cdot \nabla v) \cdot \partial_t v  dx,
\end{equation*}
where we have used the divergence-free condition of the velocity field  to get rid of the pressure term.
Therefore, integrating in time and employing H\"older's inequality with \eqref{mass:conservation} yields that
\begin{equation*}
\begin{aligned}
\|\tau^{ \frac{1}{2}}\sqrt{ \rho}  \partial_t v\|  _{L^2_t L^2}^2 +  \|\tau^{\frac{1}{2}} v \|_{L^\infty _t \dot{H}^1}^2   & \lesssim    \norm { v}_{L^2_t\dot{H}^1}^2 
\\
& \quad +   \norm {  u}_{L^\infty_t L^2} \norm { \rho_0}_{L^\infty}^\frac{1}{2} \| \tau^\frac{1}{2} \nabla  v\|_{ L^2_t 
L^\infty}\|\tau^\frac{1}{2} 
 \sqrt{\rho}\partial_t v  \|_{L^2_tL^2}\\ 
  & \lesssim   \norm { v}_{L^2_t\dot{H}^1}^2  + C_\varepsilon  \norm {  u}_{L^\infty_t L^2}^2 \norm { \rho_0}_{L^\infty}  \|\tau^\frac{1}{2} \nabla  v\|_{ L^2_t L^\infty}^2 
  \\
& \quad+  \varepsilon\|\tau^\frac{1}{2}  \sqrt{\rho}\partial_t v  \|_{L^2_tL^2}^2,
\end{aligned}
\end{equation*}
for some suitable small $\varepsilon>0$.
Accordingly, recalling that $\rho$ is uniformly bounded from below, we arrive at the bound 
\begin{equation*}
\|\tau^{ \frac{1}{2}} \partial_t v\|  _{L^2_t L^2}^2 +  \|\tau^{\frac{1}{2}} v  \|_{L^\infty _t \dot{H}^1}^2   \lesssim        \norm { v}_{L^2_t\dot{H}^1}^2  +  \norm {  u}_{L^\infty_t L^2}^2 \norm { \rho_0}_{L^\infty} \| \tau^\frac{1}{2}   v\|_{ L^2_t \dot{W}^{1, \infty}}^2 .
\end{equation*}
At last, we observe that the right-hand side above is finite  by virtue of \eqref{v:TW1} and the bounds in Proposition \ref{Thm-summary}, thereby completing the  proof of the proposition.
\end{proof}

As a by-product of   the bounds from the preceding proposition, we now deduce a corollary which establishes additional time-weighted estimates for $v$.

\begin{corollary}\label{RMK:*}
	Under the assumptions of  Proposition \ref{prop:v:tv1}, it holds that $$ t^\frac{1}{2} v\in L^2_{\loc}(\mathbb{R}^+;\dot{H}^2(\mathbb{R}^2))  , \qquad  t^{\frac{1}{2} } p_v \in L^2_{\loc}(\mathbb{R}^+;\dot{H}^1(\mathbb{R}^2)).$$
\end{corollary}

\begin{proof} 
We first apply Leray's projector  $\mathbb{P}\bydef \id - \nabla \Delta^{-1}\div$ to the equation  \eqref{v-equa*} to obtain  that 
$$  t^\frac{1}{2} \Delta v =  t^\frac{1}{2} \mathbb{P}(\rho \dot{v}) \in  L^2_{\loc}(\mathbb{R}^+;L^2(\mathbb{R}^2)),$$
as a direct consequence of the bounds established in Proposition \ref{prop:v:tv1}  and the  fact that $\mathbb{P}$ is bounded on $L^2(\mathbb{R}^2)$.

Additionally, we see from \eqref{v-equa*}  that 
$$ t^\frac{1}{2} \nabla p_v = t^\frac{1}{2} \Delta v  - t^\frac{1}{2}\rho \dot{v} \in  L^2_{\loc}(\mathbb{R}^+;L^2(\mathbb{R}^2)), $$
thereby completing the proof of the corollary.
\end{proof}
 
\begin{proposition}[in the spirit of {\cite[Proposition 3.5]{Danchin_Wang_22}}]
	\label{prop:v:tw2}
	It holds, under the assumptions of Proposition \ref{corollary:u-Lip}, that
	$$ t \dot{v} ,\, t \partial_t v  \in L^\infty _\loc(\mathbb{R}^+; L^2(\mathbb{R}^2)) \cap  L^2_\loc(\mathbb{R}^+; \dot{H}^1(\mathbb{R}^2))   .$$
\end{proposition}

\begin{proof}
First of all, observe that   Proposition  \ref{prop:v:tv1} and Corollary \ref{RMK:*} ensure that 
$$ t^\frac{1}{2} \nabla v \in  L^\infty _\loc(\mathbb{R}^+; L^2) \cap  L^2_\loc(\mathbb{R}^+; \dot{H}^1).$$
Therefore,  in view of \eqref{u:TW1} and the definition of $ \dot{v}$, one obtains that  
\begin{equation}\label{interm:ES1}
	t\dot{v}- t \partial_t v = t u\cdot \nabla v \in L^\infty _\loc(\mathbb{R}^+; L^2) \cap  L^2_\loc(\mathbb{R}^+; \dot{H}^1)  ,
\end{equation} 
thereby reducing the claims of Proposition \ref{prop:v:tw2}  to proving that 
$$ t  \partial_t v  \in L^\infty _\loc(\mathbb{R}^+; L^2) \cap  L^2_\loc(\mathbb{R}^+; \dot{H}^1)   .$$

To that end, we apply a time derivative to \eqref{v-equa*} to obtain that 
\begin{equation*} 
\rho \partial_{tt} v + \rho u \cdot \nabla \partial_t v - \Delta \partial_{t} v+ \nabla \partial_t p_v = - \partial _t \rho \dot{v} - \rho \partial_t u \cdot \nabla v.
\end{equation*}
Therefore,  multiplying both sides by $t^2\partial_t v$,  integrating over $\mathbb{R}^2$,   and   using the fact that $\div \partial_tv=0$,   we infer that 
\begin{equation}\label{Energy_v_tt}
 \begin{aligned}
\frac{1}{2} \frac{d}{dt} \int_{\mathbb{R}^2} \rho t^2    |\partial_t v|^2  +  \int_{\mathbb{R}^2}  t ^2  | \nabla \partial_t v|^2    &=    t \int_{\mathbb{R}^2} \rho   |\partial_t v|^2   -     t^2 \int_{\mathbb{R}^2}   \partial_t\rho  \dot{v} \cdot \partial_t v    
\\
& \quad
-    t^2  \int_{\mathbb{R}^2}    \rho(\partial_t u \cdot \nabla 
v )\cdot \partial_t v  \\
&  \bydef \sum_{ i=1}^3    \mathcal{N}_i(t).
\end{aligned} 
\end{equation} 
We  now take care of each summand separately.

\subsubsection*{Estimate of $\mathcal{N}_1$}

By H\"older's inequality and the maximum principle \eqref{mass:conservation}, we find that 
\begin{equation*}
\begin{aligned}
\mathcal{N}_1(t) \leq   \norm {\rho_0}_{L^\infty} \norm {t^{ \frac{1}{2}} \partial_t v(t)}_{  L^2}^2   .
\end{aligned}
\end{equation*}

\subsubsection*{Estimate of $\mathcal{N}_2$}

We first transform the time derivative into   spatial derivatives by utilizing   the continuity equation 
 $$\partial_t \rho = - \div (\rho u).$$
 Accordingly,   by integration by parts, we write  that 
 \begin{equation*}
\begin{aligned}
\mathcal{N}_2 (t)  &=     t^2   \int_{\mathbb{R}^2}  \div(\rho u)  \dot{v} \cdot \partial_t v d\tau\\
 &= - t ^2 \int_{\mathbb{R}^2}   (\rho u \cdot \nabla   \dot{v} ) \cdot \partial_t v  -  t^2   \int_{\mathbb{R}^2}  ( \rho u \cdot \nabla \partial_t v ) \cdot   \dot{v}  .
\end{aligned}
\end{equation*}
 Then, by H\"older's inequality and the maximum principle \eqref{mass:conservation}, it happens that  
  \begin{equation*}
\begin{aligned}
\mathcal{N}_2 (t)  & \leq   \norm { \rho_0}_{L^\infty}^{\frac{1}{2}} \norm {u(t)}_{L^\infty} \norm { t     \nabla\dot{v}(t)  }_{L^2}\norm { t \sqrt{ \rho} \partial_t v (t)}_{L^2}   \\
& \quad   +   \norm{\rho_0}_{L^\infty}   \norm { t ^{\frac{1}{2}}u (t) }_{L^\infty} \norm { t^\frac{1}{2} \dot{v} (t) }_{L^2}\norm {  t \nabla\partial_t v }_{L^2}  .
\end{aligned}
\end{equation*}
Therefore, by Young's inequality,  together with  the estimate $$ \norm { t   \nabla \dot{v} (t) }_{ L^2}^2\lesssim \norm { t  (\nabla \partial_t v  - \nabla \dot{v}) (t) }_{ L^2}^2+ \norm { t  \nabla \partial_t v  (t) }_{ L^2}^2,$$  
it then follows,  for any $\varepsilon \in (0,1)$, that
  \begin{equation*}
\begin{aligned} 
\mathcal{N}_2(t) & \leq    C_\varepsilon   \norm {u(t)}_{L^\infty} ^2 \norm { t  \sqrt{ \rho} \partial_t v (t)}_{L^2}^2   +  C_\varepsilon    \norm { t ^{\frac{1}{2}}u(t) }_{ L^\infty}^2 \norm {t ^{\frac{1}{2}}  \dot{v} (t) }_{  L^2}^2      \\
 &\quad+   \norm { t  (\nabla \partial_t v  - \nabla \dot{v}) (t) }_{ L^2}^2 +   \varepsilon\norm { t  \nabla \partial_t v  (t) }_{ L^2}^2 ,
\end{aligned}
\end{equation*}
for a constant $C_\varepsilon>0$ which also depends on $\norm{\rho_0}_{L^\infty}.$ Note, by choosing $\varepsilon$ small enough,  that the last term on the right-hand side can be absorbed by the left-hand side of \eqref{Energy_v_tt}.

\subsubsection*{Estimate of $\mathcal{N}_3$}

The bound on $\mathcal{N}_3 $ is obtained by, first,  applying H\"older's inequalities  to obtain that 
\begin{equation*}
\begin{aligned}
\mathcal{N}_3 (t)
&=-    t^2  \int_{\mathbb{R}^2}    \rho \nabla 
v : (\partial_t v\otimes \partial_t v)
-    t^2  \int_{\mathbb{R}^2}    \rho(\partial_t w \cdot \nabla 
v )\cdot \partial_t v 
\\
&\leq   \norm { \nabla v(t)}_{L^\infty} \norm {t     \sqrt{\rho} \partial_t v(t)}_{L^2}^2  \\
& \quad +   \norm {\rho_0}_{L^\infty}^\frac{1}{2}  t^{\frac{1}{2}} \norm {  \partial_t w(t) }_{ L^2} \norm { t^{\frac{1}{2}} \nabla v(t)}_{L^\infty}   \norm {t  \sqrt{\rho}\partial_t v(t)}_{L^2}   .
\end{aligned}
\end{equation*}
Then, further employing Young's  inequality and writing that 
$$ \begin{aligned}
 t^{\frac{1}{2}} \norm {  \partial_t w(t) }_{ L^2} \norm { t^{\frac{1}{2}} \nabla v(t)}_{L^\infty}    & \norm {t  \sqrt{\rho}\partial_t v(t)}_{L^2}
 \\ &\leq  t^{\frac{1}{2}} \norm {  \partial_t w(t) }_{ L^2} \norm { t^{\frac{1}{2}} \nabla v(t)}_{L^\infty}   \\
 &\quad   +  t^{\frac{1}{2}} \norm {  \partial_t w(t) }_{ L^2} \norm { t^{\frac{1}{2}} \nabla v(t)}_{L^\infty}   \norm {t  \sqrt{\rho}\partial_t v(t)}_{L^2}^2
\end{aligned}$$
yields
\begin{equation*}
\begin{aligned}
\mathcal{N}_3 (t)&\lesssim   \norm { \nabla v(t)}_{L^\infty} \norm {t     \sqrt{\rho} \partial_t v(t)}_{L^2}^2   +  t^{\frac{1}{2}} \norm {  \partial_t w(t) }_{ L^2} \norm { t^{\frac{1}{2}} \nabla v(t)}_{L^\infty}   \\
 &\quad   +  t^{\frac{1}{2}} \norm {  \partial_t w(t) }_{ L^2} \norm { t^{\frac{1}{2}} \nabla v(t)}_{L^\infty}   \norm {t  \sqrt{\rho}\partial_t v(t)}_{L^2}^2  .
\end{aligned}
\end{equation*}

\subsubsection*{Conclusion of the proof}

Now, by setting  
$$F_v(t) \bydef  \norm{  t   \sqrt{\rho}  \partial_t v(t)}_{ L^2}^2 + \int_0^t  \norm{  \tau    \partial_t v(\tau)}_{ \dot{H}^1}^2 d\tau  ,$$
gathering the foregoing estimates, and choosing $\varepsilon$ small enough, we infer from  \eqref{Energy_v_tt}, for all $t>0,$ that 
$$\frac{d}{dt} F_v (t) \lesssim  a_1(t)   +  a_2(t) F_v (t) ,  $$
where we denote
$$
\begin{aligned}
	a_1(t) &\bydef   \norm {t^{ \frac{1}{2}} \partial_t v(t)}_{  L^2}^2  +  \norm { t ^{\frac{1}{2}}u(t) }_{ L^\infty}^2 \norm {t ^{\frac{1}{2}}  \dot{v} (t) }_{  L^2}^2 +  \norm { t  (  \partial_t v  -  \dot{v}) (t) }_{ \dot{H}^1}^2 
  \\
& \quad+  t^{\frac{1}{2}} \norm {  \partial_t w(t) }_{ L^2} \norm { t^{\frac{1}{2}}   
v(t)}_{\dot{W}^{1,\infty}} 
\end{aligned} 
  $$ 
and 
$$a_2(t) \bydef  \norm {u(t)}_{L^\infty} ^2  +  \norm { \nabla v(t)}_{L^\infty} +  t^{\frac{1}{2}} \norm {  \partial_t w(t) }_{ L^2} \norm { t^{\frac{1}{2}}   v(t)}_{\dot{W}^{1,\infty}}   .$$

We emphasize here, by virtue of Lemma \ref{lemma:boundedness-u},    \eqref{partial_w}, \eqref{v:TW1}, \eqref{u:TW1},  Proposition \ref{prop:v:tv1} and \eqref{interm:ES1}, that  $ a_1,a_2\in  L^1_{\loc}( \mathbb{R}^+)$.  Therefore, the proof of the proposition is completed with a direct application of Gr\"onwall's lemma.
\end{proof} 
 
\begin{remark}
	The preceding result shows that $t\dot v$ belongs to $L^\infty_\loc(\mathbb{R}^+; L^2(\mathbb{R}^2)) $. Thus, following the proof of Corollary \ref{RMK:*}, we deduce that
	$$  t \Delta v =  t \mathbb{P}(\rho \dot{v}) \in  L^\infty_{\loc}(\mathbb{R}^+;L^2(\mathbb{R}^2)).$$
	That is to say,
\begin{equation}\label{RMK:H2}
 t  v \in L^\infty_\loc(\mathbb{R}^+; \dot{H}^2(\mathbb{R}^2)).
\end{equation} 
\end{remark}

\begin{proposition}[In the spirit of {\cite[Proposition 3.6]{Danchin_Wang_22}}]
	\label{prop:v:tv3}
	It holds, under the assumptions of Proposition \ref{corollary:u-Lip}, that 
$$ t  \dot{v}   \in  L^{q,1} _\loc(\mathbb{R}^+; \dot{W}^{2,p})  \cap L^{2} _\loc(\mathbb{R}^+; L^\infty)  .$$  
\end{proposition}

\begin{proof}  
   We proceed as in \cite{Danchin_Wang_22}. Thus, we first apply the operator $ \partial_t + u \cdot \nabla $ to the equation \eqref{v-equa*} to obtain   that 
      \begin{equation*}
  \rho \ddot{v}- \Delta \dot{v} + \nabla \dot{p}_v = \textbf{F},
  \end{equation*}
  where, we set 
  $$ \ddot{v} \bydef  ( \partial_t + u \cdot \nabla)\dot{v} , \qquad \textbf{F} \bydef  -\Delta u \cdot \nabla v - 2 \nabla u \cdot \nabla ^2 v + \nabla u \cdot \nabla p_v.$$

 Observe that we cannot apply Proposition \ref{M:reg:DW} at this stage, because $\div \dot{v}\neq 0$.
 Instead, we proceed by rewriting the preceding equation as 
   \begin{equation*}
  \rho \partial_t  \dot{v}- \Delta \dot{v}   =  \mathbb{P} \left( \textbf{F} - \rho u \cdot \nabla \dot{v} \right) + \mathbb{Q} \left(   \rho \partial_t  \dot{v}- \Delta \dot{v}  \right),
  \end{equation*}
  where we used the notation
  $$ \mathbb{Q} \bydef \nabla \Delta^{-1} \div, \qquad  \mathbb{P}= \text{Id} - \mathbb{Q}.$$
   Then, on the one hand,    the identity
  $$ \div \dot{v} =\nabla u^t:\nabla v= \text{Tr} \left( \nabla u \cdot \nabla v \right)$$
  entails that
  $$ \mathbb{Q} (\Delta \dot{v})=\nabla  \text{Tr} \left( \nabla u \cdot \nabla v \right).$$
  On the other hand, it is readily seen that
  $$ \mathbb{Q} \left(\rho \partial_t  \dot{v}  \right) =  \mathbb{Q} \big((\rho-1) \partial_t  \dot{v}  + \partial_tu \cdot\nabla v +  u \cdot\nabla \partial_tv\big).$$
  Accordingly, we arrive at the reformulation
 \begin{equation*} 
 \partial_t  \dot{v}- \Delta \dot{v}   =  \mathbb{P} \left( (1- \rho) \partial_t  \dot{v} - \rho u \cdot \nabla \dot{v}+  \textbf{F}  \right) + \mathbb{Q} \left(    \partial_tu \cdot\nabla v +  u \cdot\nabla \partial_tv \right) - \nabla  \text{Tr} \left( \nabla u \cdot \nabla v \right).
  \end{equation*}

By further  multiplying both sides by $t$, we find that
  \begin{equation*}
  \begin{aligned}
 \partial_t  (t\dot{v}) - \Delta (t\dot{v})    &=  \mathbb{P} \left( (1- \rho) \partial_t  (t\dot{v} ) + \rho \dot{v} - t \rho u \cdot \nabla \dot{v}+ t \textbf{F}  \right) \\
 & \qquad + \mathbb{Q} \left(  \dot{v} +   t \partial_tu \cdot\nabla v + tu \cdot\nabla \partial_tv \right) -t \nabla  \text{Tr} \left( \nabla u \cdot \nabla v \right) .
 \end{aligned}  
  \end{equation*}  
  We now utilize the  maximal regularity estimates for the heat equation  (see \cite[Proposition 2.1] {RM21}) to deduce the control
  \begin{equation}\label{Maximal_Regularity_heat}
  \begin{aligned}
\Vert \tau\dot{v} \Vert_{L_t^\infty \dot{B}_{p,1}^{-1+\frac 2p}}
&+ \Vert ( \partial_{ \tau} (\tau \dot{v}) ,\nabla^2  \tau \dot{v} ) \Vert_{L^{q,1}_t L^p}\\
&  \qquad \lesssim    \Vert (1- \rho) \partial_\tau  ( \tau \dot{v} )\Vert_{L^{q,1}_t L^p}    + \Vert  \rho \dot{v} - \tau \rho u \cdot \nabla \dot{v}+ \tau \textbf{F} \Vert_{L^{q,1}_t L^p}\\
&  \qquad  \quad + \Vert \dot{v} +  \tau \partial_\tau u \cdot\nabla v  + \tau u \cdot\nabla \partial_\tau v \Vert_{L^{q,1}_t L^p}
\\
& \quad  \qquad +\Vert \tau    \nabla v \otimes \nabla^2 u \Vert_{L^{q,1}_t L^p}
+\Vert \tau   \nabla u \otimes\nabla^2 v \Vert_{L^{q,1}_t L^p},
\end{aligned}
\end{equation}  
where we have used the boundedness of the operators $\mathbb{Q}$ and $\mathbb{P}$ on $L^p$ spaces, for $p\in (1,\infty),$ together with the fact that   $\frac 1p + \frac 1q = \tfrac 32$.

Now, observe that using H\"older's inequality, in the form
\begin{equation*}
\begin{aligned}
\Vert (1- \rho) \partial_\tau  (\tau\dot{v} ) \Vert_{L^{q,1}_t L^p}\leq&\, \norm {1-\rho}_{L^\infty_{t,x}}\Vert \partial _\tau (\tau\dot{v})\Vert  _{L^{q,1}_tL^p} ,
\end{aligned}
\end{equation*}
followed by \eqref{mass:conservation2} and \eqref{assumption:rho2} allows us to absorb the first term in the right-hand side of \eqref{Maximal_Regularity_heat} by the second term in its left-hand side. Hence, using the maximum principle  \eqref{mass:conservation}, again, 
we arrive at the bound  
\begin{equation*} 
  \begin{aligned}
\Vert \tau\dot{v}\Vert_{L_t^\infty \dot{B}_{p,1}^{-1+2/p}}  &+ \Vert \partial_\tau (\tau \dot{v}) ,\nabla^2 \tau\dot{v} \Vert_{L^{q,1}_t L^p} 
\\& \lesssim         \Vert   \dot{v} \Vert_{L^{q,1}_t L^p}   + \Vert   \tau   u \cdot \nabla \partial_\tau v \Vert _{L^{q,1}_t L^p} + \Vert   \tau   u \cdot \nabla  \dot{v} \Vert _{L^{q,1}_t L^p}  + \Vert    \tau \partial_\tau u \cdot\nabla v   \Vert_{L^{q,1}_t L^p}\\
& \quad+\Vert \tau     \nabla v \otimes \nabla^2 u  \Vert_{L^{q,1}_t L^p} + \Vert \tau \nabla u \otimes\nabla^2 v   \Vert_{L^{q,1}_t L^p} + \Vert   \tau \textbf{F} \Vert_{L^{q,1}_t L^p} .
\end{aligned}
\end{equation*}  
We shall  show that the right-hand side in the preceding estimate is finite.

To that end, by means of H\"older's inequalities, we write that 
\begin{equation*} 
\begin{aligned}
  \| \dot{v}\|_{L^{q,1}_tL^p}\lesssim&\, \| \partial_\tau v \|_{L^{q,1}_tL^p}+\| u\cdot \nabla v\|_{L^{q,1}_tL^p}\\
\lesssim&\, \|\partial_\tau v \|_{L^{q,1}_tL^p}+\Vert \nabla v\Vert_{L^2_t L^2}\Vert u\Vert_{L^{\eta,1}L^m},
\end{aligned}
\end{equation*}
which is finite by virtue of the bounds established in Lemma \ref{Energy_Estimate_v_w} and Proposition \ref{Thm-summary}.

Similarly, we find that 
\begin{equation*} 
  \begin{aligned} 
   \Vert   \tau   u \cdot \nabla \partial_\tau v \Vert _{L^{q,1}_t L^p} + \Vert   \tau   u \cdot \nabla  \dot{v} \Vert _{L^{q,1}_t L^p}  & \lesssim  \left(  \norm { \tau\partial _\tau v }_{ L^2_t \dot{H}^1}+ \norm { \tau \dot{v}}_{ L^2_t \dot{H}^1}\right) \norm { u}_{L^{\eta,1}_t L^m},
\end{aligned}
\end{equation*}  
where the right-hand side is finite due to Propositions \ref{Thm-summary} and  \ref{prop:v:tw2}.

Likewise, we obtain that  
\begin{equation*} 
  \begin{aligned} 
  \Vert   \tau \partial_\tau u \cdot\nabla v   \Vert_{L^{q,1}_t L^p}  & \lesssim     \norm {\tau \partial _\tau v }_{ L^\infty _t L^2}   \norm {  \nabla v}_{L^{q,1}_t L^m}
  \\
  & \lesssim     \norm { \tau\partial _\tau  v }_{ L^\infty _t L^2}   \norm {   v}_{L^{q,1}_t \dot{W}^{2,p}},
\end{aligned}
\end{equation*}  
where we employed the Sobolev embedding $\dot{W}^{1,p} \hookrightarrow L^m(\mathbb{R}^2)$, for $ \frac{1}{p} = \frac{1}{2} + \frac{1}{m}$.
Again,  the right-hand side above is finite by virtue of Propositions  \ref{Thm-summary} and \ref{prop:v:tw2}.

In a similar way, we write that 
\begin{equation*} 
  \begin{aligned}  
  \Vert \tau     \nabla v \otimes \nabla^2 u  \Vert_{L^{q,1}_t L^p} + \Vert   \tau \nabla u \otimes\nabla^2 v   \Vert_{L^{q,1}_t L^p} & \lesssim      \| {\tau  \nabla v\|}_{ L^\infty_t L^\infty}      \norm {\nabla^2   u}_{L^{q,1}_t L^p} 
  \\
& \quad  +   \| { \tau  \nabla ^2 v \|}_{ L^\infty _t L^2}  \norm {\nabla  u}_{L^{ 
q,1}_t 
L^m} \\
   & \lesssim   \left(   \| {\tau  v\|}_{ L^\infty_t \dot{W}^{1,\infty}}     +   \| { \tau    v \|}_{ L^\infty _t \dot{H}^2} \right)  \norm {   u}_{L^{ q,1}_t \dot{W}^{2,p}},
\end{aligned}
\end{equation*}   
which is finite, thanks to    \eqref{RMK:H2} and the bounds from   Proposition \ref{corollary:u-Lip}.
  
At last, expanding the definition of $\textbf{F}$ and employing H\"older and embedding inequalities, we infer that 
\begin{equation*}
\begin{aligned}
\Vert \tau \textbf{F}\Vert_{L^{q,1}_t L^p}\leq&\, \Vert  \tau v\Vert_{L^\infty_t \dot{W}^{1,\infty}}\Vert   u\Vert_{L^{q,1}_t \dot{W}^{2,p}}+ \Vert  u\Vert_{L^2_t\dot{H}^1}\Big(\Vert  \tau  v\Vert_{L^{\eta,1}_t \dot{W}^{2,m}}+ \Vert  \tau \nabla p_v\Vert_{L^{\eta,1}_t L^m}\Big).
\end{aligned}
\end{equation*}
It is then readily seen that the right-hand side above is finite due to  Proposition \ref{Thm-summary}, Lemma \ref{lemma:v2} and Remark \ref{RMK-corollary1}, thereby  showing that all terms in \eqref{Maximal_Regularity_heat} are finite.
We conclude that
$t  \dot{v}   \in  L^{\infty} _\loc(\mathbb{R}^+; \dot{B}_{p,1}^{-1+2/p})\cap L^{q,1} _\loc(\mathbb{R}^+; \dot{W}^{2,p})$. 

As for the $L^2_tL^\infty$ bound on $t \dot{v}$, it follows from the interpolation inequality 
\begin{equation*}
\|t\dot{v}\|_{L^2_t L^\infty }\lesssim  \|t\dot{v}\|_{L^\infty_t \dot{B}_{p,1}^{-1+\frac{2}{p}}}^{\frac{2-q}{2}}\|t \dot{v}\|_{L^{q,1}_t \dot{W}^{2,p} }^{\frac{q}{2}},
\end{equation*}
 which completes the proof of the proposition.
\end{proof}

\subsection{Control of $w$}\label{subsection:w-TW}

Note that the uniqueness result granted by Theorem \ref{Thm:stability} requires less space-regularity on $ \dot{w}$ than it does on $\dot v$,  but it necessitates a stronger control on the singularity of $w(t)$ at $t=0$. However, observe that we do not expect $w(t)$ to blow up as $t\to 0$, because its initial data is identically zero.

Before we get into the actual control of $w$, we first establish the following lemma which plays a crucial role in estimating the size of $\dot{w}$.  Specifically, this result essentially shows how to control  $\left\langle \partial_t (j\times B), \partial_t w \right\rangle_{L^2}$ under the assumption that   the electromagnetic fields enjoy some Sobolev regularity.

\begin{lemma}\label{lemma:duality}
For any $\varepsilon>0$, there is $C_\varepsilon>0$ such that, for all $t>0$,
\begin{equation*}
\begin{aligned}
t \bigg|\int_{\mathbb{R}^2} \partial_t (j\times B) &\partial_t w (t,x)dx\bigg|\\
& \leq C_\varepsilon \mathcal{P_*}(t)\left( 1+    \norm{  t^{\frac{1}{2} }  \partial_t w(t)}_{ L^2} \right)+ C_\varepsilon \norm {B(t)}_{\dot{H}^ \frac{1}{2}}^4 \norm{  t^{\frac{1}{2} }  \partial_t w(t)}_{ L^2}^2 
\\
& \quad   + \varepsilon   \norm{  t^{\frac{1}{2} }  \partial_t w(t)}_{ \dot{H}^{1}} ^2 ,
\end{aligned}
\end{equation*} 
where
$$ \begin{aligned}
\mathcal{P}_*(t) & \bydef   (1+ t)  \Big( c^4 + c^2 \norm { u(t)}_{L^\infty \cap \dot{H}^1} ^2 + \norm{  t^{\frac{1}{2}} \partial_t v (t)}_{L^2}^2     \Big)    \norm{ (E,B)(t)}_{ \dot{H}^s}^4 
\\
& \quad   +  c^2 t^{\frac{1}{2} }  \norm{ j(t)}_{L^4} \norm{ B(t)}_{\dot{H}^\frac{1}{2}} +  1   ,
\end{aligned} $$ 
for any $s\in (\frac{1}{2},1)$.
\end{lemma}

\begin{remark}
The estimate in the preceding lemma is  not uniform with respect to the speed of light $c\in (0,\infty)$. However,  note that  this bound will only  be employed for uniqueness and stability, namely to prove Theorem \ref{Thm:stability}. In particular, this
does not affect the uniform character, in $c\in (0,\infty)$, of the global solutions given in Theorem \ref{Thm:1}.
\end{remark}

\begin{remark}\label{RMK:P*}
We emphasize that $\mathcal{P}_*\in L^1_\loc(\mathbb{R}^+)$. This, indeed, can be verified by utilizing  Ohm's law to expand the current density $j$ and then  employing the bounds from Propositions \ref{Thm-summary} and    \ref{prop:v:tv1}.
\end{remark}

\begin{proof}
	To begin, we  utilize \eqref{Main_System} to write that
$$ \begin{aligned}
 \partial_t (j\times B) & = \partial_t j \times B +j \times \partial_t B \\
 &= c\sigma \partial_t E   \times B + \sigma (\partial_t u \times B)   \times B + \sigma (u \times \partial_t B)  \times B - c j \times (\nabla \times E)\\
 &= c^2 \sigma (\nabla \times B)   \times B - c^2 \sigma j \times B + \sigma (\partial_t v \times B)   \times B + \sigma (\partial_t w \times B)   \times B \\
 & \quad - c \sigma (u \times (\nabla \times E))  \times B 
 - c ^2 \sigma E  \times (\nabla \times E) - c\sigma (u\times B)  \times (\nabla \times E)\\
 &\bydef \sum_{i=1}^7 \mathcal{M}_i.
\end{aligned}$$
Now, we    estimate   each term in the preceding expansion separately.

\subsubsection*{Bound on $ \mathcal{M}_1$}

We proceed by  duality and by employing the product law \eqref{Prod_Law}  to obtain,  for any $s\in (\frac{1}{2},1)$ and $t>0$, that 
$$ \begin{aligned}
t  \bigg| \int_{\mathbb{R}^2} \mathcal{M}_1(t,x) \partial_t w (t,x)dx\bigg| &\lesssim  c^2t  \norm{ B(t)}_{ \dot{H}^s} \norm{ B  \partial_t w(t)}_{ \dot{H}^{ 1-s}} \\ 
&\lesssim  c^2t^\frac{1}{2}  \norm{ B(t)}_{ \dot{H}^s}^2 \norm{  t^{\frac{1}{2} } \partial_t w(t)}_{ \dot{H}^{2( 1-s)}} .  
\end{aligned} $$
Therefore, we find, by interpolation,   for any $\varepsilon>0$, that 
$$ \begin{aligned}
t   \bigg|\int_{\mathbb{R}^2} \mathcal{M}_1(t,x) \partial_t w (t,x)dx\bigg|
&\lesssim  c^2t^{\frac{1}{2}}  \norm{ B(t)}_{ \dot{H}^s}^2 \norm{  t^{\frac{1}{2} }  \partial_t w(t)}_{ L^2}^{2s-1} \norm{  t^{\frac{1}{2} }  \partial_t w(t)}_{ \dot{H}^{1}}^{2(1-s)} \\
&\leq \varepsilon \norm{  t^{\frac{1}{2} }  \partial_t w(t)}_{ \dot{H}^{1}}^2  + C_\varepsilon \left(c^2t^{\frac{1}{2} }  \norm{ B(t)}_{ \dot{H}^s}^2  \right)^{\frac{1}{s}}   \norm{  t^{\frac{1}{2} }  \partial_t w(t)}_{ L^2}^{\frac{2s-1}{s} }.
\end{aligned} $$
Hence,  in view of fact that  $  \frac{2s-1 }{s}<1$ and $\frac{1}{s}<2$,  together with   the  elementary inequality 
\begin{equation}\label{elementary_inequa}
z^\nu\leq z+1 ,
\end{equation}
 which is valid for all $ \nu \in (0,1]$ and $z\geq 0$, we deduce that  
$$\begin{aligned}
 t \bigg|  \int_{\mathbb{R}^2} \mathcal{M}_1(t,x) \partial_t w (t,x)dx \bigg| &\leq \varepsilon \norm{  t^{\frac{1}{2} }  \partial_t w(t)}_{ \dot{H}^{1}}^2  + C_\varepsilon \left(c^2t^{\frac{1}{2} }  \norm{ B(t)}_{ \dot{H}^s}^2  +1\right)^{2}    \\
& + C_\varepsilon   \left(c^2t^{\frac{1}{2} }  \norm{ B(t)}_{ \dot{H}^s}^2 +1  \right)^{2}    \norm{  t^{\frac{1}{2} }  \partial_t w(t)}_{ L^2}.
\end{aligned}$$
\subsubsection*{Bound on $ \mathcal{M}_2$} The control of $ \mathcal{M}_2$ follow from a simple application of H\"older's inequality to find that 
$$ \begin{aligned}
t  \bigg|\int_{\mathbb{R}^2} \mathcal{M}_2(t,x) \partial_t w (t,x)dx\bigg| & \lesssim  c^2 t  \norm{ j(t)}_{L^4} \norm{ B(t)}_{L^4} \norm{    \partial_t w(t)}_{ L^2} \\
& \lesssim  c^2 t^{\frac{1}{2} }  \norm{ j(t)}_{L^4} \norm{ B(t)}_{\dot{H}^\frac{1}{2}} \norm{ t^{\frac{1}{2} }  \partial_t w(t)}_{ L^2} . 
\end{aligned} $$
\subsubsection*{Bound on $ \mathcal{M}_3$} By employing the same arguments as before, we obtain,  as soon as  $s\in (\frac{1}{2},1)$,  that 
$$ \begin{aligned}
t  \bigg| \int_{\mathbb{R}^2} \mathcal{M}_3(t,x) \partial_t w (t,x)dx\bigg| &\lesssim   t   \norm{ \partial_t v (t)}_{L^2}  \norm{ B \otimes B \otimes   \partial_t w(t)}_{L^2}  \\
&\lesssim    t   \norm{ \partial_t v (t)}_{L^2}  \norm{ B \otimes B(t)}_{ \dot{H}^{ 2s-1}}\norm { \partial_t w(t)}_{ \dot{H}^{ 2 (1-s)} }  \\ 
&\lesssim    \norm{t^\frac{1}{2} \partial_t v (t)}_{L^2}  \norm{  B(t)}_{ \dot{H}^{ s}}^2\norm {t^{\frac{1}{2}} \partial_t w(t)}_{ \dot{H}^{ 2 (1-s)}  } \\
&\lesssim     \norm{ t^\frac{1}{2} \partial_t v (t)}_{L^2}  \norm{  B(t)}_{ \dot{H}^{ s}}^2\norm {t^\frac{1}{2} \partial_t w(t)}_{L^2} ^{2s-1} \norm {t^{\frac{1}{2}} \partial_t w(t)}_{ \dot{H}^{ 1 } } ^{2(1-s)}  \\
&\leq  C_{\varepsilon}    \norm{  t^{\frac{1}{2} } \partial_t v (t)}_{L^2} ^{\frac{1}{s} }  \norm{  B(t)}_{ \dot{H}^{ s}}^{\frac{2}{s} }\norm {t^{\frac{1}{2} }\partial_t w(t)}_{L^2} ^ {\frac{2s-1}{s} }  \\
& \quad  +  \varepsilon   \norm { t^{\frac{1}{2} }\partial_t w(t)}_{ \dot{H}^{ 1 } } ^2 ,
\end{aligned} $$ 
where $\varepsilon>0$.
Hence, by employing again the  elementary inequality \eqref{elementary_inequa} with the facts that $\frac{2s-1}{s}<1 $ and $ \frac{1}{s}<2$,  we deduce that     
\begin{equation*}
 \begin{aligned}
t  \bigg| \int_{\mathbb{R}^2} \mathcal{M}_3(t,x) \partial_t w (t,x)dx  \bigg|
&\leq   \varepsilon   \norm { t^{\frac{1}{2} }\partial_t w(t)}_{ \dot{H}^{ 1 } } ^2+ C_{\varepsilon}  \left(  \norm{  t^{\frac{1}{2}} \partial_t v (t)}_{L^2}^2   \norm{  B(t)}_{ \dot{H}^{ s}}^{4}+ 1\right)\\
& \quad + C_{\varepsilon}    \left(  \norm{  t^{\frac{1}{2}} \partial_t v (t)}_{L^2}^2   \norm{  B(t)}_{ \dot{H}^{ s}}^{4}+ 1\right) \norm {t^{\frac{1}{2} }\partial_t w(t)}_{L^2}     . 
\end{aligned} 
\end{equation*}

\subsubsection*{Bound on $ \mathcal{M}_4$}
Likewise, by H\"older's inequality, Sobolev embedding and the two dimensional interpolation inequality   
\begin{equation*} 
\Vert f \Vert_{L^4}\lesssim  	\Vert f \Vert_{ L^2}^{ \frac{1}{2}} \Vert  \nabla f\Vert_{L^2}^{ \frac{1}{2}}  ,
\end{equation*}
we find that
$$ \begin{aligned}
t   \bigg|\int_{\mathbb{R}^2} \mathcal{M}_4(t,x) \partial_t w (t,x)dx \bigg|& \lesssim   \int_{\mathbb{R}^2}| B (t,x) |^2 | t^{\frac{1}{2} } \partial_t w (t,x) |^2dx \\
& \lesssim   \norm {  B (t)}_{L^4}^2 \norm { t^{\frac{1}{2} } \partial_t w    (t)}_{L^4}^2 \\
& \lesssim   \norm {  B (t)}_{L^4}^2 \norm { t^{\frac{1}{2} } \partial_t w    (t)}_{L^2}  \norm { t^{\frac{1}{2} } \partial_t w    (t)}_{\dot{H}^1} .
\end{aligned} $$
Thus, it follows, for any $\varepsilon>0$,   that
$$
t \bigg| \int_{\mathbb{R}^2} \mathcal{M}_4(t,x) \partial_t w (t,x)dx\bigg|
\leq   \varepsilon   \norm { t^{\frac{1}{2} }\partial_t w(t)}_{ \dot{H}^{ 1 } } ^2+ C_{\varepsilon} \norm { B(t)}_{ \dot{H}^\frac{1}{2}} ^{ 4}  \norm { t^{\frac{1}{2} } \partial_t w    (t)}_{L^2}^2. 
 $$
\subsubsection*{Bounds on $ \mathcal{M}_5$ and $ \mathcal{M}_7$}

The estimates on $\mathcal{M}_5$ and  $\mathcal{M}_7$ are done in a  similar fashion, whence we outline them together. To that end,  note first that  the assumption $s\in (\frac{1}{2},1)$ implies that   $2(1-s)\in (0,1)$. Accordingly, by duality and  employing the  
product  law \eqref{Prod_Law}, we find that
$$ \begin{aligned}
 t  \bigg|\int_{\mathbb{R}^2} \Big( \mathcal{M}_5(t,x) &+ \mathcal{M}_7(t,x) \Big) \partial_t w (t,x)dx\bigg| \\
 & \lesssim   tc   \norm {\nabla \times E(t)}_{ \dot{H}^{s-1}} \norm { u\otimes  B \otimes \partial_t w(t)}_{ \dot{H}^{1-s}} \\ 
& \lesssim    t^{\frac{1}{2} } c \norm {  E(t)}_{ \dot{H}^{s }} \norm { u\otimes  B(t)}_{ \dot{H}^s} \norm{  t^{\frac{1}{2} }  \partial_t w(t)}_{ \dot{H}^{2(1-s)}}  \\ 
  & \lesssim    t^{\frac{1}{2} } c \norm {  E(t)}_{ \dot{H}^{s }} \norm { u(t)}_{L^\infty \cap \dot{H}^1}\norm { B(t)}_{ \dot{H}^s}   \norm{  t^{\frac{1}{2} }  \partial_t w(t)}_{ \dot{H}^{2(1-s)}},
\end{aligned} $$ 
where  we have also used the  product law   \eqref{Prod_Law2} in the last step. Therefore, by interpolation, it holds, for any $\varepsilon>0$, that
$$ \begin{aligned}
 t \bigg| \int_{\mathbb{R}^2} &\Big( \mathcal{M}_5(t,x) + \mathcal{M}_7(t,x) \Big) \partial_t w (t,x)dx \bigg|\\ 
  & \lesssim    t^{\frac{1}{2} } c \norm {  E(t)}_{ \dot{H}^{s }} \norm { u(t)}_{L^\infty \cap \dot{H}^1}\norm { B(t)}_{ \dot{H}^s} \norm{  t^{\frac{1}{2} }  \partial_t w(t)}_{ L^2}^{2s-1} \norm{  t^{\frac{1}{2} }  \partial_t w(t)}_{ \dot{H}^{1}}^{2(1-s) }  \\ 
  & \leq  \varepsilon   \norm{  t^{\frac{1}{2} }  \partial_t w(t)}_{ \dot{H}^{1}} ^2 +  C_{\varepsilon}  \left( t^{\frac{1}{2} } c \norm {  E(t)}_{ \dot{H}^{s }} \norm { u(t)}_{L^\infty \cap \dot{H}^1}\norm { B(t)}_{ \dot{H}^s} \right)^{\frac{1}{s}}  \norm{  t^{\frac{1}{2} }  \partial_t w(t)}_{ L^2}^{\frac{2s-1}{s}} . 
\end{aligned} $$
At last,   we deduce, since $ \frac{2s-1}{s}<1$ and $\frac{1}{s}<2$,  that 
$$ \begin{aligned}
 t \bigg|  \int_{\mathbb{R}^2} \Big( \mathcal{M}_5(t,x) &+ \mathcal{M}_7(t,x) \Big) \partial_t w (t,x)dx\bigg| \\  
  & \leq  \varepsilon   \norm{  t^{\frac{1}{2} }  \partial_t w(t)}_{ \dot{H}^{1}} ^2 +  C_{\varepsilon}  \left( t^{\frac{1}{2} }c  \norm {  E(t)}_{ \dot{H}^{s }} \norm { u(t)}_{L^\infty \cap \dot{H}^1}\norm { B(t)}_{ \dot{H}^s} + 1 \right)^{2} \\
  & \quad  +  C_{\varepsilon}  \left( t^{\frac{1}{2} } c \norm {  E(t)}_{ \dot{H}^{s }} \norm { u(t)}_{L^\infty \cap \dot{H}^1}\norm { B(t)}_{ \dot{H}^s}+ 1 \right)^{2}  \norm{  t^{\frac{1}{2} }  \partial_t w(t)}_{ L^2}  . 
\end{aligned} $$
\subsubsection*{Bound on $\mathcal{M}_6$}

Finally, for  $\mathcal{M}_6$,  we follow the procedure by which we previously controled $ \mathcal{M}_1$ to find that  
$$ \begin{aligned}
 t \bigg|\int_{\mathbb{R}^2} \mathcal{M}_6(t,x) \partial_t w (t,x)dx \bigg| & \leq  \varepsilon   \norm{  t^{\frac{1}{2} }  \partial_t w(t)}_{ \dot{H}^{1}} ^2 +  C_{\varepsilon}  \left(c^2 t^{\frac{1}{2} }  \norm {  E(t)}_{ \dot{H}^{s }}^2  + 1\right)^{2} \\
 & \quad +  C_{\varepsilon}  \left(c^2 t^{\frac{1}{2} }  \norm {  E(t)}_{ \dot{H}^{s }}^2 + 1 \right)^{2}  \norm{  t^{\frac{1}{2} }  \partial_t w(t)}_{ L^2}  .   
\end{aligned} $$
All in all, gathering the foregoing estimates completes the proof of the lemma.  
\end{proof}

We are now in a position to prove the final time-weighted estimates for the velocity field $w$. In particular, the  following proposition is analogous to Proposition \ref{prop:v:tw2}, but displays an improved time weight in the neighborhood of $ t= 0$.
\begin{proposition}\label{prop:w:tw**}  Under the assumptions of Proposition \ref{corollary:u-Lip}, it holds that 
$$t^{\frac{1}{2} } \dot{w} ,\, t^{\frac{1}{2} } \partial_t w  \in L^\infty_\loc(\mathbb{R}^+; L^2(\mathbb{R}^2)) \cap L^2_\loc(\mathbb{R}^+; \dot{H}^1(\mathbb{R}^2)).$$
\end{proposition}

\begin{proof}
First, by the definition   $ \dot{w}\bydef\partial_t w+u\cdot \nabla w$,  and by virtue of the bounds \eqref{partial_w} and \eqref{u:TW1}, which produce the control
$$ t^{\frac{1}{2}} u\in L^\infty_{\loc}(\mathbb{R}^+; L^\infty) \cap  L^2_{\loc} (\mathbb{R}^+; \dot{W}^{1,\infty})  \quad\text{and}\quad w\in L^\infty_{\loc}(\mathbb{R}^+; \dot{H}^1) \cap L^2_{\loc}(\mathbb{R}^+; \dot{H}^2),  $$
 we deduce from the basic estimates
$$  \begin{aligned}
 \norm { t^{\frac{1}{2} } \dot{w}(t) - t^{\frac{1}{2} } \partial_t w (t)}_{ L^2} \leq  \norm { t^\frac{1}{2} u(t) }_{L^\infty } \norm {\nabla  w(t)}_{ L^2}  
\end{aligned}$$
and
\begin{equation*}
 \begin{aligned}
 \|{ t^{\frac{1}{2} }  \dot{w} -  t^{\frac{1}{2} } \partial_t   w (t)}\|_{ \dot{H}^1} \leq      \| { t^\frac{1}{2} \nabla u (t)}\|_{  L^\infty} \| {\nabla  w(t)}\|_{  L^2}+   \|{ t^\frac{1}{2} u(t) }\|_{L^\infty } \| {\nabla^2  w(t)}\|_{  L^2} 
\end{aligned} 
\end{equation*}
that 
\begin{equation}\label{w-dw}
	   t^{\frac{1}{2} }  \dot{w} -  t^{\frac{1}{2} } \partial_t   w   \in   L^\infty_\loc(\mathbb{R}^+; L^2) \cap L^2_\loc(\mathbb{R}^+; \dot{H}^1).
\end{equation}
  Thus, there only remains to prove that 
$$ t^{\frac{1}{2} } \partial_t w \in L^\infty_\loc(\mathbb{R}^+; L^2) \cap L^2_\loc(\mathbb{R}^+; \dot{H}^1).$$
\subsubsection*{Energy estimate}

We first apply a time derivative   to the equation  \eqref{w-equa*} to find that
$$ \rho \partial_{tt} w + \rho u\cdot \nabla  \partial_t w - \partial_t \Delta w  + \nabla \partial_t p_w = - \partial_t\rho  \dot{w} - \rho\partial_t u \cdot \nabla w + \partial_t(j\times B) .$$ 
 Therefore, by taking the inner product with $ t \partial_t w$, we obtain   that 
$$ \begin{aligned}
\frac{1}{2} \frac{d}{dt} \int_{\mathbb{R}^2} \rho t    |\partial_t w|^2  +  \int_{\mathbb{R}^2}  t   | \nabla \partial_t w|^2    &=   \int_{\mathbb{R}^2} \rho   |\partial_t w|^2  -     t \int_{\mathbb{R}^2}   \partial_t\rho  \dot{w} \partial_t w  \\
& \quad -    t  \int_{\mathbb{R}^2}    \rho(\partial_t u \cdot \nabla w ) \partial_t w +   t  \int_{\mathbb{R}^2}\partial_t  (    j\times B )\cdot \partial_t w\\
&  \bydef \sum_{ i=1}^4    \mathcal{K}_i(t).
\end{aligned} $$
Now, we  take care of  each term in the previous sum separately. 
\subsubsection*{Bounds on $\mathcal{K}_1$}
   By H\"older's inequality and the maximum principle \eqref{mass:conservation}, we write that 
\begin{equation*}   
\begin{aligned}
\mathcal{K}_1(t) \leq   \norm {\rho_0}_{L^\infty} \norm {\partial_t w(t)}_{  L^2}^2   .
\end{aligned}
\end{equation*}
\subsubsection*{Bounds on $\mathcal{K}_2$} In order to estimate the second term $\mathcal{K}_2$, we  first  write, by utilizing   the continuity equation
 $$\partial_t \rho = - \div (\rho u)$$
    and performing an integration by parts, that  
 \begin{equation*}
\begin{aligned}
\mathcal{K}_2 (t)  &=       t   \int_{\mathbb{R}^2}  \div(\rho u)  \dot{w} \partial_t w d\tau\\
 &=  -t  \int_{\mathbb{R}^2}   \rho u \cdot \nabla   \dot{w} \partial_t w  -  t   \int_{\mathbb{R}^2}   \rho u \cdot \nabla \partial_t w    \dot{w}  .
\end{aligned}
\end{equation*}
 Then, by H\"older's inequality and the maximum principle \eqref{mass:conservation},  we find that  
  \begin{equation*}
\begin{aligned}
\mathcal{K}_2 (t)  & \leq   \norm { \rho_0}_{L^\infty}^{\frac{1}{2}} \norm {u(t)}_{L^\infty} \norm { t^{\frac{1}{2} }   \nabla\dot{w}(t)  }_{L^2}\norm { t^{\frac{1}{2} }\sqrt{ \rho} \partial_t w (t)}_{L^2}   \\
& \quad   +   \norm{\rho_0}_{L^\infty}   \norm { t ^{\frac{1}{2}}u (t) }_{L^\infty} \norm {  \dot{w} (t) }_{L^2}\norm {  t^{\frac{1}{2} }\nabla\partial_t w }_{L^2}  .
\end{aligned}
\end{equation*}
Therefore, by virtue of the fact that  
\begin{equation*}
\norm { t^{\frac{1}{2} }   \nabla\dot{w}(t)  }_{L^2}^2\leq 2\norm{ t^{\frac{1}{2} } (\nabla \partial_t w  - \nabla \dot{w}) (t) }_{ L^2}^2+2\norm { t^{\frac{1}{2} } \nabla \partial_t w  (t) }_{ L^2}^2,
\end{equation*} 
 it then follows, for any $\varepsilon \in (0,1)$, that
  \begin{equation*}
\begin{aligned} 
\mathcal{K}_2(t) & \leq     C_\varepsilon    \norm { t ^{\frac{1}{2}}u(t) }_{ L^\infty}^2 \norm {  \dot{w} (t) }_{  L^2}^2     +  \norm { t^{\frac{1}{2} } (  \partial_t w  -  \dot{w}) (t) }_{ \dot{H}^1}^2  +  C_\varepsilon   \norm {u(t)}_{L^\infty} ^2 \norm { t^{\frac{1}{2} } \sqrt{ \rho} \partial_t w (t)}_{L^2}^2     
\\
&\quad +   \varepsilon\norm { t^{\frac{1}{2} }   \partial_t w  (t) }_{ \dot{H}^1}^2 .
\end{aligned}   
\end{equation*}   
Above, the constant $C_\varepsilon>0$ also depends on $\norm{\rho_0}_{L^\infty}.$
\subsubsection*{Bounds on $\mathcal{K}_3$}
The bound on  $\mathcal{K}_3 $ is obtained by employing H\"older's inequalities and  the maximum principle \eqref{mass:conservation}  together with the fact that $u=v+ w$ to write that 
\begin{equation*}
\begin{aligned}
\mathcal{K}_3 (t)&\leq   \norm { \nabla w(t)}_{L^\infty} \norm {t^{\frac{1}{2} }  \sqrt{\rho} \partial_t w(t)}_{L^2}^2  
\\ &\quad +   \norm {\rho_0}_{L^\infty}^\frac{1}{2}\| {   t^{\frac{1}{2}} \partial_t v(t) }\|_{ L^2} \norm {\nabla w(t)}_{L^\infty}   \norm {t^{\frac{1}{2} }  \sqrt{\rho} \partial_t w(t)}_{L^2}   .
\end{aligned}
\end{equation*}
Therefore, by further employing Young's inequality and absorbing the bound on $\rho_0$ into the implicit constant, we infer that 
\begin{equation*}
\begin{aligned}
\mathcal{K}_3 (t)&\lesssim      \| {   t^{\frac{1}{2}} \partial_t v(t) } \|_{ L^2} \norm {  w(t)}_{\dot{W}^{1,\infty}} 
\\ &\quad +  \left( 1+ \| {   t^{\frac{1}{2}} \partial_t v(t) }\| _{ L^2}   \right)  \norm {  w(t)}_{ \dot{W}^{1,\infty}} \norm {t^{\frac{1}{2} }  \sqrt{\rho} \partial_t w(t)}_{L^2}^2      .
\end{aligned}
\end{equation*}
\subsubsection*{Bounds on $\mathcal{K}_4$} Finally, we note that the bound on $\mathcal{K}_4$ has already been established in Lemma \ref{lemma:duality}, whence, we write, for any $\varepsilon>0$,  that 
 \begin{equation*}
\begin{aligned}
\mathcal{K}_4(t)  & \leq C_\varepsilon \mathcal{P_*}(t)\left( 1+    \norm{  t^{\frac{1}{2} }  \partial_t w(t)}_{ L^2} \right)
\\ &\quad+ C_\varepsilon \norm {B(t)}_{\dot{H}^ \frac{1}{2}}^4 \norm{  t^{\frac{1}{2} }  \partial_t w(t)}_{ L^2}^2 +  \varepsilon   \norm{  t^{\frac{1}{2} }  \partial_t w(t)}_{ \dot{H}^{1}} ^2   ,
\end{aligned}
\end{equation*}
where $\mathcal{P}_*(t)$ is   defined in the statement of the lemma. Thus, the preceding bound can be simplified by Young's inequality to deduce that 
\begin{equation*}
\begin{aligned}
\mathcal{K}_4(t)  & \leq C_\varepsilon \mathcal{P_*}(t) + C_\varepsilon\left( \mathcal{P_*}(t) +     \norm {B(t)}_{\dot{H}^ \frac{1}{2}}^4  \right)  \norm{  t^{\frac{1}{2} }  \partial_t w(t)}_{ L^2}^2  +  \varepsilon   \norm{  t^{\frac{1}{2} }  \partial_t w(t)}_{ \dot{H}^{1}} ^2   .
\end{aligned}
\end{equation*}
\subsubsection*{Conclusion of the proof} Now, we show how to combine the previous bounds to conclude the proof. To that end, we  first introduce the function of time
$$F_w (t)\bydef \left( \norm{  t^{\frac{1}{2} } \sqrt \rho \partial_t w(t)}_{ L^2}^2 + \int_0^t  \norm{  \tau ^{\frac{1}{2} }  \partial_t w(\tau)}_{ \dot{H}^1}^2 d\tau \right)^{\frac{1}{2}}.$$
 Then, by gathering the foregoing estimates and  choosing $\varepsilon$ small enough, we arrive at the bound  
 \begin{equation}\label{last_inequa}
 \begin{aligned}
 \frac{d}{dt} F_w^2(t) \lesssim   \mathcal{G}_*(t)  + \mathcal{G}(t) F^2_w(t),
 \end{aligned}
 \end{equation}
 where we set that 
  \begin{equation*}
 \begin{aligned} 
  \mathcal{G}_*(t) &\bydef     \mathcal{P}_*(t)+ \norm {\partial_t w(t)}_{  L^2}^2+   \norm {   t^{\frac{1}{2}} \partial_t v(t) }_{ L^2} \norm { w(t)}_{\dot{W}^{1,\infty}}  \\
  & \quad +     \norm { t ^{\frac{1}{2}}u(t) }_{ L^\infty}^2 \norm {  \dot{w} (t) }_{  L^2}^2     +  \norm { t^{\frac{1}{2} } (  \partial_t w  -  \dot{w}) (t) }_{ \dot{H}^1}^2 
 \end{aligned}
 \end{equation*}
 and
  \begin{equation*}
 \begin{aligned}  
    \mathcal{G} (t)&\bydef    \mathcal{P}_*(t)+   \norm {B(t)}_{\dot{H}^ \frac{1}{2}}^4   + \norm { u(t)}_{L^\infty}^2  +    \left( 1+ \| {   t^{\frac{1}{2}} \partial_t v(t) }\| _{ L^2}   \right)\norm {   w(t)}_{\dot{W}^{1,\infty}}  .
 \end{aligned}
 \end{equation*}
Therefore,  observing  that \eqref{last_inequa} yields that 
\begin{equation*}
 \frac{d}{dt} \big( F_w^2(t) + 1\big) \leq C    (\mathcal{G}_*+\mathcal{G})(t) \big( F_w^2(t) + 1\big) ,
\end{equation*}
for some constant  $C>0$, it   follows, by applying Gr\"{o}nwall's lemma and using the fact that $ F_w(0)=0$, that 
 $$ F_w(t) \leq   \exp\left( C \int_0 ^t (\mathcal{G}_*+\mathcal{G})(\tau) d\tau\right). $$
 At last, we can easily check that the right-hand side above is finite, for any $t>0$, due to \eqref{w-dw}, Propositions \ref{Thm-summary} and \ref{prop:v:tv1}, and Remarks \ref{RMK-S1},  \ref{RMK-corollary1} and \ref{RMK:P*}, thereby concluding the proof of the proposition.
\end{proof}

\appendix
\begin{appendix}

\section{Littlewood--Paley decompositions and Besov spaces}
\label{besov:1}

\subsection{Littlewood--Paley decompositions}  Here, we recall some classical notions as well as the celebrated Littlewood--Paley decomposition of a tempered distribution. For more details, we refer the reader to \cite{bcd11}.

 The Fourier transform and its inverse are given, respectively, by
\begin{equation*}
	\mathcal{F}f\left(\xi\right)=\hat f(\xi)\bydef\int_{\mathbb{R}^d} e^{- i \xi \cdot x} f(x) dx
\end{equation*}
and 
\begin{equation*}
	\mathcal{F}^{-1} g\left(x\right)=\tilde g(x)\bydef\frac{1}{\left(2\pi\right)^d}\int_{\mathbb{R}^d} e^{i x \cdot \xi} g(\xi) d\xi,
\end{equation*}
for any $x,\xi\in \mathbb{R}^d$, in any dimension $d\geq 1$.

Now, consider smooth radial cutoff functions $\psi,\varphi\in C_c^\infty\left(\mathbb{R}^d\right)$ with the properties that 
\begin{equation*}
	\psi,\varphi\geq 0 ,
	\quad\supp\psi\subset\left\{|\xi|\leq 1\right\},
	\quad\supp\varphi\subset\left\{\frac{1}{2}\leq |\xi|\leq 2\right\}
\end{equation*}
and
\begin{equation*}
	1= \psi(\xi)+\sum_{k=0}^\infty \varphi\left(2^{-k}\xi\right),
	\quad\text{for all }\xi\in\mathbb{R}^d.
\end{equation*}
Accordingly, we define  the scaled cutoffs
\begin{equation*}
	\psi_{k}(\xi)\bydef\psi\left(2^{-k}\xi\right),
	\quad
	\displaystyle\varphi_{k}(\xi)\bydef\varphi\left(2^{-k}\xi\right),
\end{equation*}
for each $k\in\mathbb{Z}$, so that
\begin{equation*}
	\supp\psi_{k}\subset\left\{ |\xi|\leq 2^k\right\},
	\quad \supp\varphi_{k}\subset\left\{ 2^{k-1}\leq |\xi|\leq 2^{k+1}\right\}.
\end{equation*}
In particular, we have the partitions of unity
\begin{equation*}
	1\equiv \psi+\sum_{k=0}^\infty \varphi_{k}
	\quad\text{and}\quad
	1\equiv \psi_j+\sum_{k=j}^\infty \varphi_{k},
\end{equation*}
for any $j\in\mathbb{Z}$, and
\begin{equation*}
	1\equiv \sum_{k=-\infty}^\infty \varphi_{k},
\end{equation*}
away from the origin $\xi=0$.
Next,   we introduce the Fourier multiplier operators
\begin{equation*}
	S_k,\Delta_k:
	\mathcal{S}'(\mathbb{R}^d)\rightarrow\mathcal{S}'(\mathbb{R}^d),
\end{equation*}
for all $k\in\mathbb{Z}$, defined on the space of tempered distributions by $\mathcal{S}'$, by setting  
\begin{equation}\label{dyadic:def}
	S_k f\bydef\mathcal{F}^{-1}\psi_k\mathcal{F}f
	= \left(\mathcal{F}^{-1}\psi_k\right)*f
	\quad\text{and}\quad
	\Delta_k f\bydef\mathcal{F}^{-1}\varphi_k\mathcal{F}f
	= \left(\mathcal{F}^{-1}\varphi_k\right)*f.
\end{equation}
The Littlewood--Paley decomposition of  a tempered distribution $f\in\mathcal{S}'$ is then given by
\begin{equation*}
	f= S_0 f+\sum_{k=0}^\infty\Delta_{ k}f,
\end{equation*}
where the preceding series converges in $\mathcal{S}'$.

Likewise, one can show that the homogeneous Littlewood--Paley decomposition
\begin{equation*}
	\sum_{k=-\infty}^\infty\Delta_{ k}f=f
\end{equation*}
holds in $\mathcal{S}'$, as soon as $f\in\mathcal{S}'$ satisfies that
\begin{equation}\label{origin:1}
	\lim_{k\to -\infty}\|S_kf\|_{L^\infty}=0.
\end{equation}
At last, observe that \eqref{origin:1} holds if $\hat f$ is locally integrable around the origin, or whenever $S_0f$ belongs to $ L^p(\mathbb{R}^d)$, for some $1\leq p<\infty$. In particular, note that \eqref{origin:1} excludes all nonzero polynomials.

\subsection{Besov and Sobolev spaces}

For any $s \in \mathbb{R}$ and $1\leq p,q\leq \infty$, we define now the homogeneous Besov space $\dot B^{s}_{p,q}\left(\mathbb{R}^d\right)$ as the subspace of tempered distributions satisfying \eqref{origin:1} endowed with the semi-norm
\begin{equation*}
	\left\|f\right\|_{\dot B^{s}_{p,q}\left(\mathbb{R}^d\right)}=
	\left(
	\sum_{k\in\mathbb{Z}} 2^{ksq}
	\left\|\Delta_{k}f\right\|_{L^p\left(\mathbb{R}^d\right)}^q\right)^\frac{1}{q},
\end{equation*}
if $q<\infty$, and
\begin{equation*}
	\left\|f\right\|_{\dot B^{s}_{p,q}\left(\mathbb{R}^d\right)}=
	\sup_{k\in\mathbb{Z}}\left(2^{ks}
	\left\|\Delta_{k}f\right\|_{L^p\left(\mathbb{R}^d\right)}\right),
\end{equation*}
if $q=\infty$. One can show that $\dot B^s_{p,q}$ is a Banach space if $s<\frac dp$, or if $s=\frac dp$ and $q=1$ (see \cite[Theorem 2.25]{bcd11}).

The homogeneous Sobolev space $\dot H^s\left(\mathbb{R}^d\right)$, for any real value $s\in\mathbb{R}$,   is defined as the subspace of tempered distributions whose Fourier transform is locally integrable, with the semi-norm
\begin{equation*}
	\left\|f\right\|_{\dot H^s}=\left(\int_{\mathbb{R}^d}|\xi|^{2s}|\hat f(\xi)|^2 d\xi\right)^\frac 12.
\end{equation*}
It is easy to check that  $\dot H^s$ is a Hilbert space if and only if $s<\frac d2$ (see \cite[Proposition 1.34]{bcd11}), and that it it coincides with  $ \dot B^s_{2,2}$, as soon as $s<\frac d2$. 

The Sobolev spaces for general integrability index,  denoted here by $\dot{W}^{n,p}(\mathbb{R}^d)$, for $n\in \mathbb{N}$ and $p\in [1,\infty]$,   are defined through the semi-norms 
$$ \norm {f}_{\dot{W}^{n,p}} \bydef \sum_{\underset{\alpha\in \mathbb{N}^d}{|\alpha| =n}} \norm { \textbf{D} ^{\alpha} f}_{L^p},$$
where, for any multi-index $\alpha=(\alpha_1,\dots,\alpha_d)\in \mathbb{N}^d$, we use the notation
$$|\alpha| = \sum_{i=1}^d \alpha_i, \qquad  \textbf{D}^{\alpha} \bydef \partial_{x_1}^{\alpha_1} \cdots  \partial_{x_d}^{\alpha_d}.  $$
Note that the previous definition extends to fractional  regularities, $s\in (0,1)$, by setting 
$$ \norm {f}_{\dot{W}^{s,p}} \bydef \left( \int_{\mathbb{R}^d \times \mathbb{R}^d} \frac{|f(x)- f(y)|^p}{|x-y|^{d+sp}} dxdy \right)^{\frac{1}{p}}, \qquad p\in [1,\infty).$$ 

At last,  for any time $T>0$ and any choice of parameters $s \in \mathbb{R}$ and $1\leq p,q,r\leq \infty$, with $s<\frac dp$ (or $s=\frac dp$ and $q=1$), the spaces
\begin{equation}\label{vector_valued}
	L^r\left( [0,T) ; \dot B^{s}_{p,q}(\mathbb{R}^d) \right), \qquad L^r\left( [0,T) ; \dot W^{s,q}(\mathbb{R}^d) \right)
\end{equation}
are naturally defined as $L^r$-spaces with values in the Banach spaces $\dot B^{s}_{p,q}$ and $\dot W^{s,q}$, respectively.

\subsection{Embeddings}

We present now a few embeddings and inequalities in Sobolev and  Besov spaces which are routinely used throughout this work.

First,  letting $1<r_2 \leq r_1<\infty$ and $s_1\leq s_2$ in such a way that $s_1-\frac{d}{r_1}=s_2-\frac{d}{r_2}$,  it holds that
 \begin{equation}\label{Sobolev_Embedding_Lemma}
\|f\|_{\dot{W}^{s_1,r_1}(\R^d)}\lesssim\|f\|_{\dot{W}^{s_2,r_2}(\R^d)}.
\end{equation}

Moreover, a direct application of Young's convolution inequality to \eqref{dyadic:def} yields that
\begin{equation}\label{embedding:1}
	\norm{\Delta_kf}_{L^r(\mathbb{R}^d)}
	\lesssim 2^{kd\left(\frac 1p-\frac 1r\right)}\norm{\Delta_kf}_{L^p(\mathbb{R}^d)},
\end{equation}
for any $1\leq p\leq r\leq\infty$. A suitable use of \eqref{embedding:1} then leads to the embedding
\begin{equation}\label{embedding:2}
	\norm{f}_{\dot B^s_{r,q}(\mathbb{R}^d)}\lesssim \norm{f}_{\dot B^{s+d\left(\frac 1p-\frac 1r\right)}_{p,q}(\mathbb{R}^d)},
\end{equation}
for any $1\leq p\leq r\leq \infty$, $1\leq q\leq\infty$ and $s\in\mathbb{R}$, which can be interpreted as a Sobolev embedding in the framework of Besov spaces.

Furthermore, recalling that $\ell^q\subset\ell^r$, for all $1\leq q\leq r\leq\infty$, one has that
\begin{equation*}
	\dot B^s_{p,q}(\mathbb{R}^d)\subset \dot B^s_{p,r}(\mathbb{R}^d),
\end{equation*}
for all $s\in\mathbb{R}$, $1\leq p\leq\infty$ and $1\leq q\leq r\leq\infty$.

Next, observe that
\begin{equation}\label{embedding:3}
	\norm{f}_{L^p(\mathbb{R}^d)}=
	\Big\|\sum_{k\in\mathbb{Z}}\Delta_k f
	\Big\|_{L^p(\mathbb{R}^d)}
	\leq \sum_{k\in\mathbb{Z}}\norm{\Delta_k f}_{L^p(\mathbb{R}^d)}
	=\norm{\Delta_k f}_{\dot B^0_{p,1}(\mathbb{R}^d)},
\end{equation}
for every $1\leq p\leq \infty$. Therefore, by combining \eqref{embedding:2} and \eqref{embedding:3}, we obtain that
\begin{equation}\label{sobolev-infinity}
	\norm{f}_{L^\infty(\mathbb{R}^d)}\lesssim
	\norm{f}_{\dot B^{\frac dp}_{p,1}(\mathbb{R}^d)}, \quad \text{for all} \quad p\in [1,\infty].
\end{equation}

We recall now another essential inequality in Besov spaces which is related to their interpolation properties. Specifically, one has the interpolation, or convexity, inequality
\begin{equation}\label{interpolation.AP-B}
	\norm f_{\dot B^{s}_{p,1}}
	\lesssim
	\norm f_{\dot B^{s_0}_{p,\infty}}^{1-\theta}
	\norm f_{\dot B^{s_1}_{p,\infty}}^{\theta},
\end{equation}
for any $p\in[1,\infty]$, $s,s_0,s_1\in\mathbb{R}$ and $\theta\in(0,1)$ such that $s=(1-\theta)s_0+\theta s_1$ and $s_0\neq s_1$.

Note that the preceding estimates and embeddings can be straightforwardly adapted to the setting of the spaces introduced in \eqref{vector_valued}. 

The previous interpolation inequality provides a crucial gain in terms of the third index of the Besov space. An adaptation of its proof yields the following estimate, in two dimensions of space, which is used throughout our paper
\begin{equation}\label{lemma:interpolation}
\norm f_{\dot{B}^{\frac{2}{m}} _{m,1}} \lesssim \norm {\nabla f}_{L^p}^{ \frac{\frac{1}{2}- \frac{1}{m}}{\frac{1}{p}- \frac{1}{m}}} \norm {\nabla f}_{L^m}^{   \frac{\frac{1}{p}-\frac{1}{2}} {\frac{1}{p}- \frac{1}{m}}},
\end{equation} 
for any    $f\in \dot{W}^{1,p}(\mathbb{R}^2)\cap \dot{W}^{1,m}(\mathbb{R}^2)$, with   $p<2<m$. For the sake of clarity, we provide hereafter a short justification of this inequality.  

To that end, we write, by the definition of Besov spaces and employing \eqref{embedding:1},  for some $N 
>0$ 
to be chosen afterwards, that 
\begin{equation*}
\begin{aligned}
\norm f_{\dot{B}^{\frac{2}{m}} _{m,1} (\mathbb{R}^2)}  &= \Big( \sum_{j\leq N} + \sum_{j>N}\Big) 2^{\frac{2j}{m}} \norm {\Delta _jf}_{L^m (\mathbb{R}^2)}\\
& \lesssim  \sum_{j \leq N}    2^{ j\left(\frac{2}{p}-1 \right)   } \norm {\Delta _j\nabla  f}_{L^p (\mathbb{R}^2)}+ \sum_{j>N}   2^{ j \left( \frac{2}{m}-1 \right) } \norm {\Delta _j \nabla f}_{L^m (\mathbb{R}^2) }  .
\end{aligned}
\end{equation*}
Therefore, since $ p<2<m$, we arrive at the bound 
\begin{equation*}
\begin{aligned}
\norm f_{\dot{B}^{\frac{2}{m}} _{m,1} (\mathbb{R}^2)}     \lesssim     2^{N\left(\frac{2}{p}-1 \right)   } \norm { \nabla  f}_{L^p (\mathbb{R}^2) }+   2^{ N \left( \frac{2}{m}-1 \right) } \norm {  \nabla f}_{L^m (\mathbb{R}^2) }  .
\end{aligned}
\end{equation*}
At last, the inequality \eqref{lemma:interpolation} follows from an optimization argument on $N$.

We conclude this section by the following lemma which   summarizes some of the foregoing results and  emphasizes some crucial embeddings in  Besov and Sobolev spaces.  We refer to \cite[Proposition 2.39, and Theorems 1.83, 2.40 and 2.41]{bcd11} for 
detailed 
proofs.
\begin{lemma}\label{Emb_Besov_Tr_Lemma}  
Let $d \geq 1$. Then, it holds  
\begin{enumerate}
\item for any $1< p<\infty$ and $s\in \mathbb{R}$, that 
  $$\dot{B}^s_{p,\min(p,2)}  (\mathbb{R}^d)\hookrightarrow \dot{W}^{s,p}  (\mathbb{R}^d)\hookrightarrow \dot{B}^s_{p,\max(p,2)} (\mathbb{R}^d),$$
\item for any $1\leq p< \infty$ and $s\in \mathbb{R}$, that  
 $$ 
\dot{B}^s_{p,1} (\mathbb{R}^d)\hookrightarrow \dot{W}^{s,p} (\mathbb{R}^d)\hookrightarrow \dot{B}^s_{p,\infty}  (\mathbb{R}^d),$$
\item  for  any $1< p<\infty$ and  $0\leq  s< \frac{d}{p}$, that 
$$ \dot{W}^{s,p}(\mathbb{R}^d) \hookrightarrow L^{\frac{dp}{d-sp}} (\mathbb{R}^d).$$
\end{enumerate}
\end{lemma}

\subsection{Paradifferential product estimates}\label{Paradifferential product estimates}

Here, we recall the basic principles of paraproduct decompositions and some essential paradifferential product estimates that follow from it.

For any two suitable tempered distributions $f$ and $g$,  we introduce the paraproduct
\begin{equation*}
	T_fg\bydef \sum_{j\in\mathbb{Z}}S_{j-2}f\Delta_j g,
\end{equation*}
which then allows us to write that 
\begin{equation*}
	fg=T_fg+T_gf+R(f,g),
\end{equation*}
where the remainder  $R(f,g)$ is given by 
\begin{equation*}
	R(f,g) \bydef \sum_{\substack{j,k\in\mathbb{Z}\\|j-k|\leq 2}}\Delta_jf\Delta_kg.
\end{equation*}
A detailed  analysis of various estimates for the preceding operators can be found in  \cite[Section 2.6]{bcd11}. Below, we recall a time dependent version of some useful estimates from  \cite{bcd11}, for any choice of integrability parameters in $[1,\infty]$ such that
\begin{equation*}
	\frac 1a=\frac 1{a_1}+\frac 1{a_2},
	\qquad
	\frac 1b=\frac 1{b_1}+\frac 1{b_2},
	\qquad
	\frac 1c=\frac 1{c_1}+\frac 1{c_2}.
\end{equation*} 
Specifically, it can be shown that
\begin{equation*}
	\norm{T_fg}_{  L^a_t\dot B^{\alpha+\beta}_{b,c}}
	\lesssim
	\norm{f}_{  L^{a_1}_t\dot B^\alpha_{b_1,c_1}}
	\norm{g}_{  L^{a_2}_t\dot B^\beta_{b_2,c_2}},
\end{equation*}
for any $\alpha<0$ and $\beta\in\mathbb{R}$, and that
\begin{equation*}
	\norm{R(f,g)}_{  L^a_t\dot B^{\alpha+\beta}_{b,c}}
	\lesssim
	\norm{f}_{  L^{a_1}_t\dot B^\alpha_{b_1,c_1}}
	\norm{g}_{  L^{a_2}_t\dot B^\beta_{b_2,c_2}},
\end{equation*}
for any $\alpha,\beta\in\mathbb{R}$ with $\alpha+\beta>0$. Moreover, one has the endpoint estimates
\begin{equation*}
	\norm{T_fg}_{  L^a_t\dot B^{\beta}_{b,c}}
	\lesssim
	\norm{f}_{L^{a_1}_tL^{b_1}_x}
	\norm{g}_{  L^{a_2}_t\dot B^\beta_{b_2,c}},
\end{equation*}
for all $\beta\in\mathbb{R}$,
and
\begin{equation*}
	\norm{R(f,g)}_{L^a_tL^b_x}
	\lesssim
	\norm{f}_{  L^{a_1}_t\dot B^\alpha_{b_1,c_1}}
	\norm{g}_{  L^{a_2}_t\dot B^{-\alpha}_{b_2,c_2}},
\end{equation*}
for all $\alpha\in\mathbb{R}$, provided $\frac 1{c_1}+\frac 1{c_2}=1$.  

We finally recall two important product laws whose proofs are direct consequences of the preceding bounds.  It holds that 
\begin{equation}\label{Prod_Law}
	\|fg\|_{\dot B^{s+t-\frac d2}_{2,1}} \lesssim \|f\|_{\dot H^s}\|g\|_{\dot H^t},
\end{equation}
for any $-\frac d2<s,t<\frac d2$ with $s+t>0$, and that 
\begin{equation}\label{Prod_Law2}
	\|fg\|_{\dot H^s} \lesssim \|f\|_{L^\infty\cap \dot B^\frac d2_{2,\infty}}\|g\|_{\dot H^s},
\end{equation}
for all $-\frac d2<s<\frac d2$.

We conclude this section with a few (interpolation-type) estimates based on frequency decompositions. The proof  of the next lemma is given   in \cite[Lemmas 7.3 and 7.4]{ag20}.  
 \begin{lemma}\label{Lemma_h_low_Estimate}
 In any dimension $d\geq 1$, it holds   
 \begin{enumerate}
\item  for any ${s}>\frac{d}{2}$ and $0\leq t_0<t$,  that  
 \begin{equation*}
\norm { ({\rm Id}- S_0)h}_{L^2 ([t_0,t];L^\infty)}\lesssim \norm {  h}_{L^2([t_0,t]; \dot{H}^{\frac{d}{2}}) } \log^{\frac{1}{2}}  \left( e + \frac{  \norm { h}_{L^2([t_0,t];\dot{B}_{2,\infty}^{{s}}) }}{ \norm {  h}_{L^2([t_0,t]; \dot{H}^{\frac{d}{2}}) }} \right).
\end{equation*}

\item for any  $0\leq t_0<t$, that 
\begin{equation*}
\norm {S_0h}_{L^2 ([t_0,t];L^\infty)}\lesssim \norm {  h}_{L^2([t_0,t]; \dot{H}^{\frac{d}{2}}) } \log^{\frac{1}{2}}  \left( e + \frac{  \norm { h}_{L^2([t_0,t];L^2) }}{ \norm {  h}_{L^2([t_0,t]; \dot{H}^{\frac{d}{2}}) }} \right). 
\end{equation*}
\end{enumerate}
 \end{lemma}
 
Finally, we also give in the next result a useful commutator estimate.

\begin{lemma}\label{lemma:commutator}
	Let $d\geq 1$ and $M$ be a smooth and homogeneous function of degree $m$ away from the origin, for some $m\in \mathbb{R}$. Further consider two functions 
	\begin{equation*}
		u\in \dot W^{1,\infty} \cap B^{\delta}_{r,\infty} (\mathbb{R}^d)
		 \quad \text{and}
		 \quad 
		  v \in B^{s+1}_{\infty,\infty} \cap B^{s}_{r,\infty}(\mathbb{R}^d),
	\end{equation*}
	for some $   s \in \mathbb{R}$,  $ \delta \in (0,1)$ and $ r \in [1,\infty]$.
	 Then, it holds that 
	\begin{equation*}
		\norm {[T_u, M(D)] v}_{B^{s-m+1 }_{r,1}}\lesssim  \norm {\nabla u}_{L^\infty} \norm { v}_{B^s_{r, \infty}}  \log \left(e+ \frac{    \norm { u}_{B^{\delta }_{r,\infty}} \norm { v}_{B^{s+ 1 }_{\infty , \infty}} }{ \norm {\nabla u}_{L^\infty} \norm { v}_{B^s_{r, \infty}}   } \right) .
	\end{equation*} 
\end{lemma}

\begin{proof}
Let $N \geq 1$ to be fixed in a while and write that  
\begin{equation*}
	\begin{aligned}
		\norm {[T_u, M(D)] v}_{B^{s-m+1}_{r,1}} 
		&\leq \norm { S_N  [T_u, M(D)] v }_{B^{s-m+1}_{r,1}} + \norm {(\id -S_N) [T_u, M(D)] v }_{B^{s-m+1}_{r,1}}
		\\
		& \lesssim N  \norm {[T_u, M(D)] v}_{B^{s-m+1}_{r,\infty}}  + 2^{- N\delta }  \norm {[T_u, M(D)] v}_{B^{s-m+1 + \delta}_{r,\infty}} .
	\end{aligned}
\end{equation*}
Thus, applying the two estimates from  \cite[Lemma 2.99]{bcd11} to control each of the preceding terms, respectively, one finds that 
\begin{equation*}
	\begin{aligned}
		\norm {[T_u, M(D)] v}_{B^{s-m+1}_{r,1}} 
		& \lesssim N \norm {\nabla u}_{L^\infty} \norm { v}_{B^s_{r, \infty}}  + 2^{- N\delta }  \norm {\nabla u}_{B^{\delta    -1}_{r,\infty}} \norm { v}_{B^{s+ 1 }_{\infty , \infty}}  .
	\end{aligned}
\end{equation*}
Therefore, optimizing the value of $N$ by setting
\begin{equation*}
	N\delta= \log_2  \left(e+\frac{    \norm {\nabla u}_{B^{\delta    -1}_{r,\infty}} \norm { v}_{B^{s+ 1 }_{\infty , \infty}} }{ \norm {\nabla u}_{L^\infty} \norm { v}_{B^s_{r, \infty}}   } \right),
\end{equation*}
 we end up with 
\begin{equation*}
		\norm {[T_u, M(D)] v}_{B^{s-m+1}_{r,1}} \lesssim  \norm {\nabla u}_{L^\infty} \norm { v}_{B^s_{r, \infty}}  \log \left(e+ \frac{    \norm {\nabla u}_{B^{\delta    -1}_{r,\infty}} \norm { v}_{B^{s+ 1 }_{\infty , \infty}} }{ \norm {\nabla u}_{L^\infty} \norm { v}_{B^s_{r, \infty}}   } \right)  ,
\end{equation*}
thereby concluding the proof of the lemma.
\end{proof}

\section{Lorentz spaces}\label{Section_Lorentz}

We introduce now a refined version of the standard Lebesgue spaces, the so-called Lorentz spaces, and we present several useful results that are utilized throughout the paper. For more details, we refer the reader to \cite[Chapter 1]{Grafakos}.

Let $(X, \mu)$ be a measure space  and $0<p,r\leq \infty$. We say that $f$ belongs to the Lorentz space $L^{p,r}(X)$ if and only if  
\begin{equation*}
\|f\|_{L^{p,r}}\bydef \left\{
\begin{array}{ll}
p^{\frac{1}{r}}\left(\int_0^\infty (s\left(\mu\{ |f|>s\}\right)^{\frac{1}{p}})^r\frac{ds}{s} \right)^{\frac{1}{r}} &\quad \text{if}\quad r<\infty,\\
\sup_{s>0} s \left(\mu\{ |f|>s\} \right)^{\frac{1}{p}}&\quad \text{if}\quad r=\infty,
\end{array}
\right.
\end{equation*}
is finite. Note that the space $L^{\infty,r}(X)$, for $r\neq \infty$, is reduced to the trivial space $\{0\}$. This case is therefore always off consideration. 

The following lemma summarizes several classical results on Lorentz spaces.
\begin{lemma}  \label{lemma:Lorentz} The following holds
\begin{enumerate} 
\item (Embedding) For any $p\in [1,\infty]$ and  $1\leq r_1\leq r_2 \leq \infty$ $$L^{p,r_1}\hookrightarrow L^{p,r_2} \quad \text{and} \quad L^{p,p}=L^p.$$  
\item  (H\"older's inequality) For $1\leq p,p_1,p_2\leq \infty$ and $1\leq r, r_1, r_2\leq \infty$ 
\begin{equation*}
\|fg\|_{L^{p,r}}\leq \|f\|_{L^{p_1,r_1}}\|g\|_{L^{p_2,r_2}}\quad \text{with}\quad \frac{1}{p}=\frac{1}{p_1}+\frac{1}{p_2}\quad \text{and}\quad \frac{1}{r}=\frac{1}{r_1}+\frac{1}{r_2}.
\end{equation*} 
\item (Power law) For any $\alpha>0$ and nonnegative measurable function $f$ 
\begin{equation*}
\|f^\alpha\|_{L^{p,r}}=\|f\|_{L^{\alpha p,\alpha r}}^\alpha. 
\end{equation*}

\end{enumerate}
 \end{lemma}

We end up this section by proving a crucial interpolation-type estimate in Lorentz spaces. The following lemma is in the spirit of Lemma \ref{Lemma_h_low_Estimate}, and generalizes \cite[Lemma 2.2]{BC13}
 \begin{lemma}\label{Lemma:interpolation X}
Let $ f$ be a nonnegative function defined on a measure space $(X,\mu)$, with $\mu(X)< \infty.$ 
 If $f\in L^{p,\infty}(X)$, for some $p\in (1,\infty]$, then $f\in L^{q,1}(X)$, for any $q\in [1,p)$. 
 
 Moreover, the following inequality holds for any $r\in [1,\infty)$
\begin{equation*}
\norm f_{L^{q,r}(X) } ^r \leq \frac {p}{p-q}\norm f_{L^{q,\infty}(X) } ^r  \left (  \frac 1r + \log \left( \frac{ \norm f_{L^{p,\infty}(X) }} {\norm f_{L^{q,\infty}(X) }}  \left( \mu(X)\right)^{ \frac 1q - \frac 1p}\right)\right).
\end{equation*}
In particular, for $p=\infty$ and $r=2$, we have, for any $q\in (1,\infty)$, that
\begin{equation}\label{Interpolation_Ineq}
\norm f_{L^{q,2}(X) }   \leq  \norm f_{L^{q,\infty}(X) }    \left ( 1  + \log ^{\frac 12} \left( \frac{ \norm f_{L^{\infty}(X) }} {\norm f_{L^{q,\infty}(X) }}  \left( \mu(X)\right)^{ \frac 1q  }\right)\right).
\end{equation}
\end{lemma}
\begin{proof} 
It is easy to check, as long as $ 1\leq q\leq p \leq \infty $, that
$$  \norm f_{L^{q,\infty}(X) } \leq  \left( \mu(X)\right)^{ \frac 1q - \frac 1p}    \norm f_{L^{p,\infty}(X) } .$$ 
Therefore, let us split the $L^{q,1}(X) $ norm of $f$ into 
\begin{equation*}
\begin{aligned}
  \norm f_{L^{q,r}(X) }^r  = \left(  \int_{ 0}^ \alpha + \int_{ \alpha}^ \beta + \int_{ \beta}^ \infty \right) \lambda ^r \Big(   \mu \left \{  x\in X : f(x)> \lambda \right\}\Big)^{\frac rq} \frac {d\lambda}{\lambda} \bydef \sum_{i=1}^3 \mathcal{I}_i.
\end{aligned}
\end{equation*} 
We will be using H\"older's inequality to deal with each part of the above decomposition.

For the first term, we have that 
$$\mathcal{I}_1 \leq \frac 1r  \left( \mu(X)\right)^{ \frac rq  } \alpha^r.$$
For the second one, we find that    
$$ \mathcal{I}_2 \leq  \left( \sup_{ \lambda >0} \left( \lambda  \left(  \mu \left \{  x\in X : f(x)> \lambda \right\}\right)^{\frac 1q} \right) \right)^r  \int_\alpha ^ \beta\frac {d\lambda}{\lambda}  = \norm f_{L^{q,\infty}(X)} ^r\log\left( \frac {\beta}{\alpha}\right).$$
As for the last term, we obtain, as long as $p\in (q,\infty)$, that  
\begin{equation*}
\begin{aligned} 
 \mathcal{I}_3 & \leq   \sup_{ \lambda >0} \left( \lambda  \left(  \mu \left \{  x\in X : f(x)> \lambda \right\}\right)^{\frac 1p} \right) ^{  \frac {rp}q}  \int_{\beta}^\infty \lambda^{ r(1- \frac pq) } \frac {d\lambda}{\lambda} \\
 &= \frac{ q}{r(p-q)} \norm f_{L^{p,\infty}(X)} ^{\frac {rp}q}  \beta^{r( 1- \frac pq)}.
 \end{aligned}
\end{equation*} 
Then, choosing 
 $$ \alpha = \frac{\norm f_{ L^{q,\infty}} }{ \left( \mu(X)\right)^{\frac 1q} } , \qquad \beta = \left( \frac {\norm f_{L^{p,\infty}}^p}{\norm f_{L^{q,\infty}}^q}\right) ^{\frac{1}{p-q}},
 $$
 yields 
 $$\norm f_{L^{p,\infty}(X)} ^{\frac pq}  \beta^{ 1- \frac pq} = \norm f_{L^{q,\infty}(X)}  $$
and
 $$ \frac{\beta}{\alpha} = \left( \mu(X)\right)^{\frac 1q} \left( \frac {\norm f_{L^{p,\infty}}}{\norm f_{L^{q,\infty}}}\right) ^{\frac{p}{p-q}} =   \left( \left( \mu(X)\right)^{ \frac 1q - \frac 1p} \frac {\norm f_{L^{p,\infty}}}{\norm f_{L^{q,\infty}}}\right) ^{\frac{p}{p-q}}.$$
Thus, substituting the preceding values in the bounds on  $\mathcal{I}_i$ ($i\in \{1,2,3\}$) leads to  the desired estimate in the case $p<\infty$. 

 As for case $ p=\infty$, the arguments we used to estimate $\mathcal{I}_3$ are no longer valid. Instead, we  proceed by choosing $\beta$ as
 $$ \beta = \norm f_{L^\infty(X)}.$$
 This implies that $\mathcal{I}_3=0$, whence, the result follows by combining the estimates of $\mathcal{I}_1$ and $\mathcal{I}_2$.
 \end{proof}

\section{Maximal parabolic regularity estimates}

In this section,  for the convenience of the  reader, we recall the useful principles of maximal parabolic regularity estimates in Besov   and  Lorentz spaces for the Stokes equation
\begin{equation}\label{Stokes_System}
\left\{
\begin{array}{ll} 
 \partial_t u - \mu\Delta u  + \nabla p  =      f   , \\
\div u=0,\\
u|_{t=0}= u_0, 
\end{array}
\right.   
\end{equation}
where $t>0$ and $x\in \R^d$. 

\begin{lemma}[{\cite[Section 2B]{Danchin_al_2020}}]\label{MR:lemma:1}
Let $1<p,r<\infty$.  Then, for any divergence free  initial data  $u_0\in \dot{B}_{p,r}^{2-\frac 2r}(\R^d)$ and source term $f\in L^{r}(0,T;L^p(\R^d))$,  the solution $(u, \nabla p)$ of the  Stokes system \eqref{Stokes_System} enjoys the bounds
\begin{equation}\label{Maximl_Regularity_Stokes}
\Vert u\Vert_{L_t^\infty \dot{B}_{p,r}^{2-\frac 2r} }+\Vert (\partial_t u , \mu \nabla^2 u, \nabla p)\Vert_{L_t^r L^p }\lesssim \Vert u_0\Vert_{\dot{B}_{p,r}^{2-\frac 2r}}+ \Vert f \Vert _{L_t^r L^p }. 
\end{equation}
\end{lemma}

 At last, we also recall an improvement  of the previous lemma.

 \begin{lemma}[{\cite[Proposition A.5]{Danchin_Wang_22}}]\label{M:reg:DW}
 Let $1<p,q<\infty$ and $1\leq r\leq \infty$. Then, for any divergence free  initial data  $u_0\in \dot{B}_{p,r}^{-1+\frac2p}(\R^d)$ and source term $f\in L^{q,r}(0,T;L^p(\R^d))$,  the solution $(u, \nabla p)$ of the  Stokes system \eqref{Stokes_System} enjoys the bounds
 \begin{equation*}
 \begin{aligned}
\Vert  u \Vert_{L^\infty _t  \dot{B}^{-1+ \frac 2p}_{p,r}} +  \Vert (\partial_t u, \nabla ^2 u,\nabla p)\Vert_{L^{q,r}_t  L^p}  &\lesssim \norm {u_0}_{ \dot{B}^{-1+ \frac 2p}_{p,r} }+ \Vert f \Vert_{L^{q,r}_t  L^p}. 
\end{aligned} 
\end{equation*}
  In addition, if $\frac2q+\frac dp>2$, then, for all $s\in (q,\infty)$ and $m\in (p,\infty)$ such that 
  \begin{equation*}
\frac{d}{2m}+\frac{1}{s}=\frac{d}{2p}+\frac{1}{q}-1,
\end{equation*}
it holds that 
\begin{equation}\label{v_L_r_s_Estimate}
\norm {u}_{L^{s,r}_t L^m}\lesssim \Vert  u \Vert_{L^\infty _t  \dot{B}^{-1+ \frac 2p}_{p,r}} +  \Vert (\partial_t u, \nabla ^2 u)\Vert_{L^{q,r}_t  L^p}. 
\end{equation}

 \end{lemma}

 \end{appendix}  
\bibliographystyle{plain}
\bibliography{Navier_Stokes_Maxwell}

\begin{thebibliography}{10}

\bibitem{Abidi:2021ud}
Hammadi Abidi and Guilong Gui.
\newblock Global well-posedness for the 2-{D} inhomogeneous incompressible
  {N}avier-{S}tokes system with large initial data in critical spaces.
\newblock {\em Arch. Ration. Mech. Anal.}, 242(3):1533--1570, 2021.

\bibitem{AGZ12}
Hammadi Abidi, Guilong Gui, and Ping Zhang.
\newblock On the wellposedness of three-dimensional inhomogeneous
  {N}avier-{S}tokes equations in the critical spaces.
\newblock {\em Arch. Ration. Mech. Anal.}, 204(1):189--230, 2012.

\bibitem{AMBP_2007__14_1_103_0}
Hammadi Abidi and Taoufik Hmidi.
\newblock R\'esultats d{\textquoteright}existence dans des espaces critiques
  pour le syst\`eme de la {MHD} inhomog\`ene.
\newblock {\em Annales math\'ematiques Blaise Pascal}, 14(1):103--148, 2007.

\bibitem{AP07}
Hammadi Abidi and Marius Paicu.
\newblock Existence globale pour un fluide inhomog\`ene.
\newblock {\em Ann. Inst. Fourier (Grenoble)}, 57(3):883--917, 2007.

\bibitem{AP08}
Hammadi Abidi and Marius Paicu.
\newblock Global existence for the magnetohydrodynamic system in critical
  spaces.
\newblock {\em Proc. Roy. Soc. Edinburgh Sect. A}, 138(3):447--476, 2008.

\bibitem{a19}
Diogo Ars\'enio.
\newblock Recent progress on the mathematical theory of plasmas.
\newblock {\em Bol. Soc. Port. Mat.}, 77:27--38, 2019.

\bibitem{ag20}
Diogo Ars\'{e}nio and Isabelle Gallagher.
\newblock Solutions of {N}avier-{S}tokes-{M}axwell systems in large energy
  spaces.
\newblock {\em Trans. Amer. Math. Soc.}, 373(6):3853--3884, 2020.

\bibitem{ah}
Diogo Ars\'enio and Haroune Houamed.
\newblock Damped {S}trichartz estimates and the incompressible
  {E}uler--{M}axwell system.
\newblock {\em arXiv}, 2022.
\newblock \url{https://arxiv.org/abs/2204.04277}.

\bibitem{aim15}
Diogo Ars\'{e}nio, Slim Ibrahim, and Nader Masmoudi.
\newblock A derivation of the magnetohydrodynamic system from
  {N}avier-{S}tokes-{M}axwell systems.
\newblock {\em Arch. Ration. Mech. Anal.}, 216(3):767--812, 2015.

\bibitem{as}
Diogo Ars\'{e}nio and Laure Saint-Raymond.
\newblock {\em From the {V}lasov-{M}axwell-{B}oltzmann system to incompressible
  viscous electro-magneto-hydrodynamics. {V}ol. 1}.
\newblock EMS Monographs in Mathematics. European Mathematical Society (EMS),
  Z\"{u}rich, 2019.

\bibitem{ah2}
Diogo Arsénio and Haroune Houamed.
\newblock {Stability Analysis of Two-Dimensional Ideal Flows With Applications
  to Viscous Fluids and Plasmas}.
\newblock {\em International Mathematics Research Notices}, 2024.
\newblock In press.

\bibitem{AJ63}
Jean-Pierre Aubin.
\newblock Un th\'{e}or\`eme de compacit\'{e}.
\newblock {\em C. R. Acad. Sci. Paris}, 256:5042--5044, 1963.

\bibitem{bcd11}
Hajer Bahouri, Jean-Yves Chemin, and Rapha\"{e}l Danchin.
\newblock {\em Fourier analysis and nonlinear partial differential equations},
  volume 343 of {\em Grundlehren der Mathematischen Wissenschaften [Fundamental
  Principles of Mathematical Sciences]}.
\newblock Springer, Heidelberg, 2011.

\bibitem{BIE201985}
Qunyi Bie, Qiru Wang, and Zheng an~Yao.
\newblock Global well-posedness of the 3{}d incompressible {MHD} equations with
  variable density.
\newblock {\em Nonlinear Analysis: Real World Applications}, 47:85--105, 2019.

\bibitem{Biskamp_1993}
Dieter Biskamp.
\newblock {\em Nonlinear magnetohydrodynamics}, volume~1 of {\em Cambridge
  Monographs on Plasma Physics}.
\newblock Cambridge University Press, Cambridge, 1993.

\bibitem{BC13}
Fran\c{c}ois Bouchut and Gianluca Crippa.
\newblock Lagrangian flows for vector fields with gradient given by a singular
  integral.
\newblock {\em J. Hyperbolic Differ. Equ.}, 10(2):235--282, 2013.

\bibitem{CL18}
Yuan Cai and Zhen Lei.
\newblock Global well-posedness of the incompressible magnetohydrodynamics.
\newblock {\em Arch. Ration. Mech. Anal.}, 228(3):969--993, 2018.

\bibitem{CTZ11}
Qing Chen, Zhong Tan, and Yanjin Wang.
\newblock Strong solutions to the incompressible magnetohydrodynamic equations.
\newblock {\em Math. Methods Appl. Sci.}, 34(1):94--107, 2011.

\bibitem{D23}
Rapha{\"e}l Danchin.
\newblock Global well-posedness for 2d inhomogeneous viscous flows with rough
  data via dynamic interpolation.
\newblock {\em hal-04227173}, 2023.
\newblock preprint.

\bibitem{RM21}
Rapha\"{e}l Danchin and Bernard Ducomet.
\newblock On the global existence for the compressible {E}uler-{P}oisson
  system, and the instability of static solutions.
\newblock {\em J. Evol. Equ.}, 21(3):3035--3054, 2021.

\bibitem{Danchin_al_2020}
Rapha\"{e}l Danchin, Francesco Fanelli, and Marius Paicu.
\newblock A well-posedness result for viscous compressible fluids with only
  bounded density.
\newblock {\em Anal. PDE}, 13(1):275--316, 2020.

\bibitem{DM12}
Rapha\"{e}l Danchin and Piotr Bogus\l~aw Mucha.
\newblock A {L}agrangian approach for the incompressible {N}avier-{S}tokes
  equations with variable density.
\newblock {\em Comm. Pure Appl. Math.}, 65(10):1458--1480, 2012.

\bibitem{DM19}
Rapha\"{e}l Danchin and Piotr Bogus\l~aw Mucha.
\newblock The incompressible {N}avier-{S}tokes equations in vacuum.
\newblock {\em Comm. Pure Appl. Math.}, 72(7):1351--1385, 2019.

\bibitem{Danchin_Wang_22}
Rapha{\"e}l Danchin and Shan Wang.
\newblock Global unique solutions for the inhomogeneous navier-stokes equations
  with only bounded density, in critical regularity spaces.
\newblock {\em Communications in Mathematical Physics}, 399(3):1647--1688,
  2023.

\bibitem{D-book}
Peter~Alan Davidson.
\newblock {\em An introduction to magnetohydrodynamics}.
\newblock Cambridge Texts in Applied Mathematics. Cambridge University Press,
  Cambridge, 2001.

\bibitem{DZ18}
Wen Deng and Ping Zhang.
\newblock Large time behavior of solutions to 3-{D} {MHD} system with initial
  data near equilibrium.
\newblock {\em Arch. Ration. Mech. Anal.}, 230(3):1017--1102, 2018.

\bibitem{Gerbeau_1997}
J.~Gerbeau and C~Le~BrisLe~Bris.
\newblock Existence of solution for a density-dependent magneto hydrodynamic
  equation.
\newblock {\em Advances in Differential Equations, 2(3),}, 2(3):427--452, 1997.

\bibitem{Ger_Mess_Ibra_2014}
Pierre Germain, Slim Ibrahim, and Nader Masmoudi.
\newblock Well-posedness of the {N}avier-{S}tokes-{M}axwell equations.
\newblock {\em Proc. Roy. Soc. Edinburgh Sect. A}, 144(1):71--86, 2014.

\bibitem{Grafakos}
Loukas Grafakos.
\newblock {\em Classical {F}ourier analysis}, volume 249 of {\em Graduate Texts
  in Mathematics}.
\newblock Springer, New York, third edition, 2014.

\bibitem{GG14}
Guilong Gui.
\newblock Global well-posedness of the two-dimensional incompressible
  magnetohydrodynamics system with variable density and electrical
  conductivity.
\newblock {\em J. Funct. Anal.}, 267(5):1488--1539, 2014.

\bibitem{HB12}
Boris Haspot.
\newblock Well-posedness for density-dependent incompressible fluids with
  non-{L}ipschitz velocity.
\newblock {\em Ann. Inst. Fourier (Grenoble)}, 62(5):1717--1763, 2012.

\bibitem{hhz2023}
Taoufik Hmidi, Haroune Houamed, and Mohamed Zerguine.
\newblock Rigidity aspects of singular patches in stratified flows.
\newblock {\em Tunis. J. Math.}, 4(3):465--557, 2022.

\bibitem{HW13}
Xiangdi Huang and Yun Wang.
\newblock Global strong solution to the 2{D} nonhomogeneous incompressible
  {MHD} system.
\newblock {\em J. Differential Equations}, 254(2):511--527, 2013.

\bibitem{HUANG2013511}
Xiangdi Huang and Yun Wang.
\newblock Global strong solution to the 2{D} nonhomogeneous incompressible
  {MHD} system.
\newblock {\em Journal of Differential Equations}, 254(2):511--527, 2013.

\bibitem{Ibrahim_2011}
Slim Ibrahim and Sahbi Keraani.
\newblock Global small solutions for the {N}avier-{S}tokes-{M}axwell system.
\newblock {\em SIAM J. Math. Anal.}, 43(5):2275--2295, 2011.

\bibitem{LI20176512}
Jinkai Li.
\newblock Local existence and uniqueness of strong solutions to the
  navier--stokes equations with nonnegative density.
\newblock {\em Journal of Differential Equations}, 263(10):6512--6536, 2017.

\bibitem{LTY17}
Jinlu Li, Wenke Tan, and Zhaoyang Yin.
\newblock Local existence and uniqueness for the non-resistive {MHD}, equations
  in homogeneous {B}esov spaces.
\newblock {\em Adv. Math.}, 317:786--798, 2017.

\bibitem{LXZ15}
Fanghua Lin, Li~Xu, and Ping Zhang.
\newblock Global small solutions of 2-{D} incompressible {MHD} system.
\newblock {\em J. Differential Equations}, 259(10):5440--5485, 2015.

\bibitem{L61}
J.-L. Lions.
\newblock {\em \'{E}quations diff\'{e}rentielles op\'{e}rationnelles et
  probl\`emes aux limites.}
\newblock Springer-Verlag, Berlin-G\"{o}ttingen-Heidelberg,, 1961.

\bibitem{Lions_1996_Book_1}
Pierre-Louis Lions.
\newblock Mathematical topics in fluid mechanics. {V}ol. 1, 1996.
\newblock Incompressible models, Oxford Science Publications.

\bibitem{Masmoudi_2010}
Nader Masmoudi.
\newblock Global well posedness for the {M}axwell--{N}avier--{S}tokes system in
  2{D}.
\newblock {\em J. Math. Pures et Appliqu\'ees}, 93:559--571, 2010.

\bibitem{Pai_1962}
Shih-I Pai.
\newblock {\em Magnetogasdynamics and plasma dynamics}.
\newblock Springer-Verlag, Vienna; Prentice Hall, Inc., Englewood Cliffs, N.J.,
  1962.

\bibitem{PZZ13}
Marius Paicu, Ping Zhang, and Zhifei Zhang.
\newblock Global unique solvability of inhomogeneous {N}avier-{S}tokes
  equations with bounded density.
\newblock {\em Comm. Partial Differential Equations}, 38(7):1208--1234, 2013.

\bibitem{PZ13}
Marius Paicu, Ping Zhang, and Zhifei Zhang.
\newblock Global unique solvability of inhomogeneous {N}avier-{S}tokes
  equations with bounded density.
\newblock {\em Comm. Partial Differential Equations}, 38(7):1208--1234, 2013.

\bibitem{MT83}
Michel Sermange and Roger Temam.
\newblock Some mathematical questions related to the {MHD} equations.
\newblock {\em Comm. Pure Appl. Math.}, 36(5):635--664, 1983.

\bibitem{JS87}
Jacques Simon.
\newblock Compact sets in the space {$L^p(0,T;B)$}.
\newblock {\em Ann. Mat. Pura Appl. (4)}, 146:65--96, 1987.

\bibitem{WZ17}
Dongyi Wei and Zhifei Zhang.
\newblock Global well-posedness of the {MHD} equations in a homogeneous
  magnetic field.
\newblock {\em Anal. PDE}, 10(6):1361--1406, 2017.

\end{thebibliography}

\end{document}